\documentclass[a4paper, 11pt]{amsart}

\usepackage{amsmath,amssymb,amscd,xypic}
\usepackage{amsfonts}
\usepackage[english]{babel}
\usepackage[leqno]{amsmath}
\usepackage{amssymb,amsthm}
\usepackage{mathrsfs} 
\usepackage{amscd}
\usepackage{enumerate}
\usepackage{epsfig}
\usepackage{relsize}
\usepackage{layout}
\usepackage{fullpage}
\usepackage[usenames,dvipsnames]{xcolor}
\usepackage[backref=page]{hyperref}
\hypersetup{
 colorlinks,
 citecolor=Green,
 linkcolor=Red,
 urlcolor=Blue}
\usepackage[matrix,arrow,tips,curve]{xy}
\input{xy}
\xyoption{all}

\UseRawInputEncoding

\newtheorem{theorem}{Theorem}[section]
\newtheorem{lemma}[theorem]{Lemma}

\newtheorem{proposition}[theorem]{Proposition}
\newtheorem{cor}[theorem]{Corollary}

\newtheorem{fact}[theorem]{Fact}
\newtheorem{theoremalpha}{Theorem}

\numberwithin{equation}{section}
\theoremstyle{definition}
\newtheorem{definition}[theorem]{Definition}

\theoremstyle{remark}
\newtheorem{remark}[theorem]{Remark}
\newtheorem{example}[theorem]{Example}

\newenvironment{sis}{\left\{\begin{aligned}}{\end{aligned}\right.}

\newcommand{\bbR}{{\mathbb R}}

\newcommand{\bbC}{{\mathbb C}}
\newcommand{\bbQ}{{\mathbb Q}}
\newcommand{\bbZ}{{\mathbb Z}}

\newcommand{\bbG}{{\mathbb G}}
\newcommand{\bbA}{{\mathbb A}}
\newcommand{\bbE}{{\mathbb E}}

\newcommand{\cC}{{\mathcal C}}

\newcommand{\cF}{{\mathcal F}}
\newcommand{\cG}{{\mathcal G}}

\newcommand{\calL}{{\mathcal L}}

\newcommand{\cN}{{\mathcal N}}
\newcommand{\cO}{{\mathcal O}}

\newcommand{\cX}{{\mathcal X}}
\newcommand{\cY}{{\mathcal Y}}

\newcommand{\un}{\underline}
\newcommand{\ov}{\overline}
\newcommand{\wt}{\widetilde}
\newcommand{\wh}{\widehat}
\newcommand{\m}{\mathcal}

\newcommand{\U}{\mathcal{U}}
\renewcommand{\P}{\mathbb{P}}

\newcommand{\irr}{\operatorname{irr}}
\newcommand{\type}{\operatorname{type}}
\newcommand{\Spec}{\operatorname{Spec}}
\newcommand{\Spf}{\operatorname{Spf}}
\newcommand{\Proj}{\operatorname{Proj}}
\newcommand{\Aut}{\operatorname{Aut}}
\newcommand{\Pic}{\operatorname{Pic}}
\newcommand{\PGL}{\operatorname{PGL}}
\newcommand{\Cl}{\operatorname{Cl}}
\newcommand{\D}{\operatorname{\Delta}}

\newcommand{\Def}{\operatorname{Def}}
\newcommand{\ps}{\operatorname{ps}}
\newcommand{\adm}{\operatorname{adm}}
\newcommand{\codim}{\operatorname{codim}}

\renewcommand{\div}{\operatorname{div}}
\newcommand{\wei}{\operatorname{wt}}
\renewcommand{\Im}{\operatorname{Im}}
\newcommand{\car}{\operatorname{char}}

\newcommand{\M}{\operatorname{\ov{\m M}}}

\newcommand{\MM}{\operatorname{\ov{M}}}

\newcommand{\MMM}{\operatorname{\ov{M}}}

\newcommand{\Spr}{\operatorname{Sprout}_{g,n}(A_3)}
\newcommand{\EC}{\mathcal{EC}}
\newcommand{\NE}{\operatorname{NE}}
\newcommand{\NEb}{\operatorname{\ov{NE}}}
\newcommand{\Nef}{\operatorname{Nef}}

\newcommand{\Gm}{\operatorname{\mathbb{G}_m}}

\newcommand{\Ell}{\operatorname{Ell}}
\newcommand{\Tac}{\operatorname{Tac}}
\newcommand{\Exc}{\operatorname{Exc}}

\usepackage{tikz}
\usetikzlibrary{decorations.pathmorphing}
\usetikzlibrary{decorations.pathreplacing} 
\tikzset{dot/.style={
		circle,
		fill=black,
		inner sep=1.5pt,
}}

\title{On the first steps of the minimal model program for the moduli space of stable pointed curves}

\author{Giulio Codogni}
\address{Universit\`{a} Roma Tre, Dipartimento di Matematica e Fisica, Largo San Leonardo Murialdo, Rome, Italy} 
\email{codogni@mat.uniroma3.it}

\author{Luca Tasin}
\address{Dipartimento di Matematica F.\ Enriques, Universit\`a degli Studi di Milano, Via Cesare Saldini 50, 20133 Milano, Italy} 
\email{luca.tasin@unimi.it}

\author{Filippo Viviani}
\address{Universit\`{a} Roma Tre, Dipartimento di Matematica e Fisica, Largo San Leonardo Murialdo, Rome, Italy} 
\email{viviani@mat.uniroma3.it}

\begin{document}
\maketitle

\begin{abstract}
	The aim of this paper is to study all  the natural first steps of the minimal model program for the moduli space of stable pointed curves.  We prove that they admit a modular interpretation and we study their geometric properties. As a particular case, we recover the first few Hassett-Keel log canonical models.
	As a by-product, we produce many birational morphisms from the moduli space of stable pointed curves to alternative modular projective compactifications of the moduli space of pointed curves.  
\end{abstract}

\tableofcontents

\section*{Introduction}

The motivation of this work comes from the following vague but inspiring

\vspace{0.1cm}

\textbf{Question:} 
If we run a minimal model program of a moduli space, do all the steps admit a modular interpretation?

\vspace{0.1cm}

For example, this is  true for the moduli spaces of vector bundles over many classes of surfaces, see \cite{BM, Yos, Nue, LZ, BC, Giulio, CH2} or the surveys \cite{CH, Hui, Mac}.

\hspace{0.2cm}

In the present paper, we look at the above question for the coarse moduli space $\MM_{g,n}$ of Deligne-Mumford stable $n$-pointed curves of genus $g$. The main result of the paper is that \emph{all} the first natural 
steps of the MMP 	(=minimal model program) for $\MM_{g,n}$ admit a modular interpretation; more precisely,  they are moduli spaces of suitable singular curves.

\hspace{0.2cm}

The MMP for $\MM_{g,n}$ is closely related to the  Hassett-Keel program (see \cite{HH1,HH2,AFSV1,AFS2,AFS3}), which  is interested in studying the modular interpretation of the following  log canonical models 
\begin{equation}\label{E:logcan}
\MM_{g,n}(\alpha):=\Proj \bigoplus_{m\geq 0} H^0(\M_{g,n}, \lfloor m(K_{\M_{g,n}}+\psi+\alpha(\delta-\psi))\rfloor)
\end{equation}
of $\MM_{g,n}$ with respect to $K_{\M_{g,n}}+\psi+\alpha(\delta-\psi)$ as $\alpha$ decreases from $1$ to $0$. However, the point of view of the MMP is slightly different, since one is interested in contracting $K$-negative rays, or more generally faces, of the Mori cone  $\MM_{g,n}$ and then flipping them if the resulting contraction is small. It turns out that the first three steps of the Hassett-Keel program coincide with some of the steps of the MMP described in this paper, as we explain in detail towards the end of the introduction.

As a by-product of our investigation, we produce many morphisms (with connected fibres) from $\MM_{g,n}$ to other normal projective varieties. The number of these morphisms grows exponentially in  $(g,n)$. This gives a partial answer to \cite[Question, page 275]{GKM}), which asks for a classification of all such morphisms.
To the best of our knowledge, the only previously known birational morphisms  from $\MM_{g, n}$ (with $g >5$) were the first two steps of the above mentioned Hassett-Keel program, and, for $n=0$, the Torelli morphism from $\MM_g$ to the Satake compactification of the moduli space of principally polarized abelian varieties (note that it is unknown whether the Satake compactification admits a modular interpretation as moduli space of curves). 

The geometry of the morphisms that we construct in this paper will be further studied in our work \cite{CTV}. This paper is independent from its sequel \cite{CTV}, even though, for the sake of completeness, we have included here  some results from \cite{CTV}.

As a further by-product, we produce many new weakly modular (and sometimes also modular) compactification (in the sense of \cite[Sec. 2.1]{FS}) of the moduli space $M_{g,n}$ of $n$-pointed smooth curves of genus $g$,  see Remark \ref{R:modcom}. Moreover, our weakly modular compactifications involve curves  whose singularities are of the simplest kind, namely nodes, cusps and tacnodes, a problem that was explicitly discussed in \cite[p. 21--22]{FS}.

\subsection*{The first step}

As a warm-up, let us describe what are the possible first steps of the MMP for $\MM_{g,n}$, assuming for the moment that the characteristic of the base field $k$ is $0$. 

A first natural $K$-negative\footnote{In this introduction, we will be deliberately  vague on the canonical class $K$, what we are going to say works both for  the canonical class of the stack and of its coarse moduli space.} extremal ray of $\NEb(\MM_{g,n})$ is generated by the \emph{elliptic tail curve} $C_{\rm ell}$,  i.e. the curve $C_{\rm ell}$ (well-defined up to numerical equivalence) 
of $\MM_{g,n}$ parametrising a moving $1$-pointed elliptic curve $(E,p)$ attached in $p$ to a fixed $n+1$-pointed smooth irreducible curve of genus $g-1$. The contraction associated to the extremal ray $\bbR_{\geq 0}\cdot C_{\rm ell}$ has a modular meaning and it can be identified with the modular contraction
\begin{equation}\label{E:Up-ps}
\Upsilon:\MM_{g,n}\to \MM_{g,n}^{\ps},
\end{equation}
where $\MM_{g,n}^{\ps}$ is a projective normal $\bbQ$-factorial irreducible variety which is the coarse moduli space of the proper smooth Deligne-Mumford stack of \emph{$n$-pointed pseudostable}  curves of genus $g$ \footnote{We assume from now on that $(g,n)\neq (1,1),(2,0)$, because $\M_{1,1}^{\ps}$ is empty, while  $\M_{2,0}^{\ps}$ is neither separated nor with  finite inertia and $\MM_{2,0}^{\ps}$ is only an adequate  moduli space.}, i.e. $n$-pointed projective connected (reduced) curves of genus $g$ with nodes and cusps as singularities, not having elliptic tails and  with ample log canonical line bundle, and $\Upsilon$ sends an $n$-pointed stable curve  $C\in \MM_{g,n}(k)$ into the $n$-pointed pseudostable curve $\Upsilon(C)$ of $\MM_{g,n}^{\ps}(k)$ which is obtained by contracting  the elliptic tails of $C$ into cusps (see Propositions \ref{P:Mgps-DM}, \ref{P:Pic-Mgps}, \ref{P:div-contr}).

The morphism $\Upsilon$ is a birational divisorial contraction of relative Picard number one, and it is the unique such morphism at least if $g\geq 5$ by \cite[Prop. 6.4]{GKM}. Moreover, if the F-conjecture is true and $n\leq 2$, then a close inspection of formulae \cite[Thm. 2.1]{GKM} reveals that $\bbR_{\geq 0}\cdot C_{\rm ell}$  is the unique $K$-negative extremal ray of $\NEb(\MM_{g,n})$. On the other hand, if the F-conjecture is true and $n\geq 3$, then there are other extremal rays of  $\NEb(\MM_{g,n})$ that are $K$-negative, but $\bbR_{\geq 0}\cdot C_{\rm ell}$ is the unique one which is also $K+\psi$-negative. In both the  MMP and the Hassett-Keel program of $\MM_{g,n}$, it seems that the divisor class $K+\psi$ is more natural than the divisor $K$; one reason is that, on the stack, it is stable under the clutching morphisms (see e.g. \cite[Chap. XVII, Sec. 4]{GAC2}). The upshot of the above discussion is that the morphism \eqref{E:Up-ps} is the ``natural" (and conjecturally unique for $n\leq 2$)  first step of the MMP for $\MM_{g,n}$.

 \subsection*{The next steps}
 
 Let us now analyse what are the natural possible ways of continuing the MMP of $\MM_{g,n}$ by looking for $K$-negative extremal rays of $\MM_{g,n}^{\ps}$.

Given an hyperbolic pair $(g,n)$ (i.e. such that  $2g-2+n>0$),  consider the set
\begin{equation}\label{E:Tset}
T_{g,n}:=\left(\{\irr \} \cup \{ (\tau,I) : 0\leq \tau\leq g, I \subseteq [n]:=\{1,\ldots, n\}\} \setminus \{(0,\emptyset), (g,[n]) \}\right)/\sim,
\end{equation}
where $\sim$ is the equivalence relation such that $\irr$ is equivalent only to itself and $(\tau,I) \sim (\tau',I')$ if and only if $(\tau,I)= (\tau',I')$ or $(\tau',I')=(g-\tau,I^c )$, where $I^c=[n] \setminus I$. 
We will denote the class of $(\tau,I)$ in $T_{g,n}$ by $[\tau, I]$ and the class of $\irr$ in $T_{g,n}$ again by $\irr$. Set $T_{g,n}^*=T_{g,n}\setminus \{\irr\}$.

\begin{definition}\label{D:1strata}[Elliptic bridge curves]
Consider the following irreducible curves (well-defined up to numerical equivalence) in $\M_{g,n}^{\ps}$ (or in $\MM_{g,n}^{\ps}$), which we call \emph{elliptic bridge curves}:
\begin{enumerate}
\item  If $g\geq 2$ and $(g,n)\neq(2,0)$, we denote by $C(\irr)$  the closure of the curve formed by a varying $2$-pointed rational nodal elliptic curve  $(R,p,q)$ attached to a fixed $n$-pointed smooth irreducible curve $D$ of genus $g-2$ in the two points $p$ and $q$. If $(g,n)=(2,0)$, $C(\irr)$ is the closure of the curve formed by a varying rational curve with two nodes.

\item For every $\{[\tau, I], [\tau+1, I]\}=\{[\tau, I], [g-1-\tau, I^c]\}\subset T_{g,n}-\{(1,\emptyset), \irr\}$, we denote by $C([\tau, I], [\tau+1, I])$ the curve formed by a varying  $2$-pointed rational nodal elliptic curve  $(R,p,q)$ attached in $p$ to a fixed smooth  irreducible curve $D_1$ of genus $\tau$ and with marked points $\{p_i\}_{i\in I}$ and attached in $q$ to a fixed smooth  irreducible curve $D_2$ of genus $g-1-\tau$ and with marked points $\{p_i\}_{i\in I^c}$, with the convention that if $\tau=0$ and $I=\{k\}$ for some $k\in [n]$ then, instead of attaching the fixed curve $D_1$, we consider $p$ as the $k$-th marked points, and similarly for the case $(g-1-\tau, I^c)=(0,\{k\})$.
\end{enumerate}
The \emph{type}  of an elliptic bridge curve is defined as follows: $C(\irr)$ has type $\{\irr\}\subset T_{g,n}$ while $C([\tau, I], [\tau+1, I])$ has type equal to $\{[\tau, I], [\tau+1, I]\}\subset T_{g,n}$.

\end{definition}

\begin{figure}[!h]
	\begin{center}
		\begin{tikzpicture}[scale=0.6]

		\draw[very thick] (-2-14,2.5) to (7.5-14,2.5);
		
		\draw[very thick, decoration={brace, amplitude=10pt}, decorate] (-1.5-14,-3) -- (-1.5-14, 2);
		
		\coordinate (x) at (0-14,1.5);
		\coordinate (y) at (2.5-14, -2.5);
		\coordinate[label=right:$g-2$] (w) at (5-14,1.5);
		\draw [very thick] (x) to[in=180, out=-90] 
		node[dot, pos=0.8, label=left:$p_1$]{} 
		node[dot, pos=1, label=below:$\ldots$]{}
		(y) to[in=-90, out=0]  
		node[dot, pos=0.2, label=right:$p_n$]{}
		(w);

		\coordinate (a) at (-0.5-14,0);
		\coordinate (b) at (2.5-14,0.5);
		\coordinate (c) at (2.5-14, 2);
		\coordinate[label=right: $1$] (d) at (5.5-14, 0);
		\draw [very thick] (a) to[in=225,out=-15] (b) to[in=0, out=45] (c) to[out=180, in=135] (b) to[in=-165, out=-45] (d);

		\draw[very thick, decoration={brace, mirror, amplitude=10pt}, decorate] (7-14,-3) -- (7-14, 2);

		\draw[very thick] (-2,2.5) to (7.5,2.5);
		
		\draw[very thick, decoration={brace, amplitude=10pt}, decorate] (-1.5,-3) -- (-1.5, 2);
		
		\coordinate (x) at (0,1.5);
		\coordinate[label=below:{$\tau$}] (y) at (0, -2.5);
		\draw [very thick, in=105, out=-45] (x) to 
		node[dot, pos=0.6, label=left:$p_1$]{}
		node[dot, pos=0.75, label=left:$\vdots$]{}
		node[dot, pos=0.95, label=left:$p_k$]{}
		(y);

		\coordinate (a) at (-0.5,0);
		\coordinate (b) at (2.5,0.5);
		\coordinate (c) at (2.5, 2);
		\coordinate[label=above: $1$] (d) at (5.5, 0);
		\draw [very thick] (a) to[in=225,out=-15] (b) to[in=0, out=45] (c) to[out=180, in=135] (b) to[in=-165, out=-45] (d);
		
		\coordinate (w) at (5,1.5);
		\coordinate[label=below:{$g-\tau-1$}] (z) at (5, -2.5);
		\draw [very thick, in=75, out=-125] (w) to 
		node[dot, pos=0.6, label=right:$p_{k+1}$]{}
		node[dot, pos=0.75, label=right:$\vdots$]{}
		node[dot, pos=0.95, label=right:$p_n$]{}
		(z);
		
		\draw[very thick, decoration={brace, mirror, amplitude=10pt}, decorate] (7,-3) -- (7, 2);
		
		
		\end{tikzpicture}		
		
	\end{center}
	\caption{The elliptic bridge curves $C(\irr)$ and  $C([\tau, I], [\tau+1, I])$, where $I=\{1,\ldots, k\}$. The varying component is a 2-pointed rational nodal curve.}
	\label{F:ellipticbridgecurve}
\end{figure}
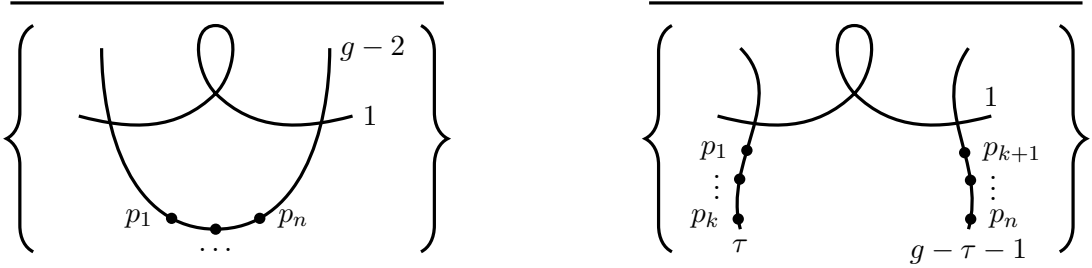

The elliptic bridge curves generate linearly independent extremal rays of $\NEb(\MM_{g,n}^{\ps})$ that are both $K$ and $K+\psi$-negative (see Proposition \ref{P:face}). For an arbitrary subset $T\subseteq T_{g,n}$, we denote by $F_T$ the $K$-negative face of $\NEb(\MM_{g,n}^{\ps})$ spanned by the classes of the elliptic bridge curves whose type is contained in $T$ (see Lemma \ref{L:poset-T} for the properties of $F_T$).

If the F-conjecture (see \cite[Conj. (0.2)]{GKM}) holds  true, then:
\begin{itemize}
\item The elliptic bridge curves are the unique $1$-strata of $\M_{g,n}^{\ps}$ which are $K_{\M_{g,n}^{\ps}}+\psi$-negative. In particular,  they are the unique $1$-strata of $\M_{g}^{\ps}$ which are $K_{\M_{g}^{\ps}}$-negative.
\item The elliptic bridge curves are the unique $K_{\M_{g,n}^{\ps}}$-negative curves of $\M_{g,n}^{\ps}$ which are the image of $K_{\M_{g,n}}$-positive $1$-strata of $\M_{g,n}$.
\end{itemize}
Hence the natural prosecution of the MMP for $\MM_{g,n}$ is the contraction of one of these extremal rays, or, more generally, of a face $F_T$, and its flip. The goal of our paper is to show that both the contractions of these $K$-negative faces and their flips have a modular description, and describe explicitly their geometrical properties.

\subsection*{$T$-semistable and $T^+$-semistable curves }

To give  these modular descriptions, we need new stability notions. Given a tacnode $p$ of an $n$-pointed projective curve of genus $g$ with ample log canonical line bundle, we define the type of $p$ as 
\begin{itemize}
\item $\type(p):=\{\irr\}\subseteq T_{g,n} $ if the normalisation of $C$ at  $p$ is connected;
\item  $\type(p):=\{[\tau,I],[\tau+1,I]\} \subseteq T_{g,n}$ if the normalisation of $C$ at $p$ consists of two connected components, one of which has arithmetic genus  $\tau$ and marked points $\{p_i\}_{i \in I}$ and the other has arithmetic genus $g-1-\tau$ and marked points $\{p_i\}_{i\in I^c}$.
\end{itemize}
In a similar fashion, we define the type of an $A_1$/$A_1$-attached elliptic chain (see Definition \ref{D:tailchain}).

\begin{definition}\label{D:Tdefi}[see Definition \ref{D:ourstacks}]
Let $T\subseteq T_{g,n}$. 
\begin{enumerate}[(i)]
\item We denote by $\M_{g,n}^T$ the stack of \emph{$T$-semistable curves}, i.e. $n$-pointed projective connected curves of genus $g$, having singularities that are nodes, cusps or tacnodes of type contained in $T$, not having neither $A_1$-attached elliptic tails nor $A_3$-attached elliptic tails and  with ample log canonical line bundle.
\item We denote by $\M_{g,n}^{T+}$ the stack of \emph{$T^+$-semistable curves}, i.e. $T$-semistable curves without any $A_1$/$A_1$-attached elliptic chain of type contained in $T$.
\end{enumerate}
\end{definition}

\subsection*{Main Results}

We can now state the three main results of this paper. We work over an algebraically closed field $k$. For some of our results, we will need to assume that the characteristic of $k$ is big enough with respect to the pair $(g,n)$, which we write as $\car(k)\gg (g,n)$  (see Definition \ref{A:char}),  and for some others that the characteristic of $k$ is zero.

The first main result describes the relation between the stacks of pseudostable curves, $T$-semistable curves and $T^+$-semistable curves and their good moduli spaces.

\begin{theoremalpha}[=Theorems \ref{T:algstack} and \ref{T:goodspaces}]
Assume that $(g,n)\neq (2,0)$ and let $T\subset T_{g,n}$. 
\begin{enumerate}
\item \label{T:A1} The stack $\M_{g,n}^T$ is algebraic, smooth, irreducible  and of finite type over $k$ and we have open embeddings 
$$
\xymatrix{
\M_{g,n}^{\ps} \ar@{^{(}->}[r]^{\iota_T} & \M_{g,n}^T & \ar@{_{(}->}[l] _{\iota_T^+}  \M_{g,n}^{T,+}.
}
$$
\item \label{T:A2} Assume that ${\rm char}(k)\gg (g,n)$. Then the algebraic stacks $\M_{g,n}^T$ and $\M_{g,n}^{T+}$ admit good moduli spaces  $\MMM_{g,n}^T$ and $\MMM_{g, n}^{T+}$ respectively, which are  proper normal irreducible algebraic spaces over $k$. Moreover, there exists a commutative diagram 
\begin{equation*}
\xymatrix{
\M_{g,n}^{\ps} \ar@{^{(}->}[r]^{\iota_T} \ar[d]^{\phi^{\ps}} & \M_{g,n}^T \ar[d]^{\phi^T} & \M_{g,n}^{T+} \ar@{_{(}->}[l]_{\iota_T^+} \ar[d]^{\phi^{T+}}\\
\MM_{g,n}^{\ps} \ar[r]^{f_T}&   \MMM_{g,n}^T  & \MMM_{g,n}^{T+} \ar[l]_{f_T^+} \\
}
\end{equation*}
where the vertical maps are the natural morphisms to the  good moduli spaces (indeed also $\phi^{\ps}$ is a good moduli space if ${\rm char}(k)\gg (g,n)$) and the bottom horizontal morphisms $f_T$ and $f_T^+$   are proper (and birational if $(g,n)\neq (1,2)$)   morphisms. 

\end{enumerate} 
\end{theoremalpha}

Part \eqref{T:A1} of the above Theorem (which coincides with Theorem \ref{T:algstack}) is proved in Section \ref{S:stacks}. In this section, we also investigate the properties of 
  the stacks $\M_{g,n}^T$ and $\M_{g,n}^{T+}$: we describe the containment relation among all these different stacks in Proposition \ref{P:equaT}; we describe the closed points and the isotrivial specialisations of $\M_{g,n}^T$ and $\M_{g,n}^{T+}$ in Propositions \ref{prop:Tclosed} and \ref{P:T+closed}; we describe the Picard groups of 
  $\M_{g,n}^T$ and $\M_{g,n}^{T+}$ in Corollary \ref{C:Pic-ps}.

Part \eqref{T:A2} of the above Theorem is proved in Section \ref{S:modspace} (see Theorem \ref{T:goodspaces}). The strategy is the same as the one pioneered by Alper-Fedorchuk-Smyth-van der Wyck in \cite{AFSV1} and \cite{AFS2} to perform the first steps of the Hassett-Keel program. The key property is the fact that the open embeddings of stacks in part \eqref{T:A1} arise from  local VGIT (=variation of 
geometric invariant theory) with respect to $\delta-\psi$ (in the sense of \cite[Def. 3.14]{AFSV1}). One little improvement of the methods of loc. cit. is provided in Proposition  \ref{P:fincov} which generalises \cite[Prop. 1.4]{AFS2} from characteristic zero to arbitrary characteristic and it allows us to construct the good moduli spaces  provided that the  automorphism group schemes of the algebraic stacks  are linearly reductive, which is true if the characteristic is big enough (see Lemma \ref{L:linred} and Remark \ref{L:linred}).

After the completion of this work, Alper-Halpern-Leistner-Heinloth posted on arXiv the preprint \cite{AHLH}, where they provide a necessary and sufficient criterion  for a stack to admit a good moduli space. 
Hence it should be possible to prove the existence of the good moduli spaces $\MMM_{g,n}^T$ and $\MMM_{g, n}^{T+}$ (and also Proposition  \ref{P:fincov}) using their criterion; however, we have not checked the details.   

Our second main result identifies, in characteristic zero,  the morphism $f_T$ with the contraction of the $K$-negative face $F_T$ of the Mori cone of $\MM_{g,n}^{\ps}$. 

\begin{theoremalpha}[=Theorem \ref{thm:contr}] \label{T:B}
Assume that ${\rm char}(k)=0$ and that $(g,n)\neq (2,0)$, and  let $T\subseteq T_{g,n}$. 
The good moduli space $\MMM_{g,n}^T$ is  projective and the morphism $f_T: \MM_{g,n}^{\ps}\to \MMM_{g,n}^T$ coincides with the contraction of the face $F_T$.
\end{theoremalpha}

The proof of the above Theorem follows, using the rigidity Lemma \ref{L:rigidity}, from the fact that $f_T$ is a contraction with the property that a curve $C\subset \M_{g,n}^{\ps}$ is contracted by $f_T$ if and only if  its class $[C]$ lies in $F_T$ (see Lemma \ref{L:poset-T} and Proposition \ref{P:fibers1}). From the above Theorem and standard corollaries of the cone theorem, we derive a description of the rational Picard group of $\MMM_{g,n}^T$ and of its nef/ample cone (see Corollary \ref{C:contr}).

In our sequel paper \cite{CTV}, we will investigate the geometric properties of the moduli space $\MMM_{g,n}^T$ and of the morphism $f_T$ (see Proposition \ref{P:geomT} for a recap of some of the results of loc. cit.).

Our last main result is a description of the morphism $f_T^+:\MMM_{g,n}^{T+}\to \MM_{g,n}^{\ps}$ (which turns out to be a projective contraction, see Propositions \ref{P:relample+} and \ref{P:fibers+}) as the flip (in the sense of Definition \ref{D:Dflip}) of $f_T$  with respect to suitable $\bbQ$-line bundles.

\begin{theoremalpha}[=Theorem \ref{T:flip+}, Corollary \ref{C:projT+}, Corollary \ref{C:K-flip}] \label{T:C}
Assume that  $\car(k)\gg (g,n)$ and $(g,n)\neq (2,0), (1,2)$, and let $T\subseteq T_{g,n}$. 
Let  $L\in \Pic(\MM_{g,n}^{\ps})_{\bbQ}=\Pic(\M_{g,n}^{\ps})_{\bbQ}=\Pic(\M_{g,n}^{T})_{\bbQ}$. The morphism $f_T^+$ is the $L$-flip of $f_T$ if and only if  $L$ is $f_T$-antiample and the restriction of $L$ to $\M_{g,n}^{T+}$ descends to a $\bbQ$-line bundle on $\MMM_{g,n}^{T+}$.
In particular:
\begin{enumerate}[(i)]
\item  The  morphism $f_T^+: \MMM_{g,n}^{T,+}  \to \MMM_{g,n}^T$ is the $(K_{\M_{g,n}^{\ps}}+\psi)$-flip of $f_T$. 
\item  The  morphism $f_T^+: \MMM_{g,n}^{T,+}  \to \MMM_{g,n}^T$ is the $K_{\MM_{g,n}^{\ps}}$-flip of $f_T$ if and only if $\MMM_{g,n}^{T,+}$ is $\bbQ$-Gorenstein, i.e. if and only if $T$ does not contain subsets of the form $\{[0,\{j\}],[1,\{j\}],[2,\{j\}]\}$ for some $j\in [n]$ or  $(g,n)=(3,1), (3,2), (2,2)$
\end{enumerate}
Therefore, $\MMM_{g,n}^{T+}$ is projective if $\car(k)=0$.
\end{theoremalpha} 

In proving  the above result, we investigate the properties of the space $\MMM_{g,n}^{T+}$ and of the  morphism $f_T^+:\MMM_{g,n}^{T+}\to \MM_{g,n}^{\ps}$ in Section \ref{S:T+}.
We compute the rational Picard group of $\MMM_{g,n}^{T+}$ in Proposition \ref{P:lbdesc} (and in particular, we describe explicitly when a $\bbQ$-line bundle on $\M_{g,n }^{T+}$ descends to a $\bbQ$-line bundle on $\MMM_{g,n}^{T+}$) and we describe when $\MMM_{g,n}^T$ is $\bbQ$-factorial or $\bbQ$-Gorenstein in Corollary \ref{C:Gor-fact}.
Moreover, we describe the exceptional locus of $f_T^+$ in Proposition \ref{P:fibers+} and its relative Mori cone in Proposition \ref{P:tacinter}. 

Finally, we prove in Corollary \ref{C:compflips} that, whenever $f_T: \MM_{g,n}^{\ps} \to \MMM_{g,n}^T$ is small and $\MMM_{g,n}^{T,+}$ is $\bbQ$-factorial, for any $\bbQ$-line bundle $L$ on $\MM_{g,n}^{\ps}$ which is $f_T$-antiample, the rational map $(f_T^+)^{-1}\circ f_T:\MM_{g,n}^{\ps} \dashrightarrow \MMM_{g,n}^{T,+}$ can be decomposed  as a sequence of elementary $L$-flips.

A posteriori, we can recover our stacks of $T$-semistable and $T^+$-semistable curves as semistable locus for convenient line bundles, as explained in the following remark.
\begin{remark}\label{Alper_stability}
	Let $\U^{lci}_{g,n}$ be the  stack of $n$-pointed  curves of arithmetic genus $g$ with locally complete intersection singularities and with ample log canonical line bundle, as in Section   \ref{S:Tstacks}.
		Recall that $\U_{g,n}^{lci}$ is a smooth and irreducible algebraic stack of finite type over $k$. The stack $\M_{g,n}^T$ of $T$-semistable curves is an open substack of $\U_{g,n}^{lci}$, and its complement  contains a unique divisor, namely the divisor $\Delta_{1,\emptyset}$ parametrising curves with an elliptic tail.
	
	Assume that $\car(k)=0$ and consider the projective good moduli space $\phi^T:\M_{g,n}^T\to \MMM_{g,n}^T$ (see Theorem \ref{T:B}). Let $M$ be an ample line bundle on  $\MMM_{g,n}^T$ and let $\calL$ be a line bundle on $\U_{g,n}^{lci}$ whose restriction to $\M_{g,n}^T$ coincides with $(\phi^T)^*(M)$ (note that such a line bundle $\calL$ exists since $\U_{g,n}^{lci}$ is regular). 
	By combining \cite[Thm. 11.5]{Alper} and the proof of \cite[Thm. 11.14(ii)]{Alper}, it follows that the stack $\mathcal{M}_{g,n}^T$ is exactly the semistable locus of $\U_{g,n}^{lci}$ with respect to $\calL_N:=\calL\otimes \cO_{\U_{g,n}^{lci}}(N\Delta_{1,\emptyset})$ for $N\gg 0$ (in the sense of \cite[Def. 11.1]{Alper}) and $\MMM_{g,n}^T$ is the good moduli space provided by \cite[Thm. 11.5]{Alper}. A similar statement holds true for $\phi^{T+}:\M_{g,n}^{T+}\to \MMM_{g,n}^{T+}$.
\end{remark}
\subsection*{Relation with the Hassett-Keel program}
We can now describe in detail the connection between our work and the first steps of the Hassett-Keel program, as established in \cite{HH1,HH2,AFSV1,AFS2,AFS3}.
From \cite[Thm. 1.1]{AFS3} and Proposition \ref{P:div-contr}\eqref{P:div-contr3}, it follows that   (assuming $\car(k)=0$):
\begin{equation}\label{E:lowMalpha}
\MM_{g,n}(\alpha)=
\begin{cases}
\MM_{g,n} & \text{ if } \frac{9}{11}<\alpha\leq 1, \\
\MM_{g,n}^{\ps} & \text{ if } \frac{7}{10}<\alpha\leq \frac{9}{11},\\
 \MMM_{g,n}^{T_{g,n}} & \text{ if } \alpha=\frac{7}{10},\\
 \MMM_{g,n}^{T_{g,n}+} & \text{ if } \frac{2}{3}<\alpha< \frac{7}{10}.\\
\end{cases}
\end{equation}
Therefore, Theorems \ref{T:B} and \ref{T:C} implies that at the second critical value $7/10$ of  the Hassett-Keel program, the variety  $\MM_{g,n}(7/10)$ is obtained from $\MM_{g,n}(7/10+\epsilon)\cong \MM_{g,n}^{\ps}$ by contracting the entire elliptic bridge face of the Mori cone of $\MM_{g,n}^{\ps}$ (whose dimension is computed in Remark \ref{R:bridgecurves}), while the variety $\MM_{g,n}(7/10-\epsilon)$ is obtained by flipping the above contraction with respect to $K+\psi$.
As a by-product of our analysis we obtain some results on the geometry of $\MM_{g,n}(7/10)$ and of $\MM_{g,n}(7/10-\epsilon)$: we compute their rational Picard groups (see Example~\ref{R:LMMP} and Corollary \ref{C:PicT+}) and we determine when they are $\bbQ$-factorial or $\bbQ$-Gorenstein (see Proposition~\ref{P:geomT} and Remark \ref{R:QfactHK}).
\subsection*{Open questions}
This work leaves out some  interesting questions, which we hope to be able to address in the future:
\begin{enumerate}
\item For any $\bbQ$-line bundle $L$  on $\MM_{g,n}^{\ps}$ which is $f_T$-antiample, we can construct the $L$-flip of $f_T$ at least if $\car(k)=0$ (see Lemma \ref{L:Dflip}\eqref{L:Dflip2}). Theorem \ref{T:C} implies that the $L$-flip of $f_T$ coincides with $f_T^+$, provided that the restriction of $L$ to $\M_{g,n}^{T+}$ is $T^+$-compatible. If this condition fails (which can only happen if $\MMM_{g,n}^{T+}$ is not $\bbQ$-factorial), is there a modular description of the $L$-flip of $f_T$? 
 
\item Can we describe modularly all the small $\bbQ$-factorialisations of $\MMM_{g,n}^T$, i.e.\ all the $\bbQ$-factorial normal proper algebraic spaces  endowed with a small contraction $X\to \MMM_{g,n}^T$?
Even more, it would be interesting to determine  the  chamber decomposition 
         $$
	\Cl(\MMM_{g,n}^T)_{\bbR}/\Pic(\MMM_{g,n}^T)_{\bbR}= \coprod \Nef(X_i/\MM_{g}^{T}),
	$$
where $X_i\to \MMM_{g,n}^T$ vary among all the small $\bbQ$-factorialisations  of $\MM_{g}^{T}$ (see \cite[Exercise 116]{Exercises} and \cite[Thm. 12.2.7]{Matsukibook}).

In this paper, we have described modularly some of the $\bbQ$-factorialisations of $\MMM_{g,n}^T$, namely: $\MMM_{g,n}^{T^{\div}}$ (which coincides with $\MM_{g,n}^{\ps}$ whenever $f_T$ is small, see Proposition 
\ref{P:geomT}) and $\MMM_{g,n}^{S+}$ for all subsets $S\subseteq T$ that satisfy the conditions of Corollary \ref{C:Gor-fact}\eqref{C:Gor-fact1}. However, when $\MMM_{g,n}^{T+}$ is not $\bbQ$-factorial, we know for sure there are other $\bbQ$-factorialisations, namely the $\bbQ$-factorial flips of the morphisms $f_S:\MM_{g,n}^{\ps}\to \MMM_{g,n}^S$ where $S\subseteq T$ and $\MMM_{g,n}^{S+}$ is not $\bbQ$-factorial (see the previous question).

 \item  Theorem \ref{T:B} implies that the moduli space $\MMM_{g,n}^T$  (and hence also $\MMM_{g,n}^{T+}$) is projective if $\car(k)=0$. Is this true in positive characteristics (big enough so that $\MMM_{g,n}^T$ exists)?  For the special case $T=T_{g,n}$, this is achieved in Example \ref{R:LMMP} building upon the GIT(=geometric invariant theory) analysis of \cite{HH2} for $n=0$. In the general case, when no GIT construction seems plausible,  one could try to use Koll\'ar's  approach \cite{KolProj}, but the main difficulties  are that the stack $\M_{g,n}^{T}$ does not have finite stabilisers and it parametrises non nodal curves. 
 
 \item Can we find some (or all) $\bbQ$-line bundles $L$ (perhaps of adjoint type) on $\MM_{g,n}$ for which $\Proj \bigoplus_{m\geq 0} H^0(\MM_{g,n}, \lfloor mL\rfloor)$ is isomorphic to $\MMM_{g,n}^T$ or $\MMM_{g,n}^{T+}$? A quite complete answer for $\MM_{g,n}^T$ is contained in \cite[Sec. 4]{CTV}. 

\end{enumerate}

\subsection*{Acknowledgment}
We had the pleasure and the benefit of conversations with  J.\ Alper, E.\ Arbarello, G.\ Farkas, M.\ Fedorchuk, R.\ Fringuelli,  A.\ Lopez,  Zs.\ Patakfalvi and R.\ Svaldi about the topics of this paper.
We  thank the referee for her/his careful reading of the paper and for the suggestions and questions that helped in improving the paper.

The first author 
is funded by the MIUR {\it ``Excellence Department 
Project"} MATH@TOV, awarded to the Department of Mathematics, 
University of Rome, Tor Vergata, CUP E83C18000100006, and the PRIN 
2017 {\it ``Advances in Moduli Theory and Birational Classifictation"}.
The second author was supported during part of this project by the DFG grant ``Birational Methods in Topology and Hyperk\"ahler Geometry".  The third author is a member of the CMUC (Centro de Matem\'atica da Universidade de Coimbra), where part of this work was carried over.  The three authors are members of the GNSAGA group of INdAM.

\section*{Notation and background}\label{S:nota}

We work over a fixed algebraically closed field $k$ of arbitrary characteristic. Further restrictions on the characteristic of $k$ will be specified when needed.

\subsection*{Notations for curves}
An \textbf{$n$-pointed curve} $(C,\{p_i\}_{i=1}^n)$ is a connected, reduced, projective $1$-dimensional scheme $C$ over $k$ with $n$ distinct smooth  points $p_i\in C$ (called \emph{marked points}). If the number of marked points is clear from the context, we will denote an $n$-pointed curve simply by $C$. The (arithmetic) genus of a curve $C$ will be denoted by $g(C)$. The \emph{log canonical line bundle}
 of a $n$-pointed curve $(C,\{p_i\}_{i=1}^n)$ is $\omega_C^{\log}:=\omega_C(\sum_{i=1}^n p_i)$. 
 
A singular point $p\in C$ is called:
 \begin{itemize}
  \item \emph{node} (or singularity of type $A_1$) if the complete local ring $\wh{\cO_{C,p}}$ of $C$ at $p$ is isomorphic to   $k[[x,y]]/(xy)$ (or to $k[[x,y]]/(y^2-x^{2}))$ if $\car(k)\neq 2$);
 \item \emph{cusp} (or singularity of type $A_2$) if $\wh{\cO_{C,p}}\cong k[[x,y]]/(y^2-x^{3}))$;
 \item \emph{tacnode} (or singularity of type $A_3$) if  $\wh{\cO_{C,p}}$ is isomorphic to $ k[[x,y]]/(y(y-x^{2}))$ (or to $k[[x,y]]/(y^2-x^{4}))$ if $\car(k)\neq 2$).
 \end{itemize}

When dealing with the deformation theory of a tacnode, we will often assume that $\car(k)\neq 2$ for simplicity (note that the semiuniversal deformation space of a tacnode has dimension $3$ if $\car(k)\neq 2$ and $4$ if $\car(k)=2$).

We use the notation $\Delta=\Spec R$ and $\Delta^*=\Spec K$, where $R$ is a $k$-discrete valuation ring with residue field $k$ and fraction field $K$; we set $0$, $\eta$ and $\ov \eta$ to be, respectively,  the closed point, the generic point and a geometric generic point of $\Delta$. 
Given a flat and proper family $\pi:\cC\to \Delta$, we denote by $\cC_0$ the special fibre, by $\cC_{\eta}$ the generic fibre and by $\cC_{\ov \eta}$ a geometric generic fibre. 

An \textbf{isotrivial specialisation} is a flat and proper  family $\pi:\cC\to \Delta$ of curves such that the restriction $\cC\times_{\Delta} \Delta^*\to \Delta^*$ is trivial, i.e. $\cC\times_{\Delta} \Delta^*\cong C\times_k \Spec K$ for some curve $C$ defined over $k$. 
In this case, we say that $C$ \emph{isotrivially specialises} to $\cC_0$, and we write $C\leadsto \cC_0$. The above isotrivial specialisation is called non-trivial if $\cC_0\not\cong C$, or, equivalently (cf. \cite[Prop. 2.6.10]{sernesi}), if $\cC\not\cong C\times_k \Delta$. 
Similar definitions can be given for pointed curves, by requiring that the family $\pi:\cC\to \Delta$ admits sections. 

\subsection*{Notations for Mori theory}

A proper morphism $f:X\to Y$ between two reduced algebraic spaces of finite type over $k$ is called a \textbf{contraction}  if $f_*\cO_X=\cO_Y$.

Given a reduced proper $k$-algebraic space $X$, we denote by $N^1(X)\cong \bbZ^{\rho_X}$ the (numerical) \emph{Neron-Severi} group, and we set 
$N^1(X)_{\bbR}=N^1(X)\otimes_{\bbZ} \bbR$ (the real Neron-Severi vector space).
Via the intersection product, the dual of $N^1(X)$ is naturally identified with the group $N_1(X)$ of $1$-cycles up to numerical equivalence and we set  
 $N_1(X)_{\bbR}=N_1(X)\otimes_{\bbZ} \bbR$.
 Inside $N_1(X)_{\bbR}$, there is the \emph{effective cone of curves} $\NE(X)$, which is the convex cone consisting of all effective $1$-cycle on $X$, and its closure $\NEb(X)$, the \emph{Mori cone}.
  Given a contraction $\pi:X\to Y$ between reduced proper $k$-algebraic spaces,  the  \emph{$\pi$-relative effective cone of curves} is the convex subcone $\NE(\pi)$ of 
  $\NE(X)$ spanned by the integral curves that are contracted by $\pi$ (i.e. the integral curves $C$ of $X$ such that $\pi(C)$ is a closed point of $Y$), and its closure  $\NEb(\pi):=\ov{\NE(\pi)}\subseteq \NEb(X)$ is called 
   the \emph{$\pi$-relative Mori cone}.
   We will use the following facts:
 \begin{itemize}
\item  If $Y$ is  projective, then $\NEb(\pi)$ is a face of $\NEb(X)$ and, hence, $\NE(\pi)$ is a face of $\NE(X)$ (the proof of \cite[Prop. 1.14(a)]{Deb} for $\NE(\pi)$ works also for $\NEb(\pi)$).  Moreover, the class of an integral curve $[C]$ belongs to $\NE(\pi)$ if and only if $\pi_*([C])=0$. 
\item If $X$ and $Y$ are projective (which implies that also $\pi$ is projective), then $\pi$ is uniquely determined by $\NE(\pi)$ up to isomorphism (see \cite[Prop. 1.14(b)]{Deb}).
\item If $\pi$ is projective, then the \emph{relative Kleiman's ampleness criterion}  holds: a  Cartier divisor $D$ on $X$ is $\pi$-ample if and only if $D$ is positive on $\NEb(\pi)\setminus \{0\}$ (see \cite[Thm. 1.44]{KM}).
\end{itemize}

 Given a projective $k$-variety $X$ and a face $F$ of $\NE(X)$, if there exists a (projective) contraction  $\pi:X\to Y$ into a  projective $k$-variety $Y$  such that $\NE(\pi)=F$ then $\pi:X\to Y$ (which is unique by what said above) is called the \textbf{contraction} of the face $F$ and it will be denoted by $\pi_F: X\to X_F$.
Note that not all the faces $F$ of $\NE(X)$ can have an associated contraction; a necessary condition for that to happen is that the closure of $F$ must be equal to a face of $\NEb(X)$. Contraction of faces of the effective cone of curves can also be characterised as follow.

\begin{lemma}\label{L:rigidity}
 Let $X$ be a projective $k$-variety and let $F$ be face of $\NE(X)$ for which there exists a contraction $\pi_F:X\to X_F$. If $f:X\to Y$ is a contraction  onto a reduced proper (not necessarily projective!) $k$-algebraic space $Y$ such that an integral curve $C\subset X$ is contracted by $f$  if and only if $[C]\in F$, 
 then there exists an isomorphism $X_F\cong Y$ under which $f=\pi_F$.
\end{lemma}
\begin{proof}
By the assumption on $f$ and the definition of the contraction $\pi_F$ of $F$, it follows that an integral curve $C\subset X$ is contracted by  $f$ if and only if it is contracted by $\pi_F$. Since $X$ is assumed to be projective, the morphisms $f$ and $\pi_F$ are projective contractions, which implies that their closed fibres are connected projective $k$-varieties.  Using suitable hyperplane sections, we can connect any two closed points  of a closed fibre of $f$ (resp. $\pi_F$)  by a chain of integral curves contained in the given fibre of $f$ (resp. $\pi_F$). Hence, from what said above for curves, we conclude that a closed subscheme of $X$ is a fibre of $f$ if and only if it is a fibre of $\pi_F$. 
 
We can now apply the rigidity Lemma of \cite[Lemma 1.15]{Deb} in order to conclude that $f$ factors through $\pi_F$ and $\pi_F$ factors through $f$. This implies that there exists an isomorphism $Y\cong X_F$ under which $f=\pi_F$. 
\end{proof}

In Lemma \ref{L:rigidity}, the assumption that a curve $C\subset X$ is contracted by $f$ if and only if $[C]\in F$ can not be replaced by the weaker condition that $\NE(f)=F$, as the following example shows. 

\begin{example}
Consider a projective smooth complex threefold $X$ with a $K_X$-negative extremal ray $R$ such that the contraction of $R$, $\pi_R: X \to Y$, contracts a divisor $E \cong \P^1 \times \P^1$ to a singular point in $Y$. In this case, by \cite[Thm. 3.3]{Mori82}, the normal bundle of $E$ is $\mathcal{O}(-1,-1)$, and the two rulings of $E$ are numerically equivalent on $X$.  Such a threefold does exist by \cite[Section 10, Example 3.44.2]{Mori82}.

By Nakano's theorem, $E$ can also be contracted analytically along one of its ruling by a holomorphic map  $f: X \to Z$. The end result $Z$ is a proper complex smooth algebraic space (or equivalently a proper  Moishezon manifold) and $\NE(f)=R$. The complex manifold $Z$ is therefore non projective and it can be seen as a small resolution of $Y$.
\end{example}

\begin{section}{The moduli stacks of $T$-semistable and $T^+$-semistable curves} \label{S:stacks}

The aim of this section is to define the relevant moduli stacks of $n$-pointed curves, with which we will work throughout the paper.

\subsection{Special subcurves}

In this subsection, we will introduce some special subcurves that will be  used in the definition of our moduli stacks. The reader can safely skip this section at a first reading and come back to the relevant definitions, when they will be needed.

\begin{definition}[Tails, bridges and chains, see {\cite[Def. 2.1 and 2.3, Lemma 2.13]{AFSV1}}]\label{D:ell}
\noindent 
\begin{enumerate}
\item An \emph{elliptic tail} is a $1$-pointed irreducible curve $(E,q)$ of arithmetic genus $1$ (i.e. $E$ is either a smooth elliptic curve or a rational curve with one node or one cusp). 

\item An \emph{elliptic bridge} is a $2$-pointed curve $(E,q_1,q_2)$ of arithmetic genus $1$ which is either irreducible  or it has two  rational smooth components $R_1$ and $R_2$ that meet in either two nodes or one tacnode   and such that $q_i\in R_i$ for $i=1,2$.
\item An \emph{elliptic chain of length $r$} is a $2$-pointed curve $(E,q_1,q_2)$ which admits a finite, surjective morphism
$
\gamma: \bigcup_{i=1}^r (E_i,p_{2i-1},p_{2i}) \to (E,q_1,q_2)
$
such that:
\begin{enumerate}
\item $(E_i,p_{2i-1},p_{2i})$ is an elliptic bridge for $i=1,\ldots,r$;
\item $\gamma$ induces an open embedding of $E_i \setminus \{p_{2i-1}, p_{2i}\}$ into $E \setminus \{q_1,q_2\}$ for $i=1,\ldots,r$; \item $\gamma(p_{2i})=\gamma(p_{2i+1})$ is a tacnode for $i=1,\ldots,r-1$;
\item $\gamma(p_1)=q_1$ and $\gamma(p_{2r})=q_2$.
\end{enumerate}
\end{enumerate}
\end{definition}
Note that an elliptic chain of length $r$ has arithmetic genus $2r-1$. An elliptic chain of length $1$ is just an elliptic bridge.

\begin{figure}[!h]
	\begin{center}
		\begin{tikzpicture}[scale=0.7]
		
		\coordinate (a) at (0, 0);
		\coordinate[label=right:{$g$=1}] (b) at (4, 0);
		\draw [very thick, in=150, out=-45] (a) to node[dot, pos=0.2, label=below:$q$]{} (b);

		\coordinate (c) at (0+10, 0);
		\coordinate[label=right:{$g$=1}] (d) at (4+10, 0);
		\draw [very thick, in=150, out=-45] (c) to 
		node[dot, pos=0.2, label=below:$q_1$]{}  
		node[dot, pos=0.7, label=below:$q_2$]{} 
		(d);

		\end{tikzpicture}
	\end{center}
	\vspace{-0.5cm}
	\caption{An elliptic tail and an elliptic bridge.}\label{F:tail}
\end{figure}
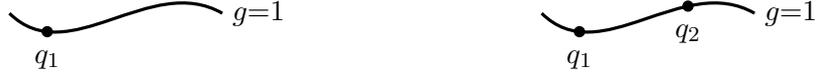

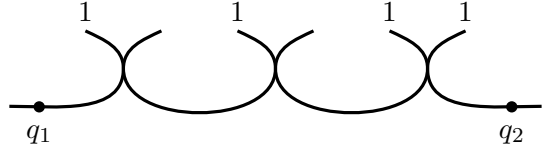
\begin{figure}[!h]
	\begin{center}
		\begin{tikzpicture}[scale=0.5]
		
		\coordinate (a) at (-1, 0);
		\coordinate (b) at (2, 1);
		\coordinate[label=above:1] (c) at (1, 2);
		\draw [very thick] (a)   to[out=0, in=-90]  
		node[dot, pos=0.2, label=below:$q_1$]{} 
		(b)  to[out=90, in=-25] (c);

		\coordinate (d) at (0+3, 2);
		\coordinate (e) at (2+4, 1);
		\coordinate[label=above:1] (f) at (1+4, 2); 
		\draw [very thick] (d)  to[out=205, in=90]   (b)  to[out=-90, in=-90] (e) to[out=90, in=-25] (f);

		\coordinate (g) at (0+7, 2);
		\coordinate (h) at (2+8, 1);
		\coordinate[label=above:1] (i) at (1+8, 2); 
		\draw [very thick] (g)  to[out=205, in=90]   (e)  to[out=-90, in=-90] (h) to[out=90, in=-25] (i);
		
		\coordinate[label=above:1] (l) at (0+11, 2);
		\coordinate (m) at (1+12, 0);
		\draw [very thick] (l)  to[out=205, in=90]   (h)  to[out=-90, in=180] 
		node[dot, pos=0.8, label=below:$q_2$]{} 
		(m);

		\end{tikzpicture}
	\end{center}
     \vspace{-0.5cm}
	\caption{An elliptic chain of length 4. The numbers 1 indicate the genus of the irreducible components.}\label{F:chain}
\end{figure}

\begin{definition}[Attached elliptic tails and chains, see {\cite[Def. 2.4]{AFSV1}}]\label{D:tailchain}
Let $(C, \{p_i \}_{i=1}^n)$ be an $n$-pointed curve of genus $g$. Let $k, k_1, k_2$ be equal to $1$ or $3$. 
\begin{enumerate}
\item $(C, \{p_i \}_{i=1}^n)$ has a \emph{$A_k$-attached elliptic tail} if there exists a finite morphism $\gamma: (E,q) \to (C, \{p_i \}_{i=1}^n)$ (called gluing morphism) such that:
\begin{enumerate}[(a)]
\item  $(E,q)$ is an elliptic tail;
\item $\gamma$ induces an open embedding of $E-\{q\}$ into $C- \cup_{i=1}^n \{p_i \}$;
\item  $\gamma(q)$ is an $A_k$-singularity.
\end{enumerate}
\item $(C, \{p_i \}_{i=1}^n)$ has an \emph{$A_{k_1}/A_{k_2}$-attached elliptic chain}  (of length $r$) if there exists a finite morphism $\gamma: (E,q_1,q_2) \to (C, \{p_i \}_{i=1}^n)$ (called gluing morphism) such that:
\begin{enumerate}[(a)]
\item $(E,q_1,q_2)$ is an elliptic chain (of length $r$);
\item  $\gamma$ induces an open embedding of $E-\{q_1,q_2\}$ into $C- \cup_{i=1}^n \{p_i \}$; 
\item $\gamma(q_i)$ is an $A_{k_i}$-singularity or $k_i=1$ and  $\gamma(q_i)$ is a marked point (for $i=1,2$).
\end{enumerate}
An $A_{k_1}/A_{k_2}$-attached elliptic chain of length $1$ is also called an \emph{$A_{k_1}/A_{k_2}$-attached elliptic bridge}. 
An $A_k/A_k$-attached elliptic chain $\gamma: (E,q_1,q_2) \to (C, \{p_i \}_{i=1}^n)$ of length $r$ such that $\gamma(q_1)=\gamma(q_2)$ is called \emph{closed}. In this case $\gamma$ is surjective and $(g,n)=(2r-1+\frac{k+1}{2},0)$.

\end{enumerate}
\end{definition}

\begin{figure}[!h]
	\begin{center}
		\begin{tikzpicture}[scale=0.7]
		
		\coordinate (a) at (0, 0);
		\coordinate[label=right:{1}] (b) at (2, 0);
		\draw [very thick, in=205, out=25] (a) to (b);

		\coordinate (x) at (0.3,1);
		\coordinate[label=right:{$g-1$}] (y) at (0.3, -2.5);
		\draw [very thick, in=105, out=-45] (x) to (y);

		\coordinate (x) at (0.3+7,1);
		\coordinate (y) at (0.7+7,0);
		\coordinate[label=right:{$g-2$}] (z) at (0.3+7, -2.5);
		\draw [very thick] (x) to[in=90, out=-45] (y) to[in=105, out=-90] (z);

		\coordinate (a) at (1.5+7, 0.7);
		\coordinate (b) at (0.7+7, 0);
		\coordinate[label=above:1] (c) at (2.3+7, -0.5);
		\draw [very thick] (a)  to[out=205, in=90]   (b)  to[out=-90, in=180] (c);

		\coordinate (x) at (-1.5+15.5,1);
		\coordinate[label=right:$\tau$] (y) at (-1.5+15.5, -2.5);
		\draw [very thick, in=105, out=-45] (x) to (y);
		
		\coordinate (a) at (-2+15.5, 0);
		\coordinate[label=above:1] (b) at (2+15.5, 0);
		\draw [very thick, in=205, out=25] (a) to (b);
		
		\coordinate (w) at (1.5+15.5,1);
		\coordinate[label=right:$g-\tau-1$] (z) at (1.5+15.5, -2.5);
		\draw [very thick, in=75, out=-125] (w) to (z);

		\end{tikzpicture}
	\end{center}
	\caption{Three curves with respectively an $A_1$-attached elliptic tail, an $A_3$-attached elliptic tail and an $A_1/A_1$-attached elliptic bridge.}\label{F:attachedtails}
\end{figure}
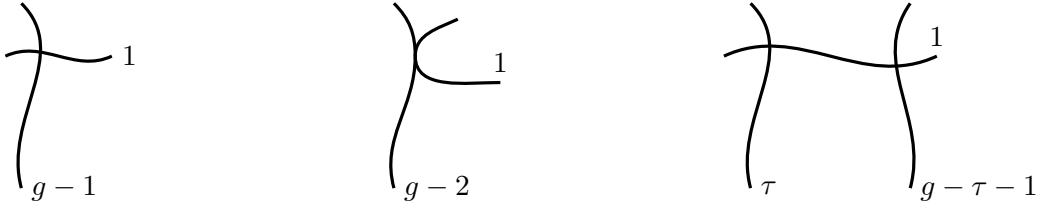

In analysing the automorphism group of the curves we will be dealing with, a central role is played by rosaries as introduced in \cite{HH2} (see also \cite[Sec. 2.5]{AFSV1}). 
  
\begin{definition}[Open and closed rosaries, see {\cite[Def. 6.1, 6.3]{HH2}, \cite[Def. 2.26]{AFSV1}}]\label{D:ros}
\noindent 
\begin{enumerate}
\item \label{D:ros1} An \emph{open rosary} of length $r$, or simply a \emph{rosary} of length $r$,   is a $2$-pointed curve $(R,q_1,q_2)$ which admits a finite, surjective morphism
$
\gamma: \bigcup_{i=1}^r (L_i,p_{2i-1},p_{2i}) \to (R,q_1,q_2)
$
with:
\begin{enumerate}
\item $(L_i,p_{2i-1},p_{2i})$ is $2$-pointed smooth rational curve  for $i=1,\ldots,r$;
\item $\gamma$ induces an open embedding of $L_i \setminus \{p_{2i-1}, p_{2i}\}$ into $R \setminus \{q_1,q_2\}$ for $i=1,\ldots,r$; 
\item $a_i:=\gamma(p_{2i})=\gamma(p_{2i+1})$ is a tacnode for $i=1,\ldots, r-1$;
\item $\gamma(p_1)=q_1$ and $\gamma(p_{2r})=q_2$.
\end{enumerate}

\item \label{D:ros2} A \emph{closed rosary} of length $r$ is a ($0$-pointed) curve $R$ which admits a finite, surjective morphism
$
\gamma: \bigcup_{i=1}^r (L_i,p_{2i-1},p_{2i}) \to R
$
such that:
\begin{enumerate}
\item $(L_i,p_{2i-1},p_{2i})$ is $2$-pointed smooth rational curve  for $i=1,\ldots,r$;
\item $\gamma$ induces an open embedding of $L_i \setminus \{p_{2i-1}, p_{2i}\}$ into $R$ for $i=1,\ldots,r$; 
\item $a_i:=\gamma(p_{2i})=\gamma(p_{2i+1})$  for $i=1,\ldots,r-1$ and $a_r:=\gamma(p_1)=\gamma(p_{2r})$ are tacnodes.
\end{enumerate}
\end{enumerate}

\end{definition}
Note that an open  rosary $(R,q_1,q_2)$ of length $r$ has arithmetic genus $g(R)=r-1$ while a closed rosary $R$ of length $r$ has arithmetic genus $g(R)=r+1$. 

An open rosary $(R,q_1,q_2)$ of length $r$ is such that  $\omega_R(q_1+q_2)$ is ample if (and only if) $r\geq 2$ (this is the reason why open rosaries of length $1$ will not play any role in the sequel). An open  rosary of length $2$ is an elliptic bridge and it is the unique elliptic bridge 
containing a tacnode; for this reason, we will also call it the \emph{tacnodal elliptic bridge.} More generally, any open rosary of even length $r$ can be regarded as an elliptic chain of length $r/2$ in which all the elliptic bridges are tacnodal.

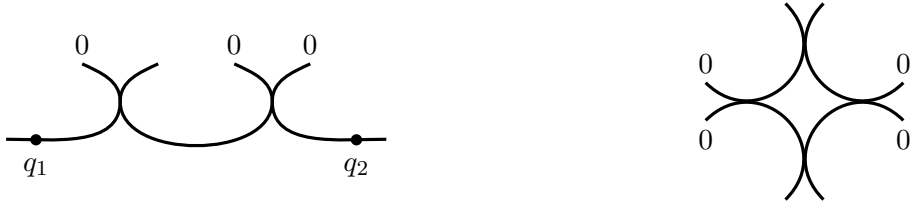
\begin{figure}[!h]
	\begin{center}
		\begin{tikzpicture}[scale=1]
		
		\coordinate (a) at (-0.5, -0.5);
		\coordinate (b) at (1, 0);
		\coordinate[label=above:0] (c) at (0.5, 0.5);
		\draw [very thick] (a)   to[out=0, in=-90]  
		node[dot, pos=0.2, label=below:$q_1$]{} 
		(b)  to[out=90, in=-25] (c);

		\coordinate (d) at (0+1.5, 0.5);
		\coordinate (e) at (1+2, 0);
		\coordinate[label=above:0] (f) at (0.5+2, 0.5); 
		\draw [very thick] (d)  to[out=205, in=90]   (b)  to[out=-90, in=-90] (e) to[out=90, in=-25] (f);

		\coordinate[label=above:0] (l) at (0+3.5, 0.5);
		\coordinate (m) at (0.5+4, -0.5);
		\draw [very thick] (l)  to[out=205, in=90]   (e)  to[out=-90, in=180] 
		node[dot, pos=0.8, label=below:$q_2$]{} 
		(m);

		\coordinate[label=above:0] (a) at (-1.3+10, 0.25);
		\coordinate (b) at (-0.25+10, 1.3);
		\draw [very thick, bend right=90, looseness=1.75] (a) to (b);
		
		\coordinate[label=above:0] (a) at (1.3+10, 0.25);
		\coordinate (b) at (0.25+10, 1.3);
		\draw [very thick, bend left=90, looseness=1.75] (a) to (b);
		
		\coordinate[label=below:0] (a) at (-1.3+10, -0.25);
		\coordinate (b) at (-0.25+10, -1.3);
		\draw [very thick, bend left=90, looseness=1.75] (a) to (b);
		
		\coordinate[label=below:0] (a) at (1.3+10, -0.25);
		\coordinate (b) at (0.25+10, -1.3);
		\draw [very thick, bend right=90, looseness=1.75] (a) to (b);

		\end{tikzpicture}	
	\end{center}
	\caption{A rosary of length 3 and a closed rosary of length 4.}\label{F:rosaries}
\end{figure}

\begin{remark}\label{R:ros-eq}
Assume $\car(k)\neq 2$.
Open rosaries and closed rosaries of even length share similar properties and they can be described as follows, following \cite[Prop. 6.5]{HH2} \footnote{Closed rosaries of odd length have different properties: they depend on one modulus and they do not admit an infinite  group of automorphism. Since we will not need them, we will refrain from giving an explicit description and direct the interested reader to  \cite[Prop. 6.5]{HH2}.} (see also \cite[Def. 2.20(2)]{AFSV1} for open rosaries of length $2$ that coincide with $7/10$-atoms). 
\noindent 
\begin{enumerate}[(i)]
\item \label{R:ros-eq1} An open rosary $(R,q_1,q_2)$ of length $r\geq 1$  can be obtained by gluing the disjoint union of $r$  projective lines $\{L_i\}_{i=1}^r$ with homogeneous coordinate $[s_i,t_i]$ and the $r-1$ affine tacnodal curves  $\Spec k[x_i,y_i]/(y_i^2-x_i^4)$ via the gluing relations
$$ x_i=\left(\frac{t_i}{s_i}, \frac{s_{i+1}}{t_{i+1}} \right)\in k\left[\frac{t_i}{s_i}\right]\times k\left[\frac{s_{i+1}}{t_{i+1}}\right], \quad y_i=\left(\left(\frac{t_i}{s_i}\right)^2, -\left(\frac{s_{i+1}}{t_{i+1}}\right)^2 \right)\in k\left[\frac{t_i}{s_i}\right]\times k\left[\frac{s_{i+1}}{t_{i+1}}\right].$$
Note that the marked points are equal to $q_1=[0,1]\in L_1$ and $q_2=[1,0]\in L_r$, while the tacnodes have coordinates (for every $1\leq i \leq r-1$)
$$a_i= [1,0] \text{ on } L_i \quad \text{ and } \quad  a_i=[0,1] \text{ on } L_{i+1}.
$$
The connected component of the automorphism group of  $(R,q_1,q_2)$ is equal to the multiplicative group $\Gm$ which acts, in the above coordinates, by 
$$
  \lambda\cdot [s_i,t_i]=[\lambda^{(-1)^{i+1}}s_i,t_i], \quad
 \lambda\cdot x_i=\lambda^{(-1)^{i}} x_i, \quad
\lambda \cdot y_i=\lambda^{2(-1)^{i}} y_i. 
 $$
 Note that the weight of the $\Gm$-action on the tangent spaces at the marked points are 
 $$\wei_{\Gm}(T_{q_1}(R))=1 \text{ and } \wei_{\Gm}(T_{q_2}(R))=(-1)^r. $$

 \item \label{R:ros-eq2} A closed rosary $R$ of even length $r\geq 1$  can be obtained by gluing the disjoint union of $r$  projective lines $\{L_i\}_{i=1}^r$ with homogeneous coordinate $[s_i,t_i]$ and the $r$ affine tacnodal curves  $\Spec k[x_i,y_i]/(y_i^2-x_i^4)$ via the gluing relations
$$ x_i=\left(\frac{t_i}{s_i}, \frac{s_{i+1}}{t_{i+1}} \right)\in k\left[\frac{t_i}{s_i}\right]\times k\left[\frac{s_{i+1}}{t_{i+1}}\right], \quad y_i=\left(\left(\frac{t_i}{s_i}\right)^2, -\left(\frac{s_{i+1}}{t_{i+1}}\right)^2 \right)\in k\left[\frac{t_i}{s_i}\right]\times k\left[\frac{s_{i+1}}{t_{i+1}}\right],$$
where we adopt the cyclic convention $L_{r+1}:=L_1$, $x_{r+1}:=x_1$ and $y_{r+1}:=y_1$. 
Note that the  tacnodes have coordinates (for every $1\leq i \leq r$)
$$a_i= [1,0] \text{ on } L_i \quad \text{�and }�\quad a_i=[0,1] \text{ on } L_{i+1}.$$
 The connected component of the automorphism group of  $R$ is equal to the multiplicative group $\Gm$ which acts, in the above coordinates, by 
 $$
 \lambda\cdot [s_i,t_i]=[\lambda^{(-1)^{i+1}}s_i,t_i], \quad \lambda\cdot x_i=\lambda^{(-1)^{i}} x_i,  \quad  \lambda \cdot y_i=\lambda^{2(-1)^{i}} y_i. 
 $$
Note that this is well-defined since $(-1)^{r+1}=(-1)^1$ because $r$ is even.
 \end{enumerate}
\end{remark}

Similarly to elliptic chains, also open rosaries can be attached in different way inside a pointed curve. However, we will need to  consider only nodal attachments, as we now define.

\begin{definition}[Attached rosaries, see {\cite[Def. 6.3]{HH2}} and {\cite[Def. 2.26]{AFSV1}}]
Let $(C, \{p_i \}_{i=1}^n)$ be an $n$-pointed curve. 
We say that $(C, \{p_i \}_{i=1}^n)$ has an \emph{$A_1/A_1$-attached rosary}  (of length $r$), or simply an \emph{attached rosary}, if there exists a finite morphism $\gamma: (R,q_1,q_2) \to (C, \{p_i \}_{i=1}^n)$ (called gluing morphism) such that:
\begin{enumerate}[(a)]
\item $(R,q_1,q_2)$ is a rosary (of length $r$);
\item  $\gamma$ induces an open embedding of $R-\{q_1,q_2\}$ into $C- \cup_{i=1}^n \{p_i \}$; 
\item $\gamma(r_i)$ is a node or a marked point (for any $i=1,2$).
\end{enumerate}

Note that we could have an $A_1/A_1$-attached rosary $\gamma: (R,q_1,q_2) \to (C, \{p_i \}_{i=1}^n)$ of length $r$ such that $\gamma(q_1)=\gamma(q_2)$: in this case we have that $C=R$ and $(g,n)=(r,0)$. 

\end{definition}

Next, we want  to define the type of a tacnode, of an $A_{k_1}/A_{k_2}$-attached elliptic chain, of an attached rosary and of a closed rosary, which will be a subset of  the set $T_{g,n}$ (see \eqref{E:Tset}).

\begin{definition}\label{D:type}[Types of tacnodes, attached elliptic chains, attached and closed rosaries]
Let $(C, \{p_i \}_{i=1}^n)$ be a $n$-pointed curve such that $C$ is Gorenstein and $\omega_C(\sum_{i=1}^n p_i)$ is ample.
\begin{enumerate}
\item \label{D:type1} Let $p \in C$ be a tacnode. We say that $p$ is of type:
\begin{itemize}
\item $\type(p):=\{\irr\}\subseteq T_{g,n} $ if the normalisation of $C$ at  $p$ is connected;
\item  $\type(p):=\{[\tau,I],[\tau+1,I]\} \subseteq T_{g,n}$ if the normalisation of $C$ at $p$ has two connected components, one of which has arithmetic genus  $\tau$ and marked points $\{p_i\}_{i \in I}$.
\end{itemize}

\item \label{D:type2}
Let $\gamma:(E,q_1,q_2)\to (C, \{p_i \}_{i=1}^n)$ be an $A_{k_1}/A_{k_2}$-attached elliptic chain of length $r\geq 1$ and with $k_1,k_2=1$ or $3$. Set 
$$\epsilon(k_1,k_2)=
\begin{cases}
0 & \text{ if } k_1=k_2=1, \\
1 & \text{ if } (k_1,k_2)=(1,3) \text{ or } (3,1),\\
2 & \text{ if } k_1=k_2=3.
\end{cases}
$$
We say that $(E,q_1,q_2)$ is of type:
\begin{itemize}
\item $\type(E,q_1,q_2):=\{[0,\{p_i\}],[1,\{p_i\}],\ldots,[2r-1+\epsilon(k_1,k_2),\{p_i\}]\} \subseteq T_{g,n}$ if either $\gamma(q_1)=p_i$ or $\gamma(q_2)=p_i$;

\item $\type(E,q_1,q_2):=\{\irr\}\subseteq T_{g,n} $ if $\gamma(q_1)$ and $\gamma(q_2)$ are singular points (either nodes or tacnodes) of $C$ and $\ov{C \setminus \gamma(E)}$ is connected (which includes also the case of a closed $A_{k_1}/A_{k_2}$-attached elliptic chain, in which case $\ov{C \setminus \gamma(E)}=\emptyset$);
\item  $\type(E,q_1,q_2):=\{[\tau,I],[\tau+1,I],\ldots,[\tau+2r-1+\epsilon(k_1,k_2),I]\} \subseteq T_{g,n}$ if  $\gamma(q_1)$ and $\gamma(q_2)$ are are singular points (either nodes or tacnodes) of $C$ and $\ov{C \setminus \gamma(E)}$ consists of two connected component, one of  which has arithmetic genus  $\tau$ with marked points $\{p_i\}_{i \in I}$.
\end{itemize}
\item \label{D:type3}
Let $\gamma:(R,q_1,q_2)\to (C, \{p_i \}_{i=1}^n)$ be an attached rosary   of length $r$. We say that $(R,q_1,q_2)$ is of type:
\begin{itemize}
\item $\type(R,q_1,q_2):=\{[0,\{p_i\}],[1,\{p_i\}],\ldots,[r-1,\{p_i\}]\} \subseteq T_{g,n}$ if either $\gamma(q_1)=p_i$ or $\gamma(q_2)=p_i$;
\item $\type(R,q_1,q_2):=\{\irr\}\subseteq T_{g,n} $ if  $\gamma(q_1)$ and $\gamma(q_2)$ are nodes of $C$ and $\ov{C \setminus \gamma(R)}$ is connected (which includes also the case where $\ov{C \setminus \gamma(R)}=\emptyset$, which can happen only if $(g, n)=(r, 0)$ and $\gamma(q_1)=\gamma(q_2))$;
\item  $\type(R,q_1,q_2):=\{[\tau,I],[\tau+1,I],\ldots,[\tau+r-1,I]\} \subseteq T_{g,n}$ if  $\gamma(q_1)$ and $\gamma(q_2)$ are nodes of $C$ and $\ov{C \setminus \gamma(R)}$ consists of two connected component, one of  which has arithmetic genus  $\tau$ with marked points $\{p_i\}_{i \in I}$.
\end{itemize}
\item \label{D:type4}
The type of a closed rosary $R$ is set to be $\type(R):=\{\irr\}$.
\end{enumerate}
\end{definition}

\begin{figure}[!h]
	\begin{center}
		\begin{tikzpicture}[scale=0.7]

		\coordinate (x) at (-2,2.5);
		\coordinate (y) at (-1.5, 0.7);
		\coordinate[label=right:{$\tau$}] (z) at (-2, -3);
		\draw [very thick] (x) to[in=90, out=-45] (y) to[in=105, out=-90]  
		node[dot, pos=0.3, label=left:$p_1$]{}
		node[dot, pos=0.5, label=left:$\vdots$]{}
		node[dot, pos=0.6]{}  
		node[dot, pos=0.8, label=left:$p_k$]{}
		(z);

		\coordinate[label=above:1] (a) at (-0.7, 1.5);
		\coordinate (c) at (2, 0);
		\draw [very thick] (a)  to[out=205, in=90]   (y)  to[out=-90, in=180] (c);

		\coordinate (w) at (1.5,2.5);
		\coordinate[label=right:{$g-\tau-2$}] (z) at (1.5, -3);
		\draw [very thick, in=75, out=-125] (w) to 
		node[dot, pos=0.6, label=right:$p_{k+1}$]{}
		node[dot, pos=0.7, label=right:$\vdots$]{}
		node[dot, pos=0.75]{}  
		node[dot, pos=0.85, label=right:$p_n$]{}
		(z);

		\coordinate (x) at (-0.5+8,2.5);
		\coordinate[label=right:{$\tau$}] (y) at (-0.5+8, -3);
		\draw [very thick, in=105, out=-45] (x) to 
		node[dot, pos=0.6, label=left:$p_1$]{}
		node[dot, pos=0.7, label=left:$\vdots$]{}
		node[dot, pos=0.75]{}  
		node[dot, pos=0.85, label=left:$p_k$]{}
		(y);

		\coordinate (a) at (-1+8, 0);
		\coordinate (b) at (2+8, 0.7);
		\coordinate[label=above:1] (c) at (1+8, 1.5);
		\draw [very thick] (a)   to[out=0, in=-90]  (b)  to[out=90, in=-25] (c);

		\coordinate (d) at (0+3+8, 1.5);
		\coordinate (e) at (2+3+8, 0.7);
		\coordinate[label=above:1] (f) at (1+3+8, 1.5); 
		\draw [very thick] (d)  to[out=205, in=90]   (b)  to[out=-90, in=-90] (e) to[out=90, in=-25] (f);

		\coordinate[label=above:1] (l) at (0+6+8, 1.5);
		\coordinate (m) at (1+7+8, 0);
		\draw [very thick] (l)  to[out=205, in=90]   (e)  to[out=-90, in=180] (m);

		\coordinate (w) at (7.5+8,2.5);
		\coordinate[label=right:{$g-\tau-5$}] (z) at (7.5+8, -3);
		\draw [very thick, in=75, out=-125] (w) to 
		node[dot, pos=0.6, label=right:$p_{k+1}$]{}
		node[dot, pos=0.7, label=right:$\vdots$]{}
		node[dot, pos=0.75]{}  
		node[dot, pos=0.85, label=right:$p_n$]{}
		(z);

		\end{tikzpicture}
	\end{center}
	\caption{A curve with an $A_3/A_1$-attached elliptic bridge of type $\{[\tau,I], [\tau+1,I], [\tau+2,I]\}$ and a curve with an $A_1/A_1$-attached elliptic chain of type $\{[\tau,I], [\tau+1,I], \ldots, [\tau+5,I]\}$, where $I=\{1,\ldots, k\}$. }\label{F:attachedbridges}
\end{figure}
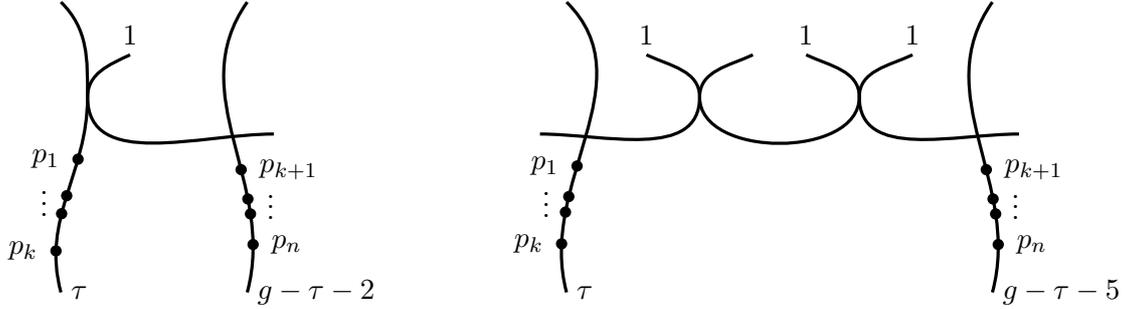

One can check that the above definitions are well posed.

\begin{remark}\label{R:tipo-ros}
Note that the type $\gamma:(R,q_1,q_2)\to (C, \{p_i \}_{i=1}^n)$ of an attached rosary  is the union of the types of all the tacnodes contained in $\gamma(R)$. And similarly for a closed rosary.
\end{remark}

We conclude this subsection by describing some isotrivial specialisations that come from the $\Gm$-action on open rosaries and closed rosaries of even lengths (see Remark \ref{R:ros-eq}) and that will play a crucial role in the sequel.
Given a (possibly $n$-pointed) curve $C$ with a special subcurve $R$, we say that $R$ specialises isotrivially to $R'$ if there exists an isotrivial specialisation of $C$ into a (possibly  $n$-pointed) curve $C'$ which is obtained by attaching $R'$ to $\ov{C\setminus R}$.

\begin{lemma}\label{L:spec-ros}
Assume that $\car(k)\neq 2$.
We have the following isotrivial specialisations:
\begin{enumerate}[(i)]
\item an $A_1/A_1$-attached elliptic chain of length $r\geq 1$ isotrivially specialises to an attached rosary of length $2r$;
\item an $A_1/A_3$-attached elliptic chain of length $r\geq 1$ isotrivially specialises to an attached rosary of length $2r+1$;
\item an $A_3/A_3$-attached elliptic chain of length $r\geq 0$ (which for $r=0$ it is  a tacnode by convention) isotrivially specialises to an attached rosary of length $2r+2$;
\item a closed $A_3/A_3$-attached elliptic chain of length $r\geq 1$ isotrivially specialises to a closed rosary of length $2r$.
\end{enumerate}
Moreover, each of the above isotrivial specialisations  preserves the type, i.e. the type of the attached elliptic chain (or of the tacnode) is the same as the type of the closed or attached rosary to which it specialises.  
\end{lemma}
\begin{proof}
See \cite[Prop. 8.3, 8.6]{HH2}
\end{proof}

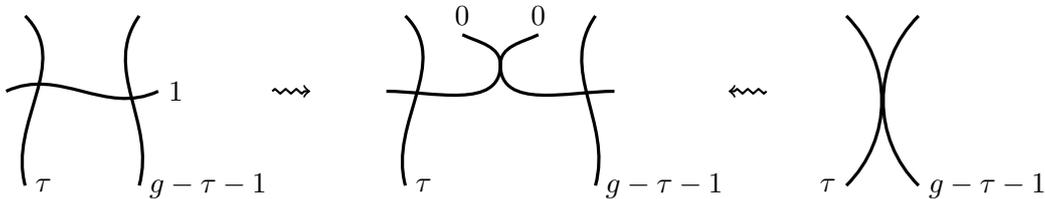
\begin{figure}[!h]
	\begin{center}
		\begin{tikzpicture}[scale=0.5]

		\coordinate (x) at (-1.5,2);
		\coordinate[label=right:{$\tau$}] (y) at (-1.5, -2.5);
		\draw [very thick, in=105, out=-45] (x) to (y);
		
		\coordinate (a) at (-2, 0);
		\coordinate[label=right:{1}] (b) at (2, 0);
		\draw [very thick, in=205, out=25] (a) to (b);
		
		\coordinate (w) at (1.5,2);
		\coordinate[label=right:{$g-\tau-1$}] (z) at (1.5, -2.5);
		\draw [very thick, in=75, out=-125] (w) to (z);

		\draw[thick,->,
		line join=round,
		decorate, decoration={
			zigzag,
			segment length=4,
			amplitude=.9,post=lineto,
			post length=2pt
		}](5,0) -- (6,0);

		\coordinate (x) at (-0.5+9,2);
		\coordinate[label=right:{$\tau$}] (y) at (-0.5+9, -2.5);
		\draw [very thick, in=105, out=-45] (x) to (y);

		\coordinate (a) at (-1+9, 0);
		\coordinate (b) at (2+9, 0.7);
		\coordinate[label=above:0] (c) at (1+9, 1.5);
		\draw [very thick] (a)   to[out=0, in=-90]  (b)  to[out=90, in=-25] (c);

		\coordinate[label=above:0] (d) at (0+3+9, 1.5);
		\coordinate (e) at (1+4+9, 0);
		\draw [very thick] (d)  to[out=205, in=90]   (b)  to[out=-90, in=180] (e);

		\coordinate (w) at (7.5+6,2);
		\coordinate[label=right:{$g-\tau-1$}] (z) at (7.5+6, -2.5);
		\draw [very thick, in=75, out=-125] (w) to (z);
		
		\draw[thick,->,
		line join=round,
		decorate, decoration={
			zigzag,
			segment length=4,
			amplitude=.9,post=lineto,
			post length=2pt
		}](18,0) -- (17,0);

		\coordinate (x) at (20.1,2);
		\coordinate[label=left:{$\tau$}] (y) at (20.1, -2.5);
		\draw [very thick, in=45, out=-45] (x) to (y);

		\coordinate (w) at (22,2);
		\coordinate[label=right:{$g-\tau-1$}] (z) at (22, -2.5);
		\draw [very thick, in=135, out=-135] (w) to (z);

		\end{tikzpicture}
	\end{center}
	\caption{An $A_1/A_1$-attached elliptic bridge and a tacnode that isotrivially specialise to an $A_1/A_1$-attached rosary of length 2. }\label{F:iso2}
\end{figure}

\begin{figure}[!h]
	\begin{center}
		\begin{tikzpicture}[scale=0.5]

		\coordinate (x) at (-2,2.5);
		\coordinate (y) at (-1.5, 0.7);
		\coordinate[label=right:{$\tau$}] (z) at (-2, -2);
		\draw [very thick] (x) to[in=90, out=-45] (y) to[in=105, out=-90] (z);

		\coordinate[label=above:1] (a) at (-0.7, 1.5);
		\coordinate (c) at (2, 0);
		\draw [very thick] (a)  to[out=205, in=90]   (y)  to[out=-90, in=180] (c);

		\coordinate (w) at (1.5,2.5);
		\coordinate[label=right:{$g-\tau-2$}] (z) at (1.5, -2);
		\draw [very thick, in=75, out=-125] (w) to (z);

		\draw[thick,->,
		line join=round,
		decorate, decoration={
			zigzag,
			segment length=4,
			amplitude=.9,post=lineto,
			post length=2pt
		}](4,0) -- (5,0);
		
		\coordinate (x) at (-0.5+8,2.5);
		\coordinate[label=right:{$\tau$}] (y) at (-0.5+8, -2);
		\draw [very thick, in=105, out=-45] (x) to (y);

		\coordinate (a) at (-1+8, 0);
		\coordinate (b) at (2+8, 0.7);
		\coordinate[label=above:0] (c) at (1+8, 1.5);
		\draw [very thick] (a)   to[out=0, in=-90]  (b)  to[out=90, in=-25] (c);

		\coordinate (d) at (0+3+8, 1.5);
		\coordinate (e) at (2+3+8, 0.7);
		\coordinate[label=above:0] (f) at (1+3+8, 1.5); 
		\draw [very thick] (d)  to[out=205, in=90]   (b)  to[out=-90, in=-90] (e) to[out=90, in=-25] (f);

		\coordinate[label=above:0] (l) at (0+6+8, 1.5);
		\coordinate (m) at (1+7+8, 0);
		\draw [very thick] (l)  to[out=205, in=90]   (e)  to[out=-90, in=180] (m);

		\coordinate (w) at (7.5+8,2.5);
		\coordinate[label=right:{$g-\tau-2$}] (z) at (7.5+8, -2);
		\draw [very thick, in=75, out=-125] (w) to (z);

		\draw[thick,->,
		line join=round,
		decorate, decoration={
			zigzag,
			segment length=4,
			amplitude=.9,post=lineto,
			post length=2pt
		}](20,0) -- (19,0);

		\coordinate (x) at (22,2.5);
		\coordinate[label=right:{$\tau$}] (y) at (22, -2);
		\draw [very thick, in=105, out=-45] (x) to (y);

		\coordinate[label=above:1] (a) at (24.3, 1.5);
		\coordinate (b) at (25, 0.7);
		\coordinate (c) at (21.5, 0);
		\draw [very thick] (a)  to[out=-25, in=90]   (b)  to[out=-90, in=0] (c);

		\coordinate (w) at (25.5,2.5);
		\coordinate[label=right:{$g-\tau-2$}] (z) at (25.5, -2);
		\draw [very thick] (w) to[in=90, out=-125] (b)  to[out=-90, in=75] (z);

		\end{tikzpicture}
	\end{center}
	\caption{Two $A_3/A_1$-attached elliptic bridges that isotrivially specialise to an $A_1/A_1$-attached rosary of length 3. }\label{F:iso3}
\end{figure}
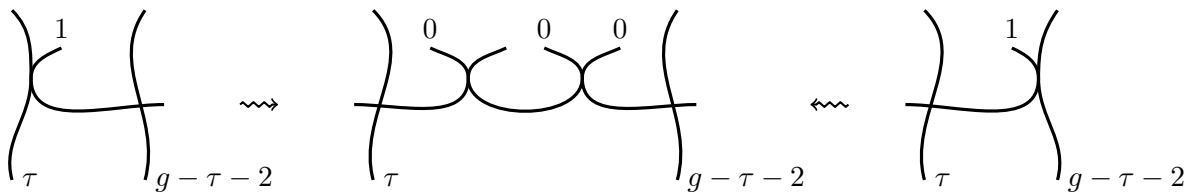

\subsection{The stacks of $T$-semistable curves and $T^+$-semistable curves}\label{S:Tstacks}

The aim of this subsection is to introduce the  stacks of $T$-semistable and $T^+$-semistable $n$-pointed curves.

Let $\U_{g,n}$ (resp. $\U_{g,n}^{lci}$) be the algebraic stack of flat, proper families of $n$-pointed curves $(\pi: \m C \to B,\{\sigma_i\}_{i=1}^n)$, where $\{\sigma_i\}_{i=1}^n$ are distinct sections that lie in the smooth locus of $\pi$, such that the geometric fibres of $\pi$ are  Gorenstein (resp. lci=locally complete intersection) curves of arithmetic genus $g$ and the line bundle $\omega_{\m C /B}(\sum \sigma_i)$ is relatively ample. Note that $\U_{g,n}$ is of finite type over $k$ since it parametrises log canonically polarized $n$-pointed curves and $\U_{g,n}^{lci}$ is an open substack of $\U_{g,n}$ which is smooth and irreducible since lci curves are unobstructed (see \cite[Cor. 3.1.13(ii)]{Ser} or  \cite[\href{http://stacks.math.columbia.edu/tag/0DZX}{Tag 0DZX}]{stacks-project}) and smoothable (see \cite[Ex. 29.0.1, Cor. 29.10]{Har}) and the condition of being lci is open (see \cite[19.3.6, 19.3.8]{EGAIV4}). For any $1\leq k\leq 3$, we denote by $\m U_{g,n}(A_k) \subset \m U_{g,n}^{lci}$ the open substack parametrizing curves with at worst $A_1,\ldots, A_k$-singularities. Note that $\U_{g,n}(A_1)=\M_{g,n}$.

Let us now recall the definition and the basic properties of the stack of pseudo-stable curves.

\begin{definition}\label{D:pseudo-stable}
\noindent 
\begin{enumerate}[(i)]
\item An $n$-pointed \emph{pseudo-stable}  curve of genus $g$ is an $n$-pointed curve $(C, \{p_i \}_{i=1}^n)$ in $\U_{g,n}(A_2)$ that does not have $A_1$-attached elliptic tails.
\item The stack of pseudo-stable $n$-pointed curves of genus $g$ is denoted by $\M_{g,n}^{\ps}$.
\end{enumerate}
\end{definition}
The stack of pseudo-stable curves $\M_{g,n}^{\ps}$ coincides with the stack $\M_{g,n}(9/11-\varepsilon)=\M_{g,n}(7/10+\varepsilon)$ from \cite[Def. 2.5 and Sec. 2.2]{AFSV1}. We have decided to adopt this terminology 
because it is a natural extension of the case $n=0$ originally  considered by Schubert \cite{Sch} (see also Hassett-Hyeon \cite{HH1} and Hyeon-Morrison \cite{HM}).

\begin{fact}[{\cite[Thm. 2.7]{AFSV1}}] \label{F:embAFSV} 
We have the following open embeddings 
$$
\M_{g,n} \hookrightarrow \M_{g,n}(9/11):=\U_{g,n}(A_2)\hookleftarrow \M_{g,n}(9/11-\varepsilon)=\M_{g,n}^{\ps}.
$$
In particular, the stack $\M_{g,n}^{\ps}$ is a smooth irreducible algebraic stack of finite type over $k$.
\end{fact}
Note that for small values of $(g,n)$, the stack $\M_{g,n}^{\ps}$ is degenerate: if $g=0$ then $\M_{0,n}^{\ps}=\M_{0,n}$, while for $(g,n)=(1,1)$ we have that  $\M_{1,1}^{\ps}=\emptyset$. 

The properties of the algebraic stack $\M_{g,n}^{\ps}$ and its relation with the stack $\M_{g,n}$  of stable curves are collected in the following Proposition.
\begin{proposition}\label{P:Mgps-DM}
Assume that $(g,n)\neq (1,1), (2,0)$. 
\begin{enumerate}[(i)]
\item \label{P:Mgps-DM0}
There is a surjective morphism $\wh{\Upsilon}:\M_{g,n}\to \M_{g,n}^{\ps}$ which, on geometric points, sends a stable $n$-pointed curve $(C,\{p_i\})$ into the pseudostable $n$-pointed curve $\wh{\Upsilon}(C,\{p_i\})$ which is obtained by replacing every ($A_1$-attached) elliptic tail of $(C,\{p_i\})$ by a cusp. 

\item \label{P:Mgps-DM1}
 $\M_{g,n}^{\ps}$ is a proper stack with finite inertia. 
\item \label{P:Mgps-DM2}
If $\car(k)\neq 2 \textrm{ or }3$, then $\M_{g,n}^{\ps}$ is a Deligne-Mumford (DM) stack. 
\end{enumerate}
\end{proposition}
\begin{proof}
This is well-known to the expert, so we only give a sketch of the proof. 

Part \eqref{P:Mgps-DM0}: the morphism of stacks $\wh{\Upsilon}$ can be constructed (using that $(g,n)\neq (2,0)$) as in \cite[Thm. 1.1]{HH1}, which deals with $n=0$ (note that the assumption  $\car(k)=0$ is not  needed in the proof of loc. cit.).

Part \eqref{P:Mgps-DM1}: the properness of $\M_{g,n}^{\ps}$ can be deduced from the properness of $\M_{g,n}$ and the existence of the surjective morphism $\wh{\Upsilon}:\M_{g,n}\to \M_{g,n}^{\ps}$, as in \cite[Prop. 2.23]{FS}.

In order to show that $\M_{g,n}^{\ps}$ has finite inertia and to prove part \eqref{P:Mgps-DM2}, consider $(C,\{p_i\})\in \M_{g,n}^{\ps}(k)$, with  $k=\ov k$, and denote by $(\wt{C},\{q_j\})$ the normalization of $C$ together with special points $\{q_j\}$ that are  either the inverse images of the points $\{p_i\}$ or  the inverse images of the singular points of $C$. It can be checked (using that $(g,n)\neq (1,1)$)  that
\begin{equation*}
\text{ Every  component of $\wt{C}$ of genus  $0$ (resp. $1$)  has at least $3$ (resp. $1$)  special points. } \tag{*}
\end{equation*} 
Since the abstract automorphism group of $(C,\{p_i\})$ injects into the abstract automorphism group of $(\wt{C},\{q_j\})$, and this latter is finite by (*), we deduce that $\M_{g,n}^{\ps}$ has finite inertia. 

Moreover, if $\car(k)\neq 2,3$, then the Lie algebra of the automorphism group scheme of $(C,\{p_i\})$, which is isomorphic to  $H^0(C,T_C(-\sum p_i))$, injects into $H^0(\wt{C},T_{\wt{C}}(-\sum q_j))$ by \cite[Proposition 2.3]{Smyth_non_reduced}, and this latter vector space is zero by (*), which shows part \eqref{P:Mgps-DM2}. 
\end{proof}

\begin{remark}
If $\car(k)$ is equal to $2$ or $3$,  \cite[Example 1]{Smyth_non_reduced} shows that a high genus cuspidal curve can have non-zero vector fields, hence $\M_{g,n}^{\ps}$ is not a Deligne-Mumford stack.

\end{remark}

If $(g,n)=(2,0)$ then the stack $\M_{g,n}^{\ps}$ does not have finite inertia and it is not separated (hence it is neither proper nor DM), as we now discuss. 

\begin{remark}\label{R:ps-g2}[Pseudostable curves with $(g,n)=(2,0)$]
In the special case $(g,n)=(2,0)$, pseudostable curves are of these types: smooth curve $C_{\emptyset}$, integral curve $C_n$ with one node and geometric genus $1$,   integral curve  $C_c$ with one cusp and geometric genus $1$, rational curve with two nodes $C_{nn}$,  rational curve $C_{nc}$ with one node and one cusp, curve $C_{nnn}$ made of two smooth rational curves meeting in three nodes, and rational curve $C_{cc}$ with two cusps  (see \cite[Fig. 1]{CTV} for a picture of all the strata of $\M_2^{\ps}$).
A pseudostable curve in $\M_{2}^{\ps}$ is a closed point if and only if it is either nodal or  it is the curve $C_{cc}$ with two cusps.
The pseudostable $C_c$ and $C_{nc}$ with only one cusp isotrivially specialises to $C_{cc}$ and hence they both contain $C_{cc}$ in their closure (see \cite[Thm. 1]{HL}). Moreover, the automorphism group of $C_{cc}$ is equal to $\Gm$. 
\end{remark}

Since $\M_{g,n}^{\ps}$ is a proper (smooth and irreducible) stack with  finite inertia, we can apply \cite{KeM} in order to deduce the following result.
\begin{cor}\label{C:coarseps}
If $(g,n)\neq (1,1), (2,0)$, then there exists a proper normal irreducible algebraic space $\MM_{g,n}^{\ps}$ together with a morphism $\phi^{\ps}:\M_{g,n}^{\ps}\to \MM_{g,n}^{\ps}$ which is a coarse moduli space.
\end{cor}

\begin{remark}\label{R:mod-g2}
If $(g,n)=(2,0)$ then it follows from \cite[Thm. 1]{HL} that $\M_2^{\ps}$ is the quotient stack of  the GIT semistable locus in  the Chow variety of tricanonical curves of genus $2$. This implies that the associated GIT quotient, which we will denote by $\MM_2^{\ps}$, is a normal irreducible projective variety that comes equipped with a morphism $\phi^{\ps}:\M_{2}^{\ps}\to \MM_{2}^{\ps}$ which is an adequate moduli space in the sense of Alper \cite{Alper2}. 
\end{remark}


We now define the stack of $T$-semistable and $T^+$-semistable curves, for  $T\subseteq T_{g,n}$ (see \eqref{E:Tset}).

\begin{definition}\label{D:ourstacks}
Fix a subset $T\subseteq T_{g,n}$.
\begin{enumerate}
\item Let $\m U_{g,n}(A_3(T))$ be the substack of $\U_{g,n}(A_3)$ parametrizing $n$-pointed curves in  $\m U_{g,n}(A_{3})$ such that all their tacnodes have type contained in $T$.
\item  In $\m U_{g,n}$ define the following constructible loci:
$$
\m B^T :=\{\mbox{Curves containing an $A_1/A_1$-attached elliptic chain of type contained in $T$}\},
$$
$$
\m T^{A_k}:= \{\mbox{Curves containing an $A_k$-attached elliptic tail}\}, \, \textrm{for } k=1 , 3 \, .
$$
\item Consider the following substacks of $\U_{g,n}(A_3(T))$:
$$
\M_{g,n}^T := \m U_{g,n}(A_3(T)) \setminus \left(\m  T^{A_1} \cup \m  T^{A_3}\right),\quad 
\M_{g,n}^{T,+} := \M_{g,n}^T \setminus \m B^T.
$$
The $n$-pointed curves in $\M_{g,n}^T$ are called \emph{$T$-semistable} while the $n$-pointed curves in $\M_{g,n}^{T,+}$ are called \emph{$T^+$-semistable}.
\end{enumerate}
\end{definition}

\begin{remark}\label{R:ourstacks}
\noindent 
The two extreme cases of the above definition are easily described. 
\begin{enumerate}[(i)]
\item If $T=\emptyset$ then 
$\M_{g,n}^T=\M_{g,n}^{T,+}=\M_{g,n}^{\ps}.$
\item If $T=T_{g,n}$ then 
$$\M_{g,n}^T=\M_{g,n}(7/10) \quad \text{ and } \quad  \M_{g,n}^{T,+} =\M_{g,n}(7/10-\epsilon),$$
with the notation of \cite[Def. 2.8]{AFSV1}.
\end{enumerate}
\end{remark}

We now want to prove that $\M_{g,n}^T$ and $\M_{g,n}^{T,+}$ are algebraic stacks of finite type over $k$.  Let us first consider the stack $\U_{g,n}(A_3(T))$.

\begin{lemma}\label{lem:A_3(T)}
The locus $\m U_{g,n}(A_3(T))$ is open in $\m U_{g,n}(A_{3})$.  In particular, $\U_{g,n}(A_3(T))$ is an algebraic  stack of finite type over $k$.
\end{lemma}
\begin{proof}
We will show that $\m U_{g,n}(A_{3}) \setminus \m U_{g,n}(A_3(T))$ is closed. Since $ \m U_{g,n}(A_3(T))$ is clearly constructible in $\U_{g,n}(A_3)$, it suffices to show that $\m U_{g,n}(A_{3}) \setminus \m U_{g,n}(A_3(T))$ is closed under specialisations. 

To this aim, consider a family $(\pi: \m C \to \D, \{\sigma_i\}_{i=1}^n)$ of curves in $\m U_{g,n}(A_3)$ (over the spectrum $\Delta=\Spec R$ of a DVR) such that $\m C_{\ov \eta}$ has a tacnode $p_{\ov \eta}$. It is enough to show that the central fibre $\m C_0$ has a tacnode $p_0$ of the same type of $p_{\ov\eta}$. Up to passing to a finite base change of $\Delta$, we can assume that there exists a section $s$ of $\pi$ such that $s(\bar\eta)=p_{\ov \eta}$. We are now going to show that $p_0:=s(0)$ is a tacnode of $\cC_0$ of the same type of $s(\ov \eta)$.

Since the $\delta$-invariant is upper semicontinuous and  the tacnodes are the unique singular points of curves in $\U_{g,n}(A_3)$ that have $\delta$-invariant equal to $2$, we get that $s(0) \in C_0$ is also a tacnode. Hence the family $\pi:\cC\to \Delta$ is equigeneric (even equisingular) along the section 
    $s$; this implies that the partial normalisation of $\cC$ along the section $s$ produces a flat and proper family $\pi':\cY\to \Delta$ of curves whose geometric fibres $\cY_0$ and $\cY_{\ov \eta}$ are the partial normalisations of, respectively, $\cC_0$ and $\cC_{\ov \eta}$ at the points, respectively, $s(0)$ ans $s(\ov \eta)$ (see I.1.3.2 of the first paper of Teissier in \cite{Tei} for $k=\bbC$ and \cite[Thm. 4.1]{CHL} for an arbitrary  field $k=\ov k$; see also \cite[Prop. 2.10]{AFSV1} for an ad hoc proof in the case of outer $A$-singularities). Since in a flat and proper morphism with reduced geometric fibres, the number of connected components of the  fibres stays constant and it coincides with the number of connected components of the geometric fibres, we see that there are two possibilities: either $\cY_0$ and $\cY_{\ov \eta}$ are both connected or they have both two connected components. In the first case, we have that  $\type(s(\ov \eta))=\irr=\type(s(0))$. In the second case, we have that $\cY$ is the disjoint union of two flat and proper families $\pi_1:\cY_1\to \Delta$ and $\pi_2:\cY_2\to \Delta$ with geometrically connected fibres of arithmetic genera equal to, respectively, $\tau\geq 0$ and $g-\tau-1\geq 0$. Moreover, since the sections $\sigma_i$ of $\pi$ do not meet the section $s$, they can be lifted uniquely to sections $\sigma_i'$ of $\pi'$ and hence there will exists $I\subseteq [n]$ such that $\{\sigma_i'\}_{i\in I}$ are sections of $\pi_1$ and $\{\sigma_i'\}_{i\in I^c}$ are sections of $\pi_2$. This  clearly implies that $\type(s(0))=\{[\tau,I],[\tau+1,I]\}=\type(s(\ov \eta))$. 
\end{proof}

This is the main result of this subsection.

\begin{theorem}\label{T:algstack}
Assume that $(g,n)\neq (2,0)$ and fix a subset  $T\subseteq T_{g,n}$.
The stack $\M_{g,n}^T$ is algebraic, smooth, irreducible  and of finite type over $k$ and we have open embeddings 
$$
\xymatrix{
\M_{g,n}^{\ps} \ar@{^{(}->}[r]^{\iota_T} & \M_{g,n}^T & \ar@{_{(}->}[l] _{\iota_T^+}  \M_{g,n}^{T,+}.
}
$$
\end{theorem}
The above result is false for $(g,n)=(2,0)$, see \cite[Rmk. 3.9]{CTV}.  
If $T=T_{g,n}$ then, using  Remark \ref{R:ourstacks}, the above Theorem reduces to \cite[Thm. 2.7]{AFSV1} for $\alpha_c=7/10$ (but one has to assume that $(g,n)\neq (2,0)$).

\begin{proof} 
Since the locus $\m T^{A_1}\cup \m T^{A_3}$ is closed in $\U_{g,n}(A_3)$ by \cite[Prop. 2.15(1)]{AFSV1}, we get that $\M_{g,n}^T$ is open in $\U_{g,n}(A_3(T))$, and hence it is open in $\U_{g,n}^{lci}$ 
by Lemma \ref{lem:A_3(T)}. Therefore, we conclude that $\M_{g,n}^T$ is a smooth and irreducible algebraic stack of finite type over $k$ because the same is true for $\U_{g,n}^{lci}$.
Moreover, since $\U_{g,n}(A_2)$ is open in $\U_{g,n}(A_3(T))$, we get that the inclusion 
$$\M_{g,n}^{\ps}=\U_{g,n}(A_2)\setminus \m T^{A_1}=\U_{g,n}(A_2)\setminus (\m T^{A_1} \cup \m T^{A_3})\subseteq \U_{g,n}(A_3(T))\setminus (\m T^{A_1} \cup \m T^{A_3})$$ 
is an open embedding. 
 It remains to prove that $\m B^T$ is closed in $\M_{g,n}^T$. Since $\m B^T$ is constructible, it is enough to prove that $\m B^T$ is closed under specialisation. 

 To this aim, consider  a family  $(\m C \to \D, \{\sigma_i\})$ of curves in $\M_{g,n}^T$ (over the spectrum $\Delta=\Spec R$ of a DVR) such that $(\m C_{\ov \eta},  \{\sigma_i(\ov \eta)\})$ contains an $A_1/A_1$-attached elliptic chain $(E,q_1,q_2)$ of length $r$ (for some $r\geq 1$) and type contained in $T$. 
 Since $(g,n)\neq (2,0)$ then $q_1$ is not attached to $q_2$. Therefore, 
following the proof of \cite[Prop. 2.15(2)]{AFSV1} \footnote{The proof of this result is correct  if one assumes that $q_1$ is not attached to $q_2$ (which is always the case if $(g,n)\neq (2,0)$), while the result is not in general true if $q_1$ is attached to $q_2$ (which always happens for $(g,n)=(2,0)$).
}
and using that $(\m C_0,  \{\sigma_i(0)\})$ is not contained in $\m T^{A_1} \cup \m  T^{A_3}$, we get that $(\m C_0,  \{\sigma_i(0)\})$ contains an $A_1/A_1$-attached elliptic chain $(E_0,t_1,t_2)$ of length $s\le r$ which is contained in the limit of $(E, q_1,q_2)$. From the explicit description of all such possible limits given in \cite[Lemma 2.14]{AFSV1}, it follows that $\type(E_0,t_1,t_2)\subseteq \type(E,q_1,q_2)$, and hence that $\type(E_0,t_1,t_2)\subseteq T$. Therefore the central fibre $(\m C_0,  \{\sigma_i(0)\})$ is contained in $\m B^T$ and we are done.      
\end{proof}

\begin{remark}\label{R:quotstack}
As  observed after  \cite[Thm. 2.7]{AFSV1}, the stack $\m U_{g,n}$ is the quotient stack of a locally closed smooth subscheme of an appropriate Hilbert scheme of some projective space $\P^N$ by $\PGL_{N+1}$. 
Hence the same is true for all the stacks $\M_{g,n}^T$ and $\M_{g,n}^{T+}$ since they are open substacks of $\m U_{g,n}$. 
\end{remark}

The containment relation among the different stacks $\M_{g,n}^T$ is determined in the Proposition that follows, whose proof is given in \cite{CTV}. Before that, we need the following 

\begin{definition}\label{def:admissible}
\noindent
\begin{enumerate}[(i)]
\item A subset  $T \subseteq T_{g,n}$  is called \emph{admissible} if $[1,\emptyset]\not\in T$ and $\irr \not \in T$ if $g=1$ 
and for every $[\tau,I]$ in $T$ then either $[\tau-1,I]$ or $[\tau+1,I]$ are in $T$.
\item Given a subset $T\subset T_{g,n}$, we obtain an admissible subset $T^{\adm}\subseteq T$ as follows:
\begin{itemize}
\item first we set $\wt T:=T-\{[1,\emptyset]\}$ if $g\geq 2$ and $\wt T:=T-\{[1,\emptyset], \irr\}$ if $g\leq 1$;
\item then we remove from $\wt T$ all the elements $[\tau,I]\in \wt T$ such that $[\tau-1,I]\not \in \wt T$ and $[\tau+1,I]\not \in \wt T$.
\end{itemize}  
\item A subset $T\subset T_{g,n}$ is said to be \emph{minimal} if $T=\{\irr\}$ and $g\geq 2$ or $T=\{[\tau, I],[\tau+1,I]\}$ (which then forces $g\geq 2$ or $g=1$ and $n\geq 2$) for some element $[\tau, I]\neq [1,\emptyset]$ of $T_{g,n}$. 
\end{enumerate}
\end{definition}
Observe that the empty set is admissible and it is the unique admissible subset if $g=0$ or if $(g,n)=(1,0)$. If $g\geq 2$ or $g=1$ and $n\geq 2$ then the minimal subsets are exactly the minimal admissible non-empty subsets of $T_{g,n}$. Moreover,  a subset $T\subset T_{g,n}$ is admissible if and only if it is the union of the minimal subsets contained in $T$.

\begin{proposition}\label{P:equaT}\cite[Prop. 3.4]{CTV}
Given two subsets $T,S\subseteq T_{g,n}$, we have that 
$$\M_{g,n}^T\subseteq \M_{g,n}^S \subset \U_{g,n}(A_3) \Longleftrightarrow  T^{\adm}\subseteq S^{\adm}.$$ 
In particular, we have that $\M_{g,n}^T= \M_{g,n}^S  \Longleftrightarrow  T^{\adm}= S^{\adm}$, in which case we also have that 
$\M_{g,n}^{T,+}=\M_{g,n}^{S,+}$. 
\end{proposition}
On the other hand, it can be shown that if $S^{\adm}\neq T^{\adm}$ then $\M_{g,n}^{T+}$ and $\M_{g,n}^{S+}$ are incomparable.

\subsection{$T$-closed and $T^+$-closed curves}

The aim of this subsection is to describe the closed points of the stacks of $T$-semistable and $T^+$-semistable curves. \footnote{In analogy with GIT, we could call these closed points $T$-polystable (resp. $T^+$-polystable) curves. We decided not to use this terminology.}

Let us start by describing the closed points of the stack of $T$-semistable curves.

\begin{definition}\label{def:Tclosed}($T$-closed curves)
Assume that $(g,n)\neq (2,0)$. 
A curve $(C, \{p_i \})$ in $\M_{g,n}^T(k)$ is  \emph{$T$-closed} if there is a decomposition $(C,\{p_i\})=K \cup (E_1,q_1^1,q_2^1) \cup \cdots \cup (E_r,q^r_1,q^r_2)$, where
\begin{enumerate}

\item  \label{def:Tclosed1} $ (E_1,q_1^1,q_2^1), \ldots, (E_r,q^r_1,q^r_2)$ are attached rosaries of length two, or equivalently $A_1$/$A_1$-attached tacnodal elliptic bridges,  of type contained in $T$.

\item \label{def:Tclosed2} $K$ does not contain tacnodes nor  $A_1$/$A_1$-attached elliptic bridges of type contained in $T$. In particular, every connected component of $K$  is a pseudo-stable curve that does not contain any 
$A_1$/$A_1$-attached elliptic bridge of type contained in $T$.
\end{enumerate}
Here $K$ (which could be empty or disconnected) is regarded as a pointed curve with marked points given by the union  of $\{p_i\}_{i=1}^n \cap K$ and of $K\cap (\overline{C \setminus K})$. We call $K$ the $T$-\emph{core} of $(C, \{p_i \}_{i=1}^n)$ and we call the decomposition $C=K \cup E_1 \cup \cdots \cup E_r$  the $T$-\emph{canonical decomposition} of $C$. 
\end{definition}

Note that a $T_{g,n}$-closed curve is the same as a $7/10$-closed curve as in \cite[Def. 2.21]{AFSV1}.

\begin{proposition}\label{prop:Tclosed}
Fix a subset $T\subset T_{g,n}$ and assume that $(g,n)\neq (2,0)$ and $\car(k)\neq 2$.
\begin{enumerate}[(i)]
\item \label{prop:Tclosed0} A curve $(C, \{p_i \})\in \M_{g,n}^T(k)$ isotrivially specialises to the $T$-closed curve $(C, \{p_i \})^{\star}$ which is the stabilisation of the $n$-pointed curve obtained from $(C, \{p_i \})$ by replacing each tacnode (necessarily of type contained in $T$) by a
rosary of length $2$ and each  $A_1$/$A_1$-attached elliptic bridges of type contained in $T$ by a rosary of length $2$.

\item \label{prop:Tclosed1} A curve $(C, \{p_i \})$ is a closed point of $\M_{g,n}^T$ if and only if $(C, \{p_i \})$ is $T$-closed.

\end{enumerate}
\end{proposition}
Note that if $T=T_{g,n}$ then the above Proposition becomes \cite[Thm. 2.22]{AFSV1} for $\alpha_c=7/10$ (or \cite[Prop. 9.7]{HH2} if furthermore $n=0$). 

The above Proposition is false for $(g,n)=(2,0)$ and $T=\{\irr\}$; see \cite[Rmk. 3.8]{CTV}  for an explicit description of all the isotrivial specialisations and of the closed points of $\M_{2}^{\irr}$.

 \begin{proof}

Part \eqref{prop:Tclosed0} follows directly from Lemma \ref{L:spec-ros}. 

Let us now prove part \eqref{prop:Tclosed1}. Part \eqref{prop:Tclosed0} implies that if $(C, \{p_i \})\in \M_{g,n}^T(k)$ is a closed point of $\M_{g,n}^T$ then it must be $T$-closed. Conversely, let $(C, \{p_i \})\in \M_{g,n}^T(k)$ be $T$-closed and consider an isotrivial specialisation $(C, \{p_i \})\leadsto (C', \{p_i' \})$ to a closed (and hence $T$-closed) point $(C', \{p_i' \})$ of $\M_{g,n}^T$. Since the connected component $\Aut(C', \{p_i' \})^o$ of the automorphism group scheme $\Aut(C', \{p_i' \})$ is a torus (see \cite[Prop. 2.6]{AFSV1}), we can apply  \cite[Thm. 1.2]{AHR} in order to deduce that 
$\M_{g,n}^T$ is \'etale locally at  $(C', \{p_i' \})$ isomorphic to $[W/\Aut(C', \{p_i' \})^o]$, for some affine variety $W$ endowed with an action of the torus $\Aut(C', \{p_i' \})^o$. We can now apply \cite[Thm. 1.4]{Kempf} in order to deduce that there exists a one parameter subgroup $\lambda:\Gm\to \Aut(C', \{p_i' \})^o$ such that $\lim_{t\to 0} \lambda(t)\cdot [(C, \{p_i \})]=[(C', \{p_i' \})]$. In other words, $(C, \{p_i \})$ is in the basin of attraction of $(C', \{p_i' \})$ with respect  $\lambda$.  

Now, mimicking the explicit analysis in \cite[Prop. 9.7]{HH2} of the basin of attraction of the one parameter subgroups of $\Aut(C', \{p_i' \})^o$ (which come from the automorphism groups of the attached length $2$ rosaries of $(C', \{p_i' \})$, as described in Remark \ref{R:ros-eq}), one deduces that $(C, \{p_i \})\cong (C', \{p_i' \})$, and hence that $(C, \{p_i \})$ is a closed point of $\M_{g,n}^T$.
\end{proof}

\begin{remark}\label{r:altra_dim}
It is possible to give an alternative proof of  Proposition \ref{prop:Tclosed}\eqref{prop:Tclosed1} (and also of  Proposition \ref{P:T+closed}\eqref{P:T+closed1} below) by proving directly,
by arguing as in \cite[Thm. 2.22]{AFSV1}, that any isotrivial specialisation of a $T$-closed (or of a $T^+$-closed) curve is actually trivial. 
\end{remark}

We now move to the description of the closed points of the stack of $T^+$-semistable curves.

\begin{definition}\label{def:T+closed}($T^+$-closed curves)
	We say that a curve $(C, \{p_i \})$ in $\M_{g,n}^{T,+}$ is \emph{$T^+$-closed} if either $C$ is a closed rosary of even length $r$ (which can happen only if $(g,n)=(r+1,0)$ and $\irr\in T$) or 
	if there is a decomposition $(C,\{p_i\})=K \cup (R_1,q_1^1,q_2^1) \cup \cdots \cup (R_r,q^r_1,q^r_2)$, where
	\begin{enumerate}
		\item  \label{def:T+closed1} $ (R_1,q_1^1,q_2^1), \ldots, (R_r,q^r_1,q^r_2)$ are attached rosaries of length $3$  (automatically of type contained in $T$);
		\item \label{def:T+closed2} $K$ does not contain $A_1$/$A_3$-attached elliptic bridges of type contained in $T$ nor closed $A_3$/$A_3$-attached elliptic chains of type contained in $T$.
	\end{enumerate}
	
	Here $K$ (which is allowed to be empty or disconnected) is regarded as a pointed curve with marked points given by the union  of $\{p_i\}_{i=1}^n \cap K$ and of $K\cap (\overline{C \setminus K})$.	
	
	We call $K$ the $T^+$-\emph{core} of $(C, \{p_i \}_{i=1}^n)$ and we call the decomposition $C=K \cup R_1 \cup \cdots \cup R_r$  the $T^+$-\emph{canonical decomposition} of $C$. Note that $K$ does not contain any $A_1$/$A_3$-attached elliptic chain of type contained in $T$
	because such a chain would necessarily contain an $A_1$/$A_3$-attached elliptic bridge of type contained in $T$, contradicting the assumptions on $K$. 
\end{definition}

\begin{proposition}\label{P:T+closed}
Fix a subset $T\subset T_{g,n}$ and assume that $(g,n)\neq (2,0)$ and $\car(k)\neq 2$.
\begin{enumerate}[(i)]
\item \label{P:T+closed0} A curve $(C, \{p_i \})\in \M_{g,n}^{T,+}(k)$ isotrivially specialises to the $T^+$-closed curve $(C, \{p_i \})^{\dagger}$ which is the stabilisation of the $n$-pointed curve obtained from $(C, \{p_i \})$ by replacing each $A_1$/$A_3$-attached elliptic bridge of type contained in $T$
by a rosary of length $3$ and each closed $A_3$/$A_3$-attached elliptic chain of length $r$ and of type contained in $T$ by a closed rosary of length $2r$. 

\item \label{P:T+closed1} A curve $(C, \{p_i \})$ is a closed point of $\M_{g,n}^{T+}$ if and only if $(C, \{p_i \})$ is $T^+$-closed.

\end{enumerate}
\end{proposition}

Note that if $T=T_{g,n}$ and $n=0$ then the above Proposition recovers \cite[Prop. 9.9]{HH2}.

\begin{proof}
Part \eqref{P:T+closed0} follows directly from Lemma \ref{L:spec-ros}. Arguing as in the proof of Proposition \ref{prop:Tclosed}\eqref{prop:Tclosed1}, part \eqref{P:T+closed1} follows from \eqref{P:T+closed0} and the fact that a $T^+$-closed curve does not lie on any basin of attraction of any other $T^+$-closed curve, a  property that is checked  as in \cite[Prop. 9.9]{HH2}. 
\end{proof}

\subsection{Line bundles on the stacks $\M_{g,n}^{\ps}$ $\M_{g,n}^T$ and $\M_{g,n}^{T+}$}\label{S:Pic-stacks}

The aim of this section is to describe the Picard group of the three stacks $\M_{g,n}^{\ps}$, $\M_{g,n}^T$ and $\M_{g,n}^{T+}$ that were introduced in \S\ref{S:Tstacks}.

From the deformation theory of lci singularities, it follows that the stack $\U_{g,n}^{lci}$ is smooth and the open substack $\M_{g,n}=\U_{g,n}(A_1)\subset \U_{g,n}^{lci}$ has complement of codimension two (which can be proved as in \cite[Prop. 3.1.5]{Ser}).
Hence, any line bundle on $\M_{g,n}$ extends uniquely to a line bundle on $\U_{g,n}^{lci}$. In particular, we can define the Hodge line bundle $\lambda$, the canonical  line bundle $K$, the cotangent line bundles $\psi_i$, the boundary line bundles  $\delta_{\irr}$ and $\delta_{i,I}$ (for every $[i,I]\in T_{g,n}-\{\irr\}$ such that $|I|\geq 2$ if $i=0$) associated to the boundary divisors $\Delta_{\irr}$ and $\Delta_{i,I}$
(for an explicit definition of the line bundles $\lambda$ and $K$ on the entire $\U_{g,n}$, see \cite[Sec. 1.1]{AFS0}.)
Following \cite{GKM}, we will set $\delta_{0,\{i\}}=-\psi_i$ so that the  divisors $\delta_{i,I}$ are defined for every $[i,I]\in T_{g,n}^*$. The total boundary line bundle, the total cotangent line bundle and the extended total boundary line bundle  are defined as follows
$$
\begin{sis}
& \delta:=\sum_{\substack{[i,I]\in T_{g,n}^*:\\ |I|\geq 2\: \text{ if } \: i=0}}\delta_{i,I}+\delta_{\irr}, \\
& \psi:=\sum_{i=1}^n \psi_i, \\
& \wh\delta=\delta-\psi=\sum_{[i,I]\in T_{g,n}^*}\delta_{i,I}+\delta_{\irr}.
\end{sis}
$$

\begin{fact}\label{F:PicU}
\noindent 
\begin{enumerate}
\item \label{F:PicU1}
The rational Picard group $\Pic(\U_{g,n}^{lci})_{\bbQ}=\Pic(\U_{g,n}^{lci})\otimes \bbQ$ of $\U_{g,n}^{lci}$ is generated by $\lambda$, $\delta_{\irr}$ and $\{\delta_{i,I}\}_{[i,I]\in T_{g,n}-\{\irr\}}$ with no relations if $g\geq 3$ and with the following relations for $g=1, 2$:
\begin{enumerate}[(i)]
\item If $g=2$ then 
$$10\lambda=\delta_{\irr}+2\delta_1 \quad \text{ where } \delta_1:=\sum_{[1,I]\in T_{2.n}^*} \delta_{1,I}.
$$
\item If $g=1$ then 
$$
\begin{sis}
& 12 \lambda=\delta_{\irr}, \\
& \delta_{\irr}+12 \sum_{\substack{[0,I]\in T_{1,n}^*: \\ p\in I}}\delta_{0,I}=0 \quad \text{ for any } 1\leq p\leq n.
\end{sis}
$$
\end{enumerate}
\item  \label{F:PicU2} [Mumford's formula] 
The canonical line bundle $K$ is equal to 
$$K=13\lambda-2\delta+\psi.$$
\end{enumerate}
\end{fact}   
Indeed the relations for $g=0$ are also known, but we do not include them in the above statement since we will not need them in this paper (see  \cite[Chap. XIX]{GAC2}).
\begin{proof}
Since $\U_{g,n}^{lci}$ is smooth then the Picard group of $\U_{g,n}^{lci}$ is equal to its divisor class group ${\rm Cl}(\U_{g,n}^{lci})$ and moreover, since $\M_{g,n}$ is an open subset of $\U_{g,n}^{lci}$ whose 
complement has codimension two, we get that ${\rm Cl}(\U_{g,n}^{lci})={\rm Cl}(\M_{g,n})=\Pic(\M_{g,n})$. Hence, both  statements follow from the analogous statements for $\M_{g,n}$: for \eqref{F:PicU1} see \cite[Chap. XIX]{GAC2}  and the references therein  if ${\rm char}(k)=0$
and \cite{Mor} if ${\rm char}(k)>0$; for \eqref{F:PicU2} see \cite[Chap. XIII, Thm. 7.15]{GAC2} (whose proof works over an arbitrary field).
\end{proof}

As a corollary of the above Fact, we can determine the rational Picard group of the stacks $\M_{g,n}^{\ps}$, $\M_{g,n}^T$ and $\M_{g,n}^{T+}$. 

\begin{cor}\label{C:Pic-ps}
We have that:
$$\begin{sis} 
& \Pic(\M_{g,n}^{\ps})_{\bbQ}=\Pic(\M_{g,n}^T)_{\bbQ}=\frac{\Pic(\U_{g,n}^{lci})_{\bbQ}}{(\delta_{1,\emptyset})}, \\
& \Pic(\M_{g,n}^{T+})_{\bbQ}=\frac{\Pic(\U_{g,n}^{lci})_{\bbQ}}{(\delta_{1,\{i\}}: \{[0,\{i\}, [1,\{i\}]\subseteq T)}. 
\end{sis}$$
\end{cor} 
\begin{proof}
Since $\M_{g.n}^{\ps}$ is an open substack of the smooth stack $\U_{g,n}^{lci}$,  its rational Picard group coincide with its rational divisor class group and it is a quotient of $\Cl(\U_{g,n}^{lci})_{\bbQ}$ by the classes of the irreducible divisors in  $\U_{g,n}^{lci}\setminus \M_{g,n}^{\ps}$, namely $\delta_{1,\emptyset}$. The argument for  $\M_{g,n}^T$ and $\M_{g,n}^{T+}$ is similar using that unique divisor in  $\U_{g,n}^{lci}\setminus \M_{g,n}^T$ is again $\Delta_{1,\emptyset}$ while the irreducible divisors in $\U_{g,n}^{lci}\setminus \M_{g,n}^{T+}$ are $\Delta_{1,\emptyset}$ and $\{\Delta_{1,\{i\}}:  \{[0,\{i\}, [1,\{i\}]\}\subseteq T\}$.
\end{proof}
From now on, we will denote the restriction of a line bundle on $\U_{g,n}^{lci}$ to one of the open substacks $\M_{g,n}^{\ps}$, $\M_{g,n}^T$ and $\M_{g,n}^{T+}$ with the same symbol.

\begin{remark}
Recently, Fringuelli and the third author have proved in \cite{FV} that, if $\car(k)\neq 2$, then $\Pic(\M_{g,n})$ is generated by the tautological line bundles subject to the relations of Fact \ref{F:PicU}\eqref{F:PicU1} (if $g\geq 1$).
Hence Fact  \ref{F:PicU}\eqref{F:PicU1} and Corollary \ref{C:Pic-ps} holds true for the integral Picard group  if $\car(k)\neq 2$.
\end{remark}

\end{section}

\begin{section}{Existence of good moduli spaces}\label{S:modspace}

In this section we want to prove that the moduli stacks of $T$-semistable and $T^+$-semistable curves admit a good moduli space in the sense of Alper \cite{Alper}.

From now on, we will assume that the characteristic is big enough as specified in the following

\begin{definition}[Characteristic big enough with respect to $T$ or $(g,n)$]\label{A:char}
Given $T\subseteq T_{g,n}$, we will say that the base field $k$ has characteristic  \emph{big enough with respect to $T$}, and we will write ${\rm char}(k)\gg T$,  if either ${\rm char}(k)=0$ or the characteristic is positive and it does  not divide the order of the \'etale  group scheme of connected components of the automorphism group schemes of every $k$-point of $\M_{g,n}^T$. Given an hyperbolic pair $(g,n)$, we will say that the base field $k$ has characteristic \emph{big enough with respect to $(g,n)$}, and we will write $\car(k)\gg (g,n)$,  if   ${\rm char}(k)\gg T_{g',n'}$ for any hyperbolic pair $(g',n')$ such that $g'\leq g$ and $n'\leq n+(g-g')$. 
\end{definition}
The relevance of the first condition ${\rm char}(k)\gg T$ for the moduli stack $\M_{g,n}^T$ is explained in the following Lemma, while the definition of $\car(k)\gg (g,n)$ is dictated by the induction used in the proof of Theorem \ref{T:goodspaces} below.

\begin{lemma}\label{L:linred}
Given $T\subseteq T_{g,n}$,   the automorphism group scheme of every $k$-point $\M_{g,n}^T$ is linearly reductive if and only if ${\rm char}(k)\gg T$.
\end{lemma}
\begin{proof}
The automorphism group scheme $\un{\Aut}(C,\{p_i\})$ of every $k$-point $(C,\{p_i\})$ of $\M_{g,n}^T$ is an extension of the  \'etale group scheme $\pi_0(\un{\Aut}(C,\{p_i\}))$ of its connected components 
by the connected component $\un{\Aut}(C,\{p_i\})^o$ containing the identity, which is a torus by  \cite[Prop. 2.6]{AFSV1}. Hence 
$\un{\Aut}(C,\{p_i\})$  is linearly reductive if and only if $\car(k)$ does not divide the order of the \'etale group scheme $\pi_0(\un{\Aut}(C,\{p_i\}))$ (see \cite[\S 2]{AOV}), i.e. if and only if $\car(k)\gg T$.
\end{proof}

\begin{remark}\label{R:linred}
For any $T\subseteq T_{g,n}$, there exists a constant $C(T)$ such that if $\car(k)\geq C(T)$ then $\car(k)\gg T$. This follows from the fact that, since $\M_{g,n}^T$ is of finite type over $k$, the order of the  finite  group schemes of connected components of $k$-points of $\M_{g,n}^T$ is bounded from above. Similarly, for any hyperbolic pair $(g,n)$ there exists a constant $C(g,n)$ such that if $\car(k)\geq C(g,n)$ then $\car(k)\gg (g,n)$.

It would be interesting to  find upper bounds for $C(T)$ and for $C(g,n)$ (for the analogue problem for $\M_{g}$, see \cite{vOV}).
\end{remark}




\begin{theorem}\label{T:goodspaces}
Let $(g,n)\neq (2,0)$ and fix a subset $T\subseteq T_{g,n}$. Assume  that $\car(k)\gg (g,n)$ as in Definition \ref{A:char}.
The algebraic stacks $\M_{g,n}^{\ps}$, $\M_{g,n}^T$ and $\M_{g,n}^{T+}$ admit good moduli spaces $\MM_{g,n}^{\ps}$, $\MMM_{g,n}^T$ and $\MMM_{g, n}^{T+}$ respectively, which are normal proper irreducible algebraic spaces over $k$. Moreover, there exists a commutative diagram 
\begin{equation}\label{E:diag-spaces}
\xymatrix{
\M_{g,n}^{\ps} \ar@{^{(}->}[r]^{\iota_T} \ar[d]^{\phi^{\ps}} & \M_{g,n}^T \ar[d]^{\phi^T} & \M_{g,n}^{T+} \ar@{_{(}->}[l]_{\iota_T^+} \ar[d]^{\phi^{T+}}\\
\MM_{g,n}^{\ps} \ar[r]^{f_T}&   \MMM_{g,n}^T  & \MMM_{g,n}^{T+} \ar[l]_{f_T^+} \\
}
\end{equation}
where the vertical maps are the natural morphisms to the good moduli spaces and the bottom horizontal morphisms $f_T$ and $f_T^+$   are proper   morphisms. 

\end{theorem}

By Remark  \ref{R:ourstacks}, the two extremal cases of the above theorem are either trivial or already known at least in characteristic zero: if $T^{\adm}=\emptyset$ (which is always the case for $g=0$ or $(g,n)=(1,1)$), then the theorem is trivially true since $\M_{g,n}^{\ps}=\M_{g,n}^T=\M_{g,n}^{T+}$; if  $T^{\adm}=T_{g,n}^{\adm}$ and $\car(k)=0$ then the theorem reduces to \cite[Thm. 1.1]{AFS2} for $\alpha_c=7/10$ (but one has to exclude the case $(g,n)=(2,0)$).

\begin{remark}\label{R:special-gn}
The above theorem degenerates (but it is still true) in the cases $(g,n)=(1,1)$ and $(g,n)=(1,2)$ while it is false for $(g,n)=(2,0)$ and $T^{\adm}\neq \emptyset$ (which implies that $T^{\adm}=\{\irr\}$), as we now discuss.
\begin{enumerate}
\item If $(g,n)=(1,1)$ then $\M_{g,n}^{\ps}=\M_{g,n}^T=\M_{g,n}^{T+}=\emptyset$ for every $T$.
\item If $(g,n)=(1,2)$ and $T^{\adm}\neq \emptyset$ (in which case it must be the case that $T^{\adm}=\{[0,\{1\}], [1,\{1\}]\}$) then all the curves in $\M_{1,2}^T$ isotrivially specialise to the tacnodal elliptic bridge so that $\MMM_{1,2}^{T}$  is equal to a point. On the other hand, the stack $\M_{1,2}^{T+}$ (and hence also its good moduli space $\MMM_{1,2}^{T+}$) is empty. 
\item If $(g,n)=(2,0)$ and $T^{\adm}=\{\irr\}$ then we do not know if the good moduli space for $\M_2^{T}=\M_2^{\irr}$ exists but certainly,  if it exists, it will not be separated, see \cite{CTV}. On the other hand, the stack $\M_2^{T+}=\M_2^{\irr+}$ is not well-defined since it is not an open substack of $\M_2^{\irr}$ (but only locally closed), see \cite[Rmk. 3.9]{CTV}.
\end{enumerate}
\end{remark}

Following the strategy of \cite{AFS2}, there are two key ingredients in the proof.  The first one is the following

\begin{proposition}\label{P:VGIT}
Assume that $(g,n)\neq (2,0)$, $\car(k)\gg T$ and fix a subset  $T\subseteq T_{g,n}$.
Then the open embeddings 
$$
\xymatrix{
\M_{g,n}^{\ps} \ar@{^{(}->}[r]^{\iota_T} & \M_{g,n}^T & \ar@{_{(}->}[l] _{\iota_T^+}  \M_{g,n}^{T,+}.
}
$$
arise from local VGIT with respect to the line bundle $\delta-\psi$ on $\M_{g,n}^T$.
\end{proposition} 
We refer to \cite[Def. 3.14]{AFSV1} for the definition of when two open substacks of a given algebraic stack $\cX$ arise from local VGIT at some (or any) closed point $x\in \cX(k)$  with respect to a line bundle $\calL$ on $\cX$.
\begin{proof}
The proof of \cite[Thm. 3.17]{AFSV1} (or its expanded version in \cite[Thm. 3.11]{ASV}) carries through, as we now briefly indicate.  

Let $(C,\{p_i\})$ be a closed point of $\M_{g,n}^T$, i.e. $(C,\{p_i\})$ is a $T$-closed curve by Proposition \ref{prop:Tclosed}\eqref{prop:Tclosed1}.
As explained in Remark \ref{R:ros-eq}, for every $A_1/A_1$-attached tacnodal elliptic bridge $(E_i, q_i^1, q_i^2)$ of $(C,\{p_i\})$, we have a one parameter subgroup $\rho_i:\Aut((E_i, q_i^1, q_i^2))\cong \Gm\hookrightarrow \Aut(C,\{p_i\})$, and these one parameter subgroups freely generate the connected component $\Aut(C,\{p_i\})^o$ of $\Aut(C,\{p_i\})$ containing the identity.
Arguing as in \cite[Prop. 3.25]{AFSV1}, the character $\chi_{\delta-\psi}$ of $\Aut(C,\{p_i\})^o$ induced by the line bundle $\delta-\psi$ is equal to a positive power of the diagonal character
$$
\begin{aligned}
\chi_*: \Aut(C,\{p_i\})^o\cong \prod_{i=1}^r \Aut((E_i, q_i^1, q_i^2)) &\longrightarrow \Gm, \\ 
(t_1,\ldots, t_r) & \mapsto t_1\ldots t_r. 
\end{aligned}
$$

We can choose coordinates on the first order deformation space $T^1:=T^1((C,\{p_i\})$ of $(C,\{p_i\})$ as in  \cite[Prop. 3.22]{ASV}, in such a way that the  action of $\Aut(C,\{p_i\})^o$ is diagonal. 
Arguing as in \cite[Prop. 3.26]{AFSV1} (see also \cite[Prop. 3.29]{ASV} for more details), the VGIT ideals $I_{\chi_*}^+$ and $I_{\chi_*}^-$ of the action of $\Aut(C,\{p_i\})^o$ on $A:=k[T^1]$ with respect to the character $\chi_*$ 
are such that $V(I_{\chi_*}^+)$ is the locus of deformations of $(C,\{p_i\})$ that keep the tacnode of some  $(E_i, q_i^1, q_i^2)$ and smoothen out the attaching nodes, while $V(I_{\chi_*}^-)$ is the locus of deformations of $(C,\{p_i\})$ that 
smoothen out  the tacnode of some  $(E_i, q_i^1, q_i^2)$ and keep the attaching nodes. Remark that, given a $A_1/A_1$-attached tacnodal elliptic bridge of type $S \subset T$, if we smooth out the attaching nodes we are left with a tacnode of type $S$, and if we smooth out the tacnode we are left with an elliptic bridge of type $S$.  Therefore, the images $I_{\chi_*}^+\wh{A}$ and $I_{\chi_*}^-\wh{A}$ of the VGIT ideals in the completion $\wh{A}:=k[[T^1]]$ are equal to the ideals induced by, respectively, the reduced closed substacks  $\M_{g,n}^T\setminus \M_{g,n}^{\ps}$ and $\M_{g,n}^T\setminus \M_{g,n}^{T,+}$ of $\M_{g,n}^T$. 
By \cite[Prop. 3.15]{AFSV1}, this is enough to conclude that the open embeddings 
$$
\xymatrix{
\M_{g,n}^{\ps} \ar@{^{(}->}[r]^{\iota_T} & \M_{g,n}^T & \ar@{_{(}->}[l] _{\iota_T^+}  \M_{g,n}^{T,+}.
}
$$
arise from local VGIT with respect to the line bundle $\delta-\psi$ on $\M_{g,n}^T$.
\end{proof} 

The second key point is the proof that the complements of $\M_{g,n}^{\ps}$ and of $\M_{g,n}^{T+}$ in $\M_{g,n}^T$ admit good moduli spaces. Let us introduce a notation for these complements.

\begin{definition}\label{D:lociZ}
Consider the following closed substacks (with reduced structure) in $\M_{g,n}^T$:
\begin{equation*}\label{E;def-Z}
\m Z^-_T= \M_{g,n}^T \setminus \M_{g,n}^{\ps} \quad \text{ and, for  } (g,n)\neq (2,0), \quad  \m Z^+_T = \M_{g,n}^T \setminus \M_{g,n}^{T,+}.
\end{equation*}
\end{definition}

\noindent Observe that these loci have the following explicit description:
$$ 
\m Z^-_T= \mbox{\{Curves in $\M_{g,n}^T$ with at least one tacnode (of type contained in $T$)}  \},
$$
$$
\m Z^+_T =\mbox{\{Curves in $\M_{g,n}^T$ with at least one $A_1/A_1$-attached elliptic chain of type contained in $T$\}}.
$$

We first focus on the existence of a good moduli space for the stack $\m Z^-_T$.

\begin{proposition}\label{prop:-}
Fix $T\subseteq T_{g,n}$ and assume that $\car(k)\gg T$.  If $\M^{T'}_{g',n'}$ admits a proper good moduli space  for all $T'\subseteq T_{g',n'}$  with either $g'<g$ and $1\leq n'\leq n+1$ or $(g',n')=(g-2,n+2)$,  then $\m Z^-_T\subset \M_{g,n}^T$ admits a proper good moduli space.
\end{proposition}

Note that $\m Z_T^-$ coincides with the stack $\ov{\mathcal S}_{g,n}(7/10)$ of \cite[Section 4]{AFS2} in the case where $T^{\adm}=T_{g,n}^{\adm}$. Hence, the above Proposition generalises \cite[Prop. 4.10]{AFS2} for $\alpha_c=7/10$. At the other extreme, if $T^{\adm}=\emptyset$  then $\m Z_T^-=\emptyset$ by Remark \ref{R:ourstacks} and the result is trivial.
Moreover, if $(g,n)=(1,2)$ and $T^{\adm}\neq \emptyset$ then $\m Z_T^-=\ov{\mathcal S}_{1,2}(7/10)\cong B\bbG_m$ because it consists of one point, namely the tacnodal elliptic bridge, which has automorphism group $\Gm$
(see \cite[Lemma 4.3]{AFS2}) and the good moduli space is just a point. 

The strategy of proof of Proposition \ref{prop:-} is similar to the one of loc. cit. and it consists in finding a finite cover of $\m Z_T^-$ which is a stacky projective bundle over suitable stacks  $\M^{T'}_{g',n'}$ (as in the statement of Proposition \ref{prop:-})  and then concluding by applying the criterion contained in the following proposition, which generalises \cite[Prop. 1.4]{AFS2} from $\car(k)=0$ to arbitrary characteristic.

\begin{proposition}\label{P:fincov}
Let $f:\cX\to \cY$ be a morphism of algebraic stacks of finite type over an algebraically closed field $k$ (of arbitrary characteristic). Suppose that:
\begin{enumerate}[(i)]
\item \label{P:fincov1} the morphism $f:\cX\to \cY$ is finite and surjective;
\item \label{P:fincov2} there exists a good moduli space with $\cX\to X$ with $X$ separated; 
\item \label{P:fincov3} the algebraic stack $\cY$ is a \emph{global quotient stack}, i.e. it is isomorphic to   $[Z/G]$ for an algebraic space $Z$ of finite type over $k$ and a reductive algebraic $k$-group $G$, and it admits local quotient presentations (which implies that  the stabilisers of  its closed $k$-points are linearly reductive).

\end{enumerate}
Then  there is a good moduli space $\cY\to Y$ with $Y$ separated. Moreover, if $X$ is proper, so is $Y$.
\end{proposition}

\begin{proof} 
The proof of \cite[Prop. 1.4]{AFS2}  works verbatim provided that one replaces \cite[Lemma 3.6]{AFS2} with the Lemma below. 
\end{proof}

\begin{lemma}\label{L:Cheva}[Chevalley theorem for stacks]
Consider a commutative diagram
$$\cX \to \cY  \to X$$ 
of algebraic stacks of finite type over an algebraically closed field k (of arbitrary characteristic), where $X$ is an algebraic space. Suppose that:
\begin{enumerate}[(i)]
\item  the morphism $\cX\to \cY$ is finite and surjective;
\item  the morphism $\cX \to X$ is cohomologically affine;
\item  the algebraic stack $\cY$ is a global quotient stack such that the stabilisers of its closed $k$-points are linearly reductive. 
\end{enumerate}
Then $\cY \to X$ is cohomologically affine.
\end{lemma}
\begin{proof}
The first part of the  proof follows \cite[Lemma 3.6]{AFS2}. Write $\cY=[V/G]$ for an algebraic space $V$ of finite type over $k$ and a reductive algebraic $k$-group $G$. Since
$\cX \to \cY$ is affine, $\cX$ is the quotient stack $\cX = [U/G]$, where $U = V\times_{\cY} \cX$. Since $U \to \cX$ is affine and $\cX \to X$ is cohomologically affine, the morphism $U \to X$ is affine by Serre's criterion.
The morphism $U \to  V$ is finite and surjective and therefore, by Chevalley's theorem, we can conclude that $p:V \to X$ is affine.

Since the affine morphism $p:V\to X$ is $G$-invariant and $G$ is reductive, we can factor $p$ as 
$$p: V\to [V/G]\stackrel{\phi}{\longrightarrow} V/G:=\Spec_{\cO_X}�p_*(\cO_V)^G \to X.$$
Since the morphism $V/G \to X$ is affine (and hence cohomologically affine), it is enough to show that $\phi$ is cohomologically affine (and indeed we will show that it is a good moduli space).

Let $v$ be a $k$-point of $V$ with a closed $G$-orbit, i.e. a closed $k$-point of $\cY=[V/G]$.
Luna's slice theorem (in the generalized form \cite[Thm. 2.1]{AHR}) implies that we can find a $G_v$-invariant locally closed algebraic subspace $W_v\subset V$, containing $v$ and affine over $X$, such that the morphism $f_v:W_v/G_v\to V/G$ is \'etale and the following diagram 
$$\xymatrix{
[W_v/G_v] \ar[r] \ar[d]^{\phi_v} & [V/G] \ar[d]^{\phi}\\
W_v/G_v \ar[r]^{f_v} & V/G   
}$$
is Cartesian. Now, since $G_v$ is linearly reductive, the morphism $\phi_v$ is a good moduli space by \cite[Thm. 13.2]{Alper}. Iterating this argument for all $k$-points of $V$ with a closed $G$-orbit and using the quasi-compactness of $V/G$, we obtain an \'etale cover $f:Z\to V/G$ such that pull-back of $\phi$ via $f$ is a good moduli space. This implies that also $\phi$ is a good moduli space by \cite[Prop. 4.7(ii)]{Alper}, and we are done.
\end{proof}

\begin{remark}\label{R:comm}
\noindent 
\begin{enumerate}[(i)]
\item \label{R:comm1} The assumption \eqref{P:fincov3} in  Proposition \ref{P:fincov} is satisfied for quotient stacks of the form $[U/G]$, where $U$ is a normal and separated scheme of finite type over $k$  and $G$ is a smooth linear algebraic $k$-group, having the property that the stabilisers of the closed $k$-points  are linearly reductive. See \cite[Prop. 2.3]{AFS2} and the references therein.
\item \label{R:comm2}
If $\car(k)=0$ then the condition of the stabilisers in Lemma \ref{L:Cheva} can be removed (indeed, it follows from the first two assumptions on the Lemma), as in  \cite[Lemma 3.6]{AFS2}. However, if $\car(k)=p>0$ then the condition cannot be dropped as the following example (suggested to us by Maksym Fedorchuk) shows: 
$$\cX=\Spec k\to \cY=[\Spec k /(\bbZ/p\bbZ)]\to X=\Spec k.$$ 
\end{enumerate}
\end{remark}

Now, before entering into the proof of Proposition \ref{prop:-}, we will need to review some constructions from \cite[Sec. 4.2]{AFS2}, adapted to our setting and notation.

The \emph{sprouting} stack $\Spr$ is the algebraic stack (see \cite[Def. 4.6]{AFS2}) consisting of flat and proper families of curves $(\m C\to S, \{\sigma_i\}_{i=1}^{n+1})$ with $n+1$-sections $\sigma_i$ such that 
\begin{itemize}
\item the  family $(\m C\to S, \{\sigma_i\}_{i=1}^{n})$ is a $S$-point of $\U_{g,n}(A_3)$;
\item  $\m C$ has a tacnodal singularity along $\sigma_{n+1}$.
\end{itemize} 
Note that the type of the tacnode remains constant along $\sigma_{n+1}$ (see the proof of Lemma \ref{lem:A_3(T)}), so that $\Spr$ will be the disjoint union of closed and open substacks where the type of $\sigma_{n+1}$ is fixed. We will denote by $\Spr^{\irr}$ (resp. $\Spr^{0,\{j\}}$, resp. $ \Spr^{h, M}$) the closed and open substack of $\Spr$ where the tacnodal section $\sigma_{n+1}$ has type $\{\irr\}$ (resp. $\{[0,\{j\}], [1,\{j \}]\}$, resp. $\{[h,M], [h+1,M]\}$ with $[h,M]\neq [0,\{j\}]$ for any $j\in [n]$).

There is an obvious  forgetful morphism 
$$\cF: \Spr\to \U_{g,n}(A_3)$$
given by forgetting the last section $\sigma_{n+1}$. The morphism $\cF$ is  finite (and representable) by \cite[Prop. 4.7]{AFS2}. The restriction of $\cF$ to $\Spr^{\irr}$ (resp. $\Spr^{0,\{j\}}$, resp. $ \Spr^{h, M}$) will be denoted by 
$\cF_{\irr}$ (resp. $\cF_{0,\{j\}}$, resp. $\cF_{h,M}$).

As explained in \cite[Sec. 4.2]{AFS2}, given a family $(\m C\to S, \{\sigma_i\}_{i=1}^{n+1})\in \Spr(S)$, we can normalise along the  section $\sigma_{n+1}$ and then stabilise in order to get a new family $(\m C^s\to S, \{\sigma_i^s\}_{i=1}^{n+l})$ (with $l=0$ or $2$).
The number of connected components of $\m C^s\to S$, their genera and number of marked points, and the number $l$ is determined by the type of tacnodal section $\sigma_{n+1}$. We can distinguish the following three cases.

\begin{itemize}
\item If the tacnodal section $\sigma_{n+1}$ is of type $\{\irr\}$ then $\cC^s\to S$ is connected, hence we get a morphism
$$
\begin{aligned}
\cN_{\irr}: \Spr^{\irr} &\longrightarrow \U_{g-2,n+2}(A_3), \\
(\m C\to S, \{\sigma_i\}_{i=1}^{n+1}) & \mapsto (\m C^s\to S, \{\sigma_i^s\}_{i=1}^{n+2}),\\
\end{aligned}
$$
where the first $n$ sections $\sigma_i^s$ are the images of the first $n$ sections $\sigma_i$ and the last sections $\{\sigma_{n+1}^s, \sigma_{n+l}^s\}$  are the  two inverse images of $\sigma_{n+1}$ under the normalisation along $\sigma_{n+1}$.

\item If the tacnodal section has type equal to $\{[0,\{j\}], [1,\{j \}]\}$ then the normalisation of $\cC\to S$ will have two connected components, one of which is a family of genus $g-1$ curves with $n$ marked points, and the other one is a family of genus $0$ curves with $2$ marked points. When we stabilise, the second component gets contracted and hence we end up with a morphism 
$$\begin{aligned}
\cN_{0, \{j\}}: \Spr^{0,\{j\}} &\longrightarrow \U_{g-1,n}(A_3), \\
(\m C\to S, \{\sigma_i\}_{i=1}^{n+1}) & \mapsto (\m C^s\to S, \{\sigma_i^s\}_{i=1}^{n}),
\end{aligned}$$
where the first $n-1$ sections $\sigma_i^s$ are the images of the sections $\{\sigma_i\}_{i\neq j, n+1}$ and the last section $\sigma_{n}^s$ is  one of the two inverse images of $\sigma_{n+1}$ under the normalisation along $\sigma_{n+1}$.

\item If the tacnodal section has type equal to $\{[h,M], [h+1,M]\}$ with $[h,M]\neq [0,\{j\}]$ for any $j\in [n]$, then the normalisation of $\cC\to S$ will have two connected components, $\cC_1\to S$ of genus $h$ curves and with $|M|+1$ marked points, and $\cC_2\to S$ of genus $g-h-1$ and with $|M^c|+1$ marked points.
Hence, after stabilising,  we get a morphism  
$$\begin{aligned}
\cN_{h, M}: \Spr^{h, M} & \longrightarrow \U_{h,|M|+1}(A_3)\times \U_{g-h-1, |M^c|+1}(A_3), \\
(\m C\to S, \{\sigma_i\}_{i=1}^{n+1}) & \mapsto \left( (\cC_1^s\to S, \{\sigma_i^s\}_{i\in M}, \sigma^s_{n+1}), (\cC_2^s\to S, \{\sigma_i^s\}_{i\in M^c}, \sigma^s_{n+2})\right),
\end{aligned}$$
where the sections $\{\sigma_i^s\}_{i\in M\cup M^c}$ are the images of the first $n$ sections $\sigma_i$ and the last sections $\{\sigma_{n+1}^s, \sigma_{n+l}^s\}$  are the images of the two inverse images of $\sigma_{n+1}$ under the normalisation along $\sigma_{n+1}$.

\end{itemize}

By \cite[Prop. 4.9]{AFS2}, the above maps $\cN_{\irr}$, $\cN_{0,\{j\}}$ and $\cN_{h, M}$ are stacky projective bundles. 
For later use, observe that the codomain of these stacky projective bundles are always stacks of pointed curves with at least one marked point. This is clear for $\cN_{\irr}$ and $\cN_{h, M}$, and for 
$\cN_{0,\{j\}}$ it follows from the fact that the morphism $\cN_{0, \{j\}}: \Spr^{0,\{j\}} \to \U_{g-1,n}(A_3)$ is defined  only if $\{[0,\{j\}], [1,\{j \}]\}\subset T_{g,n}$ which implies that $n\geq 1$.

We now study the compatibility of the maps $\cN_{\irr}$, $\cN_{0,\{j\}}$ and $\cN_{h, M}$ and of $\cF_{\irr}$, $\cF_{0,\{j\}}$ and $\cF_{h,M}$ with the open substacks of $T$-semistable curves. 

\begin{lemma}\label{L:spr-T}
Let $T\subseteq T_{g,n}$. Then the preimage of $\M_{g,n}^T$ via the maps $\cF_{\irr}$, $\cF_{0,\{j\}}$ and $\cF_{h,M}$ are computed as follows.
\begin{enumerate}[(i)]
\item \label{L:spr-T1} $\cF_{\irr}^{-1}(\M_{g,n}^T)=
\begin{cases}
\emptyset & \text{ if } \irr \not\in T,\\
(\cN^{-1}_{\irr})\left(\M_{g-2,n+2}^{T_{g-2,n+2}}\right) & \text{ if } \irr\in T. 
\end{cases}$
\item \label{L:spr-T2} $\cF_{0,\{j\}}^{-1}(\M_{g,n}^T)=
\begin{cases}
\emptyset & \text{ if } \{[0,\{j\}], [1,\{j\}]\} \not\subset T,\\
(\cN^{-1}_{0,\{j\}})\left(\M_{g-1,n}^{\wh{T}}\right) & \text{ if } \{[0,\{j\}], [1,\{j\}]\} \subset T,
\end{cases}$\\
where $\wh{T}$ is the subset of $T_{g-1,n}$ defined by 
$$\begin{sis}
\irr \in \wh{T} & \Leftrightarrow \irr\in T, \\
[\tau, I]\in \wh{T} & \Leftrightarrow 
\begin{cases}
[\tau, I]\in T & \text{ if } n+1\not\in I, \\
[g-1-\tau, [n+1]-\{I\}] & \text{ if } n+1\in I.
\end{cases} 
\end{sis}$$
 \item \label{L:spr-T3} $\cF_{h,M}^{-1}(\M_{g,n}^T)=
\begin{cases}
\emptyset & \text{ if } \{[h,M], [h+1,M]\} \not\subset T,\\
(\cN^{-1}_{h, M})\left(\M_{h,|M|+1}^{\wt{T}_{h,M}}\times \M_{g-1-h,|M^c|+1}^{\wt{T}_{g-h-1,M^c}}\right) &  \text{ if } \{[h,M], [h+1,M]\} \subset T,\\ 
\end{cases}$\\
where $\wt{T}_{h,M}$ is the subset of $T_{h,|M|+1}$ defined by 
$$\begin{sis}
\irr \in \wt{T}_{h,M} & \Leftrightarrow \irr\in T, \\
[\tau, I]\in \wt{T}_{h,M} & \Leftrightarrow 
\begin{cases}
[\tau, I]\in T & \text{ if } |M|+1\not\in I, \\
[h-\tau, [|M|+1]-\{I\}\}] & \text{ if } |M|+1\in I.
\end{cases} 
\end{sis}$$
with the convention that $[|M|]=[|M|+1]-\{|M|+1\}$ is identified with the subset $M\subset [n]$, and where $\wt{T}_{g-h-1,M^c}\subseteq T_{g-h-1,|M^c|+1}$ is defined similarly by replacing $h$ with $g-h-1$ and $M$ with $M^c$.
\end{enumerate}
\end{lemma}
\begin{proof}
Recall that $\M_{g,n}^T$ is the open substack whose $k$-points are $n$-pointed curves $(C,\{p_i\})\in \U_{g,n}(A_3)$ that do not have $A_1$ or $A_3$-attached  elliptic chains and whose tacnodes have type contained in $T$. 
Hence we can argue with families of curves over $k$, i.e. with $n$-pointed curves.

Let us first prove \eqref{L:spr-T1}. First of all, since for any $(C, \{p_i\}_{i=1}^{n+1})\in \Spr^{\irr}(k)$ the $n$-pointed curve  $\cF_{\irr}(C, \{p_i\}_{i=1}^{n+1})=(C, \{p_i\}_{i=1}^{n})\in \U_{g,n}(A_3)(k)$ will have a tacnode of type $\{\irr\}$ in $p_{n+1}$, we deduce that 
$\cF_{\irr}^{-1}(\M_{g,n}^T)=\emptyset$ if $\irr \not\in T$. We can therefore assume that $\irr \in T$. Note that $\cF_{\irr}(C, \{p_i\}_{i=1}^{n+1})=(C, \{p_i\}_{i=1}^{n})$ will have an 
$A_1$ or $A_3$-attached elliptic chain if and only if the same property holds for $\cN_{\irr}(C, \{p_i\}_{i=1}^{n+1})=(C^s, \{p_i^s\}_{i=1}^{n+2})$. Hence the result follows since every tacnode of $(C^s, \{p_i^s\}_{i=1}^{n+2})$ becomes a tacnode of type $\{\irr\}$ when seen in $(C, \{p_i\}_{i=1}^{n})$.

Let us  now prove \eqref{L:spr-T2}. Since for any $(C, \{p_i\}_{i=1}^{n+1})\in \Spr^{0, \{j\}}(k)$ the $n$-pointed curve  $\cF_{0, \{j\}}(C, \{p_i\}_{i=1}^{n+1})=(C, \{p_i\}_{i=1}^{n})\in \U_{g,n}(A_3)(k)$ will have a tacnode of type $\{[0,\{j\}], [1,\{j\}]\}$ in $p_{n+1}$, we deduce that 
$\cF_{0, \{j\}}^{-1}(\M_{g,n}^T)=\emptyset$ if $\{[0,\{j\}], [1,\{j\}]\} \not\subseteq T$. We can therefore assume that $\{[0,\{j\}], [1,\{j\}]\} \subseteq T$. Note that $\cF_{0,\{j\}}(C, \{p_i\}_{i=1}^{n+1})=(C, \{p_i\}_{i=1}^{n})\in \U_{g,n}(A_3)(k)$ will have an $A_1$ or $A_3$-attached elliptic chain if and only if the same property holds for $\cN_{0, \{j\}}(C, \{p_i\}_{i=1}^{n+1})=(C^s, \{p_i^s\}_{i=1}^{n+1})$. Hence the result follows since every tacnode  of $(C^s, \{p_i^s\}_{i=1}^{n+1})$ of type $\{\irr\}$ remains of type $\{\irr\}$  when seen in $(C, \{p_i\}_{i=1}^{n})$, while every tacnode of $(C^s, \{p_i^s\}_{i=1}^{n+1})$ of type $\{[\tau,I], [\tau+1,I]\}$ becomes, when seen in $(C, \{p_i\}_{i=1}^{n})$,  of type $\{[\tau,I], [\tau+1,I]\}$ if $n+1\not\in I$ and of type $\{[g-2-\tau,[n+1]-\{I\}], [g-1-\tau,[n+1]-\{I\}]\}$ if $n+1 \in I$.

Let us finally prove \eqref{L:spr-T3}. First of all, since for any $(C, \{p_i\}_{i=1}^{n+1})\in \Spr^{h, M}(k)$ the $n$-pointed curve  $\cF_{h,M}(C, \{p_i\}_{i=1}^{n+1})=(C, \{p_i\}_{i=1}^{n})\in \U_{g,n}(A_3)(k)$ will have a tacnode of type $\{[h, M], [h+1, M]\}$ in $p_{n+1}$, we deduce that 
$\cF_{h, M}^{-1}(\M_{g,n}^T)=\emptyset$ if $\{[h, M], [h+1, M]\} \not\subset T$. We can therefore assume that $\{[h, M], [h+1, M]\} \subset T$. Note that $\cF_{h, M}(C, \{p_i\}_{i=1}^{n+1})=(C, \{p_i\}_{i=1}^{n})\in \U_{g,n}(A_3)(k)$ will not have an $A_1$ or $A_3$-attached elliptic chain if and only if the same property holds for both  $(C_1^s, \{p_i^s\}_{i\in M},\{p_{n+1}\})\in \U_{h,|M|+1}(A_3)(k)$ and   $(C_2^s, \{p_i^s\}_{i\in M},\{p_{n+2}\})\in \U_{g-h-1,|M^c|+1}(A_3)(k)$.  Hence it remains to determine to type of the tacnodes of $(C_1^s, \{p_i^s\}_{i\in M},\{p_{n+1}\})$ and   $(C_2^s, \{p_i^s\}_{i\in M},\{p_{n+2}\})$ when considered in $(C, \{p_i\}_{i=1}^{n})$. 
We will only examine the tacnodes of $(C_1^s, \{p_i^s\}_{i\in M},\{p_{n+1}\})$, the other case being analogous. 
A tacnode of $(C_1^s, \{p_i^s\}_{i\in M},\{p_{n+1}\})$ of type $\{\irr\}$ remains of type $\{\irr\}$  when seen in $(C, \{p_i\}_{i=1}^{n})$, 
while a tacnode of $(C^s, \{p_i^s\}_{i=1}^{n+1})$ of type $\{[\tau,I], [\tau+1,I]\}$ becomes, when seen in $(C, \{p_i\}_{i=1}^{n})$,  of type $\{[\tau,I], [\tau+1,I]\}$ if $|M|+1\not\in I$ and of type $\{[h-\tau-1, [|M|+1]-\{I\}], [h-\tau, [|M|+1]-\{I\}]\}$ if $|M|+1\in I$. This  implies the result. 
\end{proof}

Using the above Lemma, we can prove  the existence of the proper good moduli space for $\m Z_T^-$. 

\begin{proof}[Proof of Proposition \ref{prop:-}]
Consider the open substack of $\Spr$:
$$E_T:=\cF_{\irr}^{-1}(\M_{g,n}^T)\coprod_{\{[0,\{j\}]\in  T_{g,n}} \cF_{0,\{j\}}^{-1}(\M_{g,n}^T)  \coprod_{\stackrel{[h,M]\in  T_{g,n}:}{0\leq h \leq g-1, [h,M]\neq [0,\{j\}]}} \cF_{h,M}^{-1}(\M_{g,n}^T).$$
The morphism $\cF$ restricted to $E_T$ gives rise to a morphism 
$$\cF_T=\cF_{|E_T}:E_T\to \M_{g,n}^T,$$
which is finite by \cite[Prop. 4.7]{AFS2}. By construction, the image of $\cF_T$ is the locus of $\M_{g,n}^T$ having at least one tacnode, i.e. exactly $\m Z_T^-$.

Observe that the algebraic stack $\m Z_T^-$, being a closed substack of $\M_{g,n}^T$, is a global quotient stack of a normal variety by Remark \ref{R:quotstack} and it has linearly reductive stabilisers by Lemma \ref{L:linred} and our assumption on $\car(k)$. 
Moreover,  Lemma \ref{L:spr-T} and \cite[Prop. 4.9]{AFS2} imply that $E_T$ is a stacky projective bundle over the disjoint unions of stacks of the form $\M^{T'}_{g',n'}$ for suitable $T'\subseteq T_{g',n'}$ with either $g'<g$ and $1\leq n'\leq n+1$ or $(g',n')=(g-2,n+2)$.
Since  all the above stacks  $\M^{T'}_{g',n'}$ admit proper good moduli spaces by assumption, also $E_T$ admits a proper good moduli space. 
 We can now apply Proposition \ref{P:fincov} and Remark \ref{R:comm}\eqref{R:comm1} to infer that $\m Z_T^-$ admits a proper good moduli space.
\end{proof}

Now we turn to the existence of a good moduli space for the stack $\m Z^+_T$.

\begin{proposition}\label{prop:+}
Fix $T\subseteq T_{g,n}$ with $(g,n)\neq (2,0)$ and assume that $\car(k)\gg T$.
If $\M^{T'}_{g',n'}$ admits a proper good moduli space for all $T'\subseteq T_{g',n'}$ with either $g'<g$ and $1\leq n'\leq n+1$ or $(g',n')=(g-2,n+2)$, then $\m Z^+_T\subset \M_{g,n}^T$ admits a proper good moduli space.
\end{proposition}

Note that $\m Z_T^+$ coincides with the stack $\ov{\mathcal H}_{g,n}(7/10)$ of \cite[Sec. 4]{AFS2} in the case where $T^{\adm}=T_{g,n}^{\adm}$. Hence, the above Proposition generalises \cite[Prop. 4.15]{AFS2} for $\alpha_c=7/10$ (but one has to assume that $(g,n)\neq (2,0)$). 
At the other extreme, if $T^{\adm}=\emptyset$  then $\m Z_T^+=\emptyset$ by Remark \ref{R:ourstacks} and the result is trivial.
Moreover, if $(g,n)=(1,2)$ and $T^{\adm}\neq \emptyset$ then $\m Z_T^+=\M_{1,2}^T$ admits a point as good moduli space by Remark \ref{R:special-gn}
(which follows also from the description  $\m Z_T^+=\ov{\mathcal H}_{1,2}(7/10)\cong [\bbA^3/\bbG_m]$, where $\Gm$ acts on $\bbA^3$ with weights $2,3$ and $4$, see  \cite[Lemma 4.11]{AFS2}).

The strategy of proof of Proposition \ref{prop:+} is similar to the one of loc. cit. and it consists in finding a finite cover of $\m Z_T^+$ consisting of the disjoint union of the product of a stack admitting a good moduli space with suitable stacks  $\M^{T'}_{g',n'}$ (as in the statement of Proposition \ref{prop:+})  and then concluding by applying Proposition \ref{P:fincov}.
In order to employ this strategy we will need to review some constructions from \cite[Sec. 4.3]{AFS2}, adapted to our setting and notation.

For any integer $r\geq 1$, let 
$$
\EC_r \subset \M_{2r-1,2}(7/10)=\M_{2r-1,2}^{T_{2r-1,2}}
$$
be the closure (with reduced structure) of the locally closed substack of elliptic chains of length $r$. 
 It is proved in \cite[Lemma 4.12]{AFS2} that $\EC_r$ admits a proper good moduli space.

By gluing to an elliptic chain of length $r$ suitable pointed curves, we can obtain $n$-pointed curves in $\U_{g,n}(A_3)$. More precisely, there are the following two types of gluing morphisms.

\begin{itemize}
\item For any $1\leq r \leq g/2$, we consider the gluing morphism 
$$
\begin{aligned}
\cG_{\irr}^r: \U_{g-2r,n+2}(A_3)\times \EC_r  &\longrightarrow \U_{g,n}(A_3), \\
\left((C, \{p_i\}_{i=1}^{n+2}),  (Z, q_1, q_2) \right)& \mapsto (C\cup Z,\{p_i\}_{i=1}^n)/(p_{n+1}\sim q_1, p_{n+2}\sim q_2). 
\end{aligned}
$$
Note that we  included in this case also the limit case  $(g,n)=(2r,0)$, in which case $\U_{g-2r,n+2}(A_3)=\U_{0,2}(A_3)=\emptyset$ and in the above construction we have to glue $q_1$ with $q_2$.

\item For any  $0\leq h \leq g-2r+1$ and any $M\subseteq [n]$ with the restriction that  $|M|\geq 1$ if either $h=0$ or $h=g-2r+1$, we consider the gluing morphism
$$\begin{aligned}
\cG_{h, M}^r: \U_{h,|M|+1}(A_3)\times \U_{g-h-2r+1, |M^c|+1}(A_3)\times \EC_r  &\longrightarrow \U_{g,n}(A_3), \\
\left((C, \{p_i\}_{i\in M}, s_1), (C', \{p'_i\}_{i\in M^c}, s_2),  (Z, q_1, q_2) \right)& \mapsto (C\cup C'\cup Z,\{p_i\}_{i=1}^{n})/(s_1\sim q_1, s_2\sim q_2). 
\end{aligned}$$
Note that we  included in this case also the three degenerate cases  $(h,M)=(0,\{j\})$, in which case $\U_{h,|M|+1}(A_3)=\U_{0,2}(A_3)=\emptyset$ and the point $q_1$ is regarded as the $j$-th marked point, or $(g-h-2r+1,M^c)=(0,\{j\})$, in which case $ \U_{g-h-2r+1, |M^c|+1}(A_3)=\U_{0,2}(A_3)=\emptyset$ and the point $q_2$ is regarded as the $j$-th marked point, and or the case where both occurrences happen, namely the case $(g, n)=(2r-1, 2)$, when the above morphism is the embedding  of $\EC_r$ into $\U_{2r-1,2}(A_3)$. 
\end{itemize}

It follows from \cite[Lemma 4.13 and 4.14]{AFS2} that the morphisms $\cG_{\irr}^r$ and $\cG_{h, M}^r$ are finite. 
For later use, observe that  the stacks of the form $\U_{g',n'}(A_3)$ that  appear in the domain of the morphisms $\cG_{\irr}^r$ and $\cG_{h, M}^r$  have the property that $n'\geq 1$, i.e. there is at least one marked point.

We now study the compatibility of the maps $\cG_{\irr}^r$ and $\cG_{h, M}^r$  with the open substacks of $T$-semistable curves.

\begin{lemma}\label{L:chain-T}
Let $T\subseteq T_{g,n}$. 
\begin{enumerate}[(i)]
\item \label{L:chain-T1} If $\irr\in T$ then 
$$(\cG_{\irr}^r)^{-1}(\M_{g,n}^T)=\M_{g-2r,n+2}^{T_{g-2r,n+2}}\times \EC_r.$$
 \item \label{L:chain-T2} 
If $\{[h,M],\ldots,[h+2r-1,M]\}\subseteq T$ and $(h,M), (g-h-2r+1,M^c)\neq (1,\emptyset)$ then  
$$(\cG_{h,M}^r)^{-1}(\M_{g,n}^T)=\M_{h,|M|+1}^{\wt{T}_{h,M}}\times \M_{g-h-2r+1,|M^c|+1}^{\wt{T}_{g-h-2r+1,M^c}}\times \EC_r$$ 
where $\wt{T}_{h,M}$ is the subset of $T_{h,|M|+1}$ defined by 
$$\begin{sis}
\irr \in \wt{T}_{h,M} & \Leftrightarrow \irr\in T, \\
[\tau, I]\in \wt{T}_{h,M} & \Leftrightarrow 
\begin{cases}
[\tau, I]\in T & \text{ if } |M|+1\not\in I, \\
[h-\tau,[|M|+1]-\{I\}]
\in T 
& \text{ if } |M|+1\in I.
\end{cases} 
\end{sis}$$
with the convention that $[|M|]=[|M|+1]-\{|M|+1\}$ is identified with the subset $M\subset [n]$ (which allows to consider any subset of $[|M|]$ as a subset of $[n]$), and where $\wh{T}_{g-h-2r+1,M^c}\subseteq T_{g-h-2r+1,|M^c|+1}$ is defined similarly by replacing $h$ with $g-h-2r+1$ and $M$ with $M^c$.
\end{enumerate}
\end{lemma}
\begin{proof}

Let us prove first \eqref{L:chain-T1}. First of all, note that $\cG_{\irr}^r\left((C, \{p_i\}_{i=1}^{n+2}),  (Z, q_1, q_2) \right)$ does not have an $A_1$ or $A_3$-attached elliptic chain if and only if the same is true for $(C, \{p_i\}_{i=1}^{n+2})$. Moreover, every tacnode of $Z$ and of $C$ become of type $\{\irr\}$ in $(C\cup Z,\{p_i\}_{i=1}^n)/(p_{n+1}\sim q_1, p_{n+2}\sim q_2)$, from which the conclusion follows. 

Let us now prove \eqref{L:chain-T2}. We will assume that we are not in one of the three degenerate cases discussed above after the definition of $\cG_{h,M}^r$, and leave these three limit cases (that are easier to deal with) to the reader. 
First of all, note that, since $(h,M), (g-h-2r+1,M^c)\neq (1,\emptyset)$ by assumption, $\cG_{h,M}^r\left((C, \{p_i\}_{i\in M}, s_1), (C', \{p'_i\}_{i\in M^c}, s_2),  (Z, q_1, q_2)  \right)$ 
does not have an $A_1$ or $A_3$-attached elliptic chain if and only if the same is true for $(C, \{p_i\}_{i\in M}, s_1)$ and  $(C', \{p'_i\}_{i\in M^c}, s_2)$. Next, every tacnode of $Z$, when considered in 
$(C\cup C'\cup Z,\{p_i\}_{i=1}^{n})/(s_1\sim q_1, s_2\sim q_2)$, is of type contained in  $\{[h,M],\ldots,[h+2r-1,M]\}$, and hence in $T$ by our assumption. On the other hand, a tacnode of $(C, \{p_i\}_{i\in M}, s_1)$ of type $\{\irr\}$ remains of type $\{\irr\}$ when seen in 
$(C\cup C'\cup Z,\{p_i\}_{i=1}^{n})/(s_1\sim q_1, s_2\sim q_2)$, while if it has type $\{[\tau, I], [\tau+1,I]\}$ then it remains of the same type if $|M|+1\not\in I$ while it becomes of type $\{[h-\tau-1,[|M|+1]-\{I\}], [h-\tau,[|M|+1]-\{I\}]\}$ if $|M|+1\in I$.
A similar analysis can be done for  $C'$, and this concludes the proof.
\end{proof}

Using the above Lemma, we can prove  the existence of the good moduli space for $\m Z_T^+$. 

\begin{proof}[Proof of Proposition \ref{prop:+}]
First of all,  remark that $\M_{g,n}^T=\M_{g,n}^{T\setminus [1,\emptyset]}$, because a tacnode of type $[1,\emptyset]$ corresponds to a tacnodal elliptic tail, and we have already removed such a tail from $\M_{g,n}^T$ in Definition \ref{D:ourstacks}. We can thus assume that $[1,\emptyset]\not\in T$.

Consider the stack
$$H_T:=
\begin{sis}
\coprod_{\{[h,M],\ldots,[h+2r-1,M]\}\subseteq T}  (\cG_{h,M}^r)^{-1}(\M_{g,n}^T) & \text{ if } \irr \not \in T,\\
\coprod_{\{[h,M],\ldots,[h+2r-1,M]\}\subseteq T}  (\cG_{h,M}^r)^{-1}(\M_{g,n}^T) \coprod_{1\leq r\leq g/2 } (\cG_{\irr}^r)^{-1}(\M_{g,n}^T) & \text{ if } \irr  \in T.\\
\end{sis}
$$

The finite morphisms $\cG_{\irr}^r$ and $\cG_{h, M}^r$ give rise to a finite morphism 
$$\cG_T:H_T\to \M_{g,n}^T,$$
whose image, by construction, is  the locus of $\M_{g,n}^T$ having at least one $A_1$/$A_1$-attached elliptic chain of type contained in $T$, i.e. exactly $\m Z_T^+$.

Observe that the algebraic stack $\m Z_T^+$, being a closed substack of $\M_{g,n}^T$, is a global quotient stack of a normal variety by Remark \ref{R:quotstack} and it has linearly reductive stabilisers by Lemma \ref{L:linred} and our assumption on $\car(k)$. 
Moreover,  Lemma \ref{L:chain-T} implies that the stack $H_T$ is a (finite) disjoint union of products of the stacks $\EC_r$,  which admit proper good moduli space by   \cite[Lemma 4.12]{AFS2}, and of the stacks $\M^{T'}_{g',n'}$ for suitable $T'\subseteq T_{g',n'}$ with either $g'<g$ and $1\leq n'\leq n+1$ or $(g',n')=(g-2,n+2)$, 
 which admit proper good moduli space by assumption.  Therefore also $H_T$ admits a proper good moduli space.
 We can now apply Proposition \ref{P:fincov} and Remark \ref{R:comm}\eqref{R:comm1} to infer that $\m Z_T^+$ admits a proper good moduli space.
\end{proof}

We can now proof the main result of this section.

\begin{proof}[Proof of Theorem \ref{T:goodspaces}]
First of all, Proposition \ref{P:VGIT} implies that  the two open embeddings 
$$
\M_{g,n}^{\ps}\hookrightarrow \M_{g,n}^T\hookleftarrow \M_{g,n}^{T+}
$$
arise from local VGIT with respect to the line bundle $\delta-\psi$ on $\M_{g,n}^T$.

Next, the stack $\M_{g,n}^{\ps}$ admits a coarse proper moduli space $\phi^{\ps}:\M_{g,n}^{\ps}\to \MM_{g,n}^{\ps}$ (see Proposition  \ref{P:Mgps-DM}). Since the stabiliser of any $k$-point of $\M_{g,n}^{\ps}$ is linearly reductive by our assumption on the characteristic (see Lemma \ref{L:linred} and recall that $\M_{g,n}^{\ps}\subseteq \M_{g,n}^T$), we infer that $\phi^{\ps}$ is also a good moduli space by  \cite[Thm. 3.2]{AOV}. 

Therefore, thanks to \cite[Theorem 1.3]{AFS2}, the existence of proper good moduli spaces fitting into the commutative diagram \eqref{E:diag-spaces} will follow if we show that the stacks $\m Z_T^-=\M_{g,n}^T\setminus \M_{g,n}^{\ps}$ and $\m Z_T^+=\M_{g,n}^T\setminus \M_{g,n}^{T+}$ admit good moduli spaces. 
This follows from Propositions \ref{prop:-} and \ref{prop:+} using induction on $g$: the base of the induction is the case $g=0$ when $\M_{0,n}^T=\M_{0,n}$ is a variety (hence it is its own good moduli space) and the assumption on the characteristic of the base field $k$ allows us to apply induction. 
Observe that the non existence of a proper moduli space for $\M_{2,0}^{\irr}$ (see Remark \ref{R:special-gn}) does not interfere with this inductive proof since all the stacks $\M_{g',n'}^{T'}$ appearing in the inductive hypothesis of Propositions \ref{prop:-} and \ref{prop:+} are such that $n'\geq 1$.

Finally, observe that the morphisms $f_T$ and $f_T^+$  are proper (being  morphisms between  proper algebraic spaces) and all the good moduli spaces are normal and irreducible since the corresponding algebraic stacks are smooth and irreducible by Theorem \ref{T:algstack} (see \cite[Theorem 4.16(viii)]{Alper}). 
\end{proof}

\begin{remark}\label{R:modcom}
Since the stacks $\M_{g,n}^T$ and $\M_{g,n}^{T+}$ contain the stack ${\mathcal M}_{g,n}$ of $n$-pointed smooth curves of genus $g$  as an open substack, the spaces $\MMM_{g,n}^T$ and $\MMM_{g,n}^{T+}$ are weakly modular compactification of $M_{g,n}$ in the sense of \cite[Def. 2.6]{FS}. Moreover, they are modular compactification of $M_{g,n}$ in the sense of \cite[Def. 2.1]{FS} whenever the spaces $\MMM_{g,n}^T$ and $\MMM_{g,n}^{T+}$ are coarse moduli spaces, or equivalently whenever the stacks $\M_{g,n}^T$ and $\M_{g,n}^{T+}$ are DM, and this happens  when
\begin{itemize}
\item $\M_{g,n}^{T}$ is a DM stack if and only if $\car(k)\gg T$ and $\M_{g,n}^T=\M_{g,n}^{\ps}$, i.e. if and only if $T^{\adm}=\emptyset$.
\item Assume that $\car(k)\gg T$. Then $\M_{g,n}^{T+}$ is a DM stack if and only if $T$ does not contain subsets of the form $\{[\tau,I], [\tau+1,I], [\tau+2,I]\}$ with $[\tau,I],[\tau+2,I]\neq [1,\emptyset]$.
\end{itemize}

\end{remark}

\end{section}
		
\begin{section}{The moduli space of pseudostable curves and the Elliptic bridge face} \label{Sec:Mgps}

The aim of this section is to study the geometric properties of the moduli space $\MM_{g,n}^{\ps}$ of pseudostable curves   and to describe a face of its Mori cone, that we call the elliptic bridge face, which  will play a special role in the sequel.

We start by studying the singularities, the Picard group and the canonical class of  $\MM_{g,n}^{\ps}$.  

\begin{proposition}\label{P:Pic-Mgps}
Assume that $(g,n)\neq (2,0)$ and that $\car(k)\neq 2,3$. 
Consider the stack $\M_{g,n}^{\ps}$ of pseudostable curves of genus $g$ with $n$ marked points and let $\phi^{\ps}:\M_{g,n}^{\ps}\to \MM_{g,n}^{\ps}$ be the morphism into its coarse moduli space.  
\begin{enumerate}[(i)]

\item \label{P:Pic-Mgps2} The space $\MM_{g,n}^{\ps}$ is normal with finite quotient singularities, hence it is $\bbQ$-factorial. If ${\rm char}(k)=0$, then $\MM_{g,n}^{\ps}$  is klt. 

\item \label{P:Pic-Mgps2b} 
The pull-back via the morphism $\phi^{\ps}$  induces an isomorphism 
$$(\phi^{\ps})^*:\Pic(\MM_{g,n}^{\ps})_{\bbQ}\stackrel{\cong}{\longrightarrow}\Pic(\M_{g,n}^{\ps})_{\bbQ}.$$

\item \label{P:Pic-Mgps3} If $(g,n)\neq (1,2),  (2,1), (3,0)$, then the canonical line bundle of $\MM_{g,n}^{\ps}$ is such that 
$$(\phi^{\ps})^*(K_{\MM_{g,n}^{\ps}})=K_{\M_{g,n}^{\ps}}.$$ 
In particular, using \eqref{P:Pic-Mgps2b} and Mumford's formula for $K_{\M_{g,n}^{\ps}}$ (see Fact \ref{F:PicU}\eqref{F:PicU2}), we get
$$K_{\MM_{g,n}^{\ps}}=13\lambda-2\delta+\psi.$$
\end{enumerate}
\end{proposition}
From now on, we identify (in $\car(k)\neq 2,3$) $\bbQ$-line bundles on $\M_{g,n}^{\ps}$ with $\bbQ$-line bundles on $\MM_{g,n}^{\ps}$ via the isomorphism $(\phi^{\ps})^*$ of \eqref{P:Pic-Mgps2b}, similarly for what is usually done for $\bbQ$-line bundles on $\M_{g,n}$ and on $\MM_{g,n}$.

\begin{proof}
Part \eqref{P:Pic-Mgps2}: since $\M_{g,n}^{\ps}$ is a smooth and separated DM stack  by Fact \ref{F:embAFSV} and Proposition \ref{P:Mgps-DM}, its coarse moduli space $\MM_{g,n}^{\ps}$ is normal with finite quotient singularities 
by \cite[Lemma 2.2.3]{AV}. We conclude since finite quotient singularities are $\bbQ$-factorial and, if $\car(k)=0$,  klt  by  \cite[Prop. 5.15]{KM}.


Part \eqref{P:Pic-Mgps3}: it is enough to show that the morphism $\phi^{\ps}$ is an isomorphism in codimension one. First of all, the assumptions on $(g,n)$ guarantee that the locus of $n$-pointed smooth curves with non-trivial automorphisms has codimension at least two (see \cite[Chap. XII, Prop. 2.15]{GAC2}); hence the morphism $\phi^{\ps}$ is an isomorphism in codimension one when restricted to ${\mathcal M}_{g,n}\subset \M_{g,n}^{\ps}$. 
Secondly, a generic point of a boundary divisor of $\M_{g,n}^{\ps}$ does not have non-trivial automorphisms, hence $\phi^{\ps}$ is an isomorphism in codimension one also at the boundary of $\M_{g,n}^{\ps}$.

Part \eqref{P:Pic-Mgps2b}: consider the following commutative diagram 
$$
\xymatrix{
\Pic(\M_{g,n}^{\ps})_{\bbQ}  \ar[r] & \Cl(\M_{g,n}^{\ps})_{\bbQ} \ar[d]^{(\phi^{\ps})_*}\\
 \Pic(\MM_{g,n}^{\ps})_{\bbQ}  \ar[u]^{(\phi^{\ps})^*} \ar[r] & \Cl(\MM_{g,n}^{\ps})_{\bbQ} 
}
$$
The upper horizontal morphism is an isomorphism because $\M_{g,n}^{\ps}$ is a smooth stack; the lower horizontal arrow is an isomorphism since $\MM_{g,n}^{\ps}$ is normal and $\bbQ$-factorial by \eqref{P:Pic-Mgps2}; the right vertical arrow is an isomorphism since $\phi^{\ps}$ is an isomorphism in codimension one as observed
in the proof of \eqref{P:Pic-Mgps3}. Hence, by commutativity of the diagram, we infer that $(\phi^{\ps})^*$ is also an isomorphism. 
\end{proof}

\begin{remark}
We do not know if $\MM_{g,n}^{\ps}$ has finite quotient singularities, or simply if it is $\bbQ$-factorial, if $\car(k)=2,3$ (see also \cite[Rmk. 3.6]{Alp10}). If $\MM_{g,n}^{\ps}$ is $\bbQ$-factorial also if
$\car(k)=2,3$, then all the results of this section extend to $\car(k)=2,3$. 
\end{remark}

\begin{remark}\label{R:Picg2}
The first two points of the above Proposition remain true for $(g,n)=(2,0)$. 

Indeed, part \eqref{P:Pic-Mgps2} follows from the fact that $\MM_2^{\ps}$  is  isomorphic to the GIT quotient of binary sextics (see \cite[Thm. 2]{HL}), which is   isomorphic to the weighted projective space $\P(1,2,3,5)$ (see \cite[Prop. 2.2(1)]{Has} and the references therein), and hence it has finite quotient singularities. 

On the other hand, part \eqref{P:Pic-Mgps2b} follows from the fact that $\M_2^{\ps}$ is smooth, $\MM_2^{\ps}$ has $\bbQ$-factorial singularities and the morphism $\phi^{\ps}:\M_2^{\ps}\to \MM_2^{\ps}$ is finite in codimension one by Remark \ref{R:ps-g2}.
\end{remark}

\begin{remark}\label{R:Hurwitz}

In the exceptional  cases excluded by Proposition \ref{P:Pic-Mgps}\eqref{P:Pic-Mgps3} (and also for $(g,n)=(2,0)$) we can apply Hurwitz formula to the morphism $\phi^{\ps}:\M_{g,n}^{\ps}\to \MM_{g,n}^{\ps}$ in order to get 
$$K_{\M_{g,n}^{\ps}}=(\phi^{\ps})^*(K_{\MM_{g,n}^{\ps}})+R=K_{\MM_{g,n}^{\ps}}+R,$$ 
where $R$ is (the class of) the effective ramification divisor. Using  Mumford's formula for $K_{\M_{g,n}^{\ps}}$, we have that 
$$K_{\MM_{g,n}^{\ps}}=13\lambda-2\delta+\psi-R.$$
Moreover, from the proof of Proposition \ref{P:Pic-Mgps}\eqref{P:Pic-Mgps3}, it follows that $R$ is an effective divisor not contained in the boundary of $\M_{g,n}^{\ps}$.

\end{remark}

We now focus on the relation of the coarse moduli space $\MM_{g,n}^{\ps}$ of pseudostable curves with the coarse moduli space $\MM_{g,n}$ of stable curves.  Note that, for $(g,n)\neq (1,1), (2,0)$,  the morphism of stacks $\wh{\Upsilon}:\M_{g,n}\to \M_{g,n}^{\ps}$ of Proposition \ref{P:Mgps-DM}\eqref{P:Mgps-DM0} induces a proper morphism between their coarse moduli spaces
\begin{equation}\label{E:Ups}
 \Upsilon:\MM_{g,n}\to \MM_{g,n}^{\ps}.
 \end{equation}

\begin{proposition}\label{P:div-contr}
Assume that $(g,n)\neq (1,1), (2,0)$ and that $g\geq 1$.
\begin{enumerate}[(i)]

\item \label{P:div-contr3} The space $\MM_{g,n}^{\ps}$ is  isomorphic to the following log canonical model of $\MM_{g,n}$:
$$\MM_{g,n}^{\ps}\cong \MM_{g,n}\left(\frac{9}{11}\right):=\Proj \bigoplus_{m\geq 0} H^0(\M_{g,n}, \lfloor m(K_{\M_{g,n}}+\psi+\frac{9}{11}(\delta-\psi))\rfloor).
$$

In particular, $\MM_{g,n}^{\ps}$ is a normal projective variety. 

\item \label{P:div-contr4} The morphism $\Upsilon$  is the contraction of the extremal ray $\bbR_{\geq 0}\cdot [C_{\rm ell}]$ of  the Mori cone $\NEb(\MM_{g,n})$, which intersects negatively 
$K_{\M_{g,n}}$, $K_{\M_{g,n}}+\psi$, $K_{\MM_{g,n}}$ and $K_{\MM_{g,n}}+\psi$. Moreover, $\Upsilon$ is a divisorial contraction and the exceptional locus is the divisor $\Delta_{1,\emptyset}$.
\item \label{P:div-contr5} Assume that $\car(k)\neq 2,3$. The pull-back map $\Upsilon^*:\Pic(\MM_{g,n}^{\ps})_{\bbQ}\to \Pic(\MM_{g,n})_{\bbQ}$ is determined by the following relations:
$$\begin{sis}
& \Upsilon^*(\lambda)=\lambda+\delta_{1,\emptyset}, \\
& \Upsilon^*(\delta_{\irr})=\delta_{\irr}+12\delta_{1,\emptyset}, \\
& \Upsilon^*(\delta_{i,I})=\delta_{i,I} \quad \text{ for any } [i,I]\neq [1,\emptyset]. 
\end{sis}$$ 
\end{enumerate}
\end{proposition}

\begin{proof}

Some parts of this theorem are proved for $n=0$ in  \cite{HH1}  and \cite{HM}  and some other parts are proved in  \cite{AFS3} under the assumption that ${\rm char}(k)=0$. 
Let us convince the reader that the proofs in the above mentioned papers work for any $n$ and over an arbitrary algebraically closed field $k$.
Consider  the $\bbQ$-line bundle on $\MM_{g,n}$
$$L_{g,n}:=K_{\M_{g,n}}+\psi+\frac{9}{11}(\delta-\psi)=K_{\M_{g,n}}+\frac{9}{11}\delta+\frac{2}{11}\psi.$$
By \cite[Introduction]{AFS3} the line bundle $L_{g,n}$  is nef and it has degree $0$ precisely on the curves that are numerically equivalent to $C_{\rm ell}$. Moreover, we claim that $L_{g,n}$ is semiample on $\MM_{g,n}$. Indeed, in the case $n=0$, $L_{g,0}$ is the pull-back via $\Upsilon$ of the natural polarisation coming from the identification of $\MM_{g}^{\ps}$ with the  GIT quotient of the Chow variety of $4$-canonical curves (see \cite[Thm. 7]{HM} and \cite[Thm. 3.1]{HH2}). In the case $n>0$, $L_{g,n}$ is the pull-back of $L_{g+n,0}$ via the regular morphism $\MM_{g,n}\to \MM_{g+n}$ that attaches a fixed smooth elliptic curve to each of the marked points of an $n$-pointed stable curve of genus $g$ (see \cite[Lemma (4.38)]{GAC2}). 

These facts imply that a sufficiently high multiple of $L_{g,n}$ induces a regular morphism 
$$\pi: \MM_{g,n}\to \Proj \bigoplus_{m\geq 0} H^0\left(\MM_{g,n}, \left\lfloor m\left(K_{\M_{g,n}}+\frac{9}{11}\delta+\frac{2}{11}\psi\right)\right\rfloor\right) $$
which is  the contraction of the extremal ray $\bbR_{\geq 0}\cdot C_{\rm ell}$ of $\NE(\MM_{g,n})$. The codomain coincides with $\displaystyle \MM_{g,n}\left(\frac{9}{11} \right)$ because $$H^0\left(\MM_{g,n}, \left\lfloor m\left(K_{\M_{g,n}}+\frac{9}{11}\delta+\frac{2}{11}\psi\right)\right\rfloor\right) = H^0\left(\M_{g,n}, \left\lfloor m\left(K_{\M_{g,n}}+\frac{9}{11}\delta+\frac{2}{11}\psi\right)\right\rfloor\right)$$ 
for all $m$ divisible by the cardinality of all inertia groups of $\M_{g,n}$ (see  also \cite[Prop. A.13]{HH1}).

Now observe that, by the modular description of $\Upsilon$, an integral curve of $\MM_{g,n}$ lies on a closed fiber of $\Upsilon$ if and only if its class lies in $\bbR_{\geq 0}\cdot C_{\rm ell}$. Moreover, $\Upsilon$ is a contraction by the Zariski main theorem since it is a proper morphism between irreducible normal algebraic spaces which is moreover birational (being an isomorphism when restricted to the dense open subset of smooth curves). Therefore, using the rigidity Lemma \ref{L:rigidity}, we get an isomorphism $\MM_{g,n}^{\ps}\cong  \MM_{g,n}\left(\frac{9}{11} \right)$ under which $\Upsilon$ gets identified to $\pi$.

Using Mumford's formula $K_{\M_{g,n}}=13\lambda-2\delta+\psi$ and the formulae  \cite[Thm. 2.1]{GKM}, we compute that 
$$
C_{\rm ell}\cdot K_{\M_{g,n}}=C_{\rm ell}\cdot (K_{\M_{g,n}}+\psi)=-9.
$$
If $(g,n)\neq (1,2),  (2,1), (3,0)$ then we have that $K_{\MM_{g,n}}=K_{\M_{g,n}}-\delta_{1,\emptyset}$ by  \cite[Chap. XII, Cor. 7.16]{GAC2}, and then, using again the formulae \cite[Thm. 2.1]{GKM}, we compute that 
$$
C_{\rm ell}\cdot K_{\MM_{g,n}}=C_{\rm ell}\cdot (K_{\MM_{g,n}}+\psi)=-8.
$$
In the above exceptional cases, we have that $K_{\MM_{g,n}}=K_{\M_{g,n}^{\ps}}-\delta_{1,\emptyset}-R$ with $R$ is part the ramification divisor of the morphism $\phi:\M_{g,n}\to \MM_{g,n}$ not entirely contained in the boundary of $\MM_{g,n}$. We can choose the  curve $C_{\rm ell}$  (in its numerical equivalence class) in such a way that the automorphism group of its generic point is generated by the elliptic involution along the elliptic tail,  which implies that $C_{\rm ell}$ is not contained in $R$. 
This ensures that $C_{\rm ell}$  intersects   $R$ non-negatively and hence it  intersects negatively also $K_{\MM_{g,n}^{\ps}}$ and $K_{\MM_{g,n}^{\ps}}+\psi$.


Finally, the exceptional locus of $\Upsilon$ contains $\Delta_{1,\emptyset}$  since the curves numerically equivalent to $C_{\rm ell}$ cover  $\Delta_{1,\emptyset}$. On the other hand, since $\delta_{1,\emptyset}\cdot C_{\rm ell}=-1<0$ by  \cite[Thm. 2.1]{GKM}, any curve numerically equivalent to $C_{\rm ell}$ is contained in $\Delta_{1,\emptyset}$. Therefore the exceptional locus of $\Upsilon$ is  equal to $\Delta_{1,\emptyset}$, and hence $\Upsilon$ is a divisorial contraction.
This concludes the proof of \eqref{P:div-contr3} and \eqref{P:div-contr4}.

In order to prove the last part \eqref{P:div-contr5}, observe that, since the exceptional locus of $\Upsilon$ is equal to $\Delta_{1,\emptyset}$, the pull-back of a $\bbQ$-line bundle $L$ on $\MM_{g,n}^{\ps}$ is equal to $L+\alpha(L)\delta_{1,\emptyset}$ for some  $\alpha(L)\in \bbQ$. The rational number $\alpha(L)$ is uniquely determined by imposing that $C_{\rm ell}\cdot \Upsilon^*(L)=\Upsilon_*(C_{\rm ell})\cdot L=0$ (because $C_{\rm ell}$ is contracted by $\Upsilon$), and can be computed using  \cite[Thm. 2.1]{GKM}.
\end{proof}

\begin{remark}\label{C:Upsg2}
Some parts of the above Proposition are true also for $(g,n)=(2,0)$. More specifically, Hyeon-Lee construct in \cite[Sec. 4]{HL} (see also \cite[Prop. 4.2]{Has}) a contraction $\Upsilon:\MM_2\to \MM_2^{\ps}$ which contracts $\Delta_{1,\emptyset}$ (even though $\Upsilon$ does not come from a morphism between the corresponding stacks).
Moreover, we have the identification $\displaystyle \MM_2^{\ps}\cong \MM_2\left(\frac{9}{11}\right)$, as it follows by combining \cite[Thm. 4.10]{Has} and \cite[Thm. 4.2]{HL}. Finally, the proof of \eqref{P:div-contr5} extends verbatim to the case $(g,n)=(2,0)$.
\end{remark}

\begin{remark}\label{C:modVGIT}
In characteristic zero, the morphism $\Upsilon$ admits another description. 

Indeed, from the two open embeddings of Fact \ref{F:embAFSV}, passing to their good moduli spaces (in $\car(k)=0$), 
we get the following proper birational morphisms between normal proper algebraic spaces (see \cite[Thm. 1.1]{AFS2} for $\alpha_c=9/11$)
$$\xymatrix{\MM_{g,n}\ar[r]^(0.4){j_1^+} & \MMM_{g,n}(9/11) & \MM_{g,n}^{\ps}=\MMM_{g,n}(9/11-\epsilon)\ar[l]_(0.6){j_1^-}}.$$
By \cite[Thm. 2.2]{AFSV1}, the morphism $j_1^+$ (resp. $j_1^-$) is defined on geometric points by sending a stable (resp. pseudostable) curve into the curve which is obtained by replacing each elliptic tail (resp. cusp)  by a cuspidal elliptic tail. Since cusps do not have local moduli, the map $j_1^-$ is bijective on geometric points and hence, being 
proper and birational between normal algebraic spaces, it is an isomorphism by Zariski's main theorem. Comparing the above descriptions of $j_1^+$ and $j_1^-$ on geometric points and the description of $\Upsilon$ contained in Proposition \ref{P:Mgps-DM}\eqref{P:Mgps-DM0}, we deduce that 
 $$\Upsilon=(j_1^-)^{-1}\circ j_1^+.$$
\end{remark}

We now study the elliptic bridge curves in $\MM_{g,n}^{\ps}$ introduced in Definition \ref{D:1strata}. Let us first determine their intersections with the $\bbQ$-line bundles on $\M_{g,n}^{\ps}$ (or on $\MM_{g,n}^{\ps}$).

\begin{lemma}\label{L:int-1str}
Assume that $\car(k)\neq 2,3$. Given a $\bbQ$-line bundle $L=a\lambda+b_{\irr}\delta_{\irr}+\sum_{[i,I]\in T^*_{g,n}-\{[1,\emptyset]\}} b_{i,I} \delta_{i,I} $ in $\M_{g,n}^{\ps}$, we have the following intersection formulas
$$\begin{sis}
& C(\irr) \cdot L= a+10 b_{\irr}, \\
& C([\tau, I], [\tau+1, I])\cdot L=a+12 b_{\irr}-b_{\tau,I}-b_{\tau+1,I}.
\end{sis}$$
\end{lemma}
\begin{proof}
We can compute the intersection on the moduli space $\MM_{g,n}^{\ps}$. The curves $C(\irr)$ and $C([\tau, I], [\tau+1, I])$ in $\MM_{g,n}^{\ps}$ are push-forward via $\Upsilon$ of irreducible curves $\wt C(\irr)$ and $\wt C([\tau, I], [\tau+1, I])$ in $\MM_{g,n}$ that are defined in the same way. Therefore, by the projection formula, we have 
\begin{equation}\label{E:inter1}
\begin{sis} 
& C(\irr) \cdot L=\wt C(\irr) \cdot \Upsilon^*(L), \\
& C([\tau, I], [\tau+1, I])\cdot L=\wt C([\tau, I], [\tau+1, I])\cdot \Upsilon^*(L).
\end{sis}
\end{equation}
Now, Proposition \ref{P:div-contr}\eqref{P:div-contr5} gives that 
\begin{equation}\label{E:inter2}
\Upsilon^*(L)=a\lambda+b_{\irr}\delta_{\irr}+(a+12b_{\irr})\delta_{1,\emptyset}+\sum_{[i,I]\in T^*_{g,n}-\{[1,\emptyset]\}} b_{i,I} \delta_{i,I}.
\end{equation}
Finally, observe  that the curve $\wt C(\irr)$ coincide with the curve of \cite[Thm. 2.2(4)]{GKM} for $(i,I)=(0,\emptyset)$, while the curve $\wt C([\tau, I], [\tau+1, I])$ coincide with the curve of \cite[Thm. 2.2(5)]{GKM} for $(i,I)=(\tau,I)$ and $(j,J)=(g-1-\tau, I^c)$. Hence, using \cite[Thm. 2.1]{GKM}, we get 
\begin{equation}\label{E:inter3}
\begin{sis} 
& \wt C(\irr) \cdot \left(a\lambda+b_{\irr}\delta_{\irr}+\sum_{[i,I]\in T_{g,n}-\{ \irr\}} b_{i,I} \delta_{i,I}.\right)=-2b_{\irr}+b_{1,\emptyset}, \\
& \wt C([\tau, I], [\tau+1, I])\ \cdot \left(a\lambda+b_{\irr}\delta_{\irr}+\sum_{[i,I]\in T_{g,n}-\{ \irr\}} b_{i,I} \delta_{i,I}.\right)=-b_{\tau,I}-b_{\tau+1,I}+b_{1,\emptyset}.
\end{sis}
\end{equation}
We conclude by putting together \eqref{E:inter1}, \eqref{E:inter2} and \eqref{E:inter3}. 
\end{proof}

Now we look at the subcone of the Mori cone $\NEb(\MM_{g,n}^{\ps})$ spanned by the elliptic bridge curves.

\begin{proposition}\label{P:face}
Assume that $\car(k)\neq 2,3$. 
\begin{enumerate}[(i)]
\item \label{P:face1} The elliptic bridge curves are linearly independent in $N_1(\MM_{g,n}^{\ps})$ and they intersect   $K_{\M_{g,n}^{\ps}}$,  $K_{\M_{g,n}^{\ps}}+\psi$, $K_{\MM_{g,n}^{\ps}}$ and $K_{\MM_{g,n}^{\ps}}+\psi$ negatively.

\item \label{P:face2} The convex cone spanned by elliptic bridge curves is a  face of the Mori cone $\NEb(\MM_{g,n}^{\ps})$ (which we call the \emph{elliptic bridge face}).
In particular, each elliptic bridge curve generates an extremal ray of the Mori cone of $\MM_{g,n}^{\ps}$.  

\item \label{P:face3} If $(g,n)\neq (1,2), (2,0)$ then a curve $B\subset \M_{g,n}^{\ps}$ is such that its class  in $N_1(\MM_{g,n}^{\ps})$ lies in the elliptic bridge face if and only if  the only non-isotrivial components of the corresponding family of pseudostable curves $\mathcal C \to B$ are $A_1/A_1$-attached elliptic bridges.

\end{enumerate}
\end{proposition}
Note that part \eqref{P:face1} implies that the  elliptic bridge face is polyhedral and simplicial. 
Observe also that part \eqref{P:face3} is false for $(g,n)=(1,2)$ (resp. $(2,0)$): in these two cases, $\dim N_1(\MM_{g,n}^{\ps})_{\bbQ}=1$ and the elliptic bridge face, which is spanned by $C([0,\{1\}],[0,\{2\}])$ (resp. $C(\irr)$), coincide with the entire Mori cone $\NEb(\MM_{g,n}^{\ps})$ and it is therefore a half-line. Hence, the class of any effective curve on $\MM_{g,n}$  lies in the elliptic bridge face and there are plenty of effective curves in the projective varieties $\MM_{g,n}^{\ps}$.

\begin{proof}
Part \eqref{P:face1}: the fact that the elliptic bridge curves are linearly independent in $N_1(\MM_{g,n}^{\ps})$ follows  by a close  inspection of the intersection formulas in Lemma \ref{L:int-1str} using  the relations among the generators  of  $\Pic(\MM_{g,n}^{\ps})_{\bbQ}$ (see  Fact \ref{F:PicU}\eqref{F:PicU1}, Corollary \ref{C:Pic-ps} and Proposition \ref{P:Pic-Mgps}\eqref{P:Pic-Mgps2b}). 

The fact that the elliptic bridge curves intersect $K_{\M_{g,n}^{\ps}}$ and $K_{\M_{g,n}^{\ps}}+\psi$ negatively follows again from Lemma \ref{L:int-1str} and Mumford's formula $K_{\M_{g,n}^{\ps}}=13\lambda-2\delta+\psi$ (see Fact \ref{F:PicU}\eqref{F:PicU2}). This implies the analogous result for $K_{\MM_{g,n}^{\ps}}$ and $K_{\MM_{g,n}^{\ps}}+\psi$ if $(g,n)\neq (1,2), (2,0), (2,1), (3,0)$ by Proposition \ref{P:Pic-Mgps}\eqref{P:Pic-Mgps3}. In the above mentioned four exceptional cases, we have that $K_{\MM_{g,n}^{\ps}}=K_{\M_{g,n}^{\ps}}-R$ with $R$ being the ramification divisor of the morphism $\phi^{\ps}:\M_{g,n}^{\ps}\to \MM_{g,n}^{\ps}$ by Remark \ref{R:Hurwitz}. We can choose the elliptic bridge curves (in their numerical equivalence class) in such a way that their generic point does not have non trivial automorphisms, which implies that they are not contained in $R$. 
This ensures that the elliptic bridge curves  intersect   $R$ non-negatively and hence they  intersect negatively also $K_{\MM_{g,n}^{\ps}}$ and $K_{\MM_{g,n}^{\ps}}+\psi$.

Let us now prove part \eqref{P:face2} and  part \eqref{P:face3}.  If $(g,n)=(1,2)$ or $ (2,0)$ then $\dim N_1(\MM_{g,n}^{\ps})=1$   and part \eqref{P:face2} is obvious (while part \eqref{P:face3} is clearly false!).

Otherwise, consider the $\bbQ$-line bundle on $\MM_{g,n}^{\ps}$
$$N_{g,n}:=K_{\M_{g,n}^{\ps}}+\frac{7}{10}\delta+\frac{3}{10}\psi=\frac{13}{10}(10\lambda-\delta+\psi).$$
By \cite[Thm. 1.2(a)]{AFS3} (whose proof works in arbitrary characteristics and that can be applied since $(g,n)\neq (1,2), (2,0)$\footnote{Note that in the theorem of loc. cit., not only the case $(g,n)=(2,0)$ but  also the case $(g,n)=(1,2)$ must be excluded. The reason is that these are the only two cases where the line bundle  $K_{\M_{g,n}^{\ps}}+\frac{7}{10}\delta+\frac{3}{10}\psi$, which is proportional to $10\lambda-\delta+\psi=10\lambda-\wh{\delta}$, is zero on $\M_{g,n}^{\ps}$.}), the  line bundle $N_{g,n}$ is nef and it has degree $0$ precisely on the curves of $\M_{g,n}^{\ps}$ described in \eqref{P:face3}. Note that such curves are numerically equivalent to a non-negative linear combination of elliptic bridge curves in $\M_{g,n}^{\ps}$ (since $\MM_{1,2}^{\ps}$ has Picard number one by Corollary \ref{C:Pic-ps} and Proposition \ref{P:Pic-Mgps}\eqref{P:Pic-Mgps2b}) and every elliptic bridge curve intersects $N_{g,n}$ in $0$ by Lemma \ref{L:int-1str}.

Moreover, we claim that $N_{g,n}$ is semiample on $\MM_{g,n}^{\ps}$. Indeed, in the case $n=0$, $N_{g,0}$ is the pull-back  of the natural polarisation on the GIT quotient $\MM_g^c$ of the Chow variety of bicanonical curves of genus $g$ via a regular morphism $\Psi:\MM_{g}^{\ps}\to \MM_g^c$ (see \cite[Thm. 2.13]{HH2} and \cite[Thm. 3.1]{HH2}, whose proof work in arbitrary characteristic). In the case $n>0$, fixing an integer $h\geq 2$, we have that $N_{g,n}$ is the pull-back of $N_{g+nh,0}$ via the regular morphism $\MM_{g,n}^{\ps}\to \MMM_{g+nh}^{\ps}$ that attach a fixed smooth irreducible curve of genus $h$ to each of the marked points of an $n$-pointed stable curve of genus $g$ (see  \cite[Lemma (4.38)]{GAC2} and \cite[Sec. 5.4]{AFS3}).

These facts imply that, if we denote by $\eta$ the fibration induced by a sufficiently high power of $N_{g,n}$,  the convex cone spanned by the elliptic bridge curves coincides with the $\eta$-relative effective cone $\NE(\eta)$  of curves 
and it is therefore a face of the effective cone $\NE(\MM_{g,n}^{\ps})$ of curves (see Notation). Moreover, by what said above, property \eqref{P:face3} holds. 

It remains to see that the convex cone spanned by the elliptic bridge curves is also a face of the Mori cone $\NEb(\MM_{g,n}^{\ps})$. However, this  convex cone, which coincide with $\NE(\eta)$, is polyhedral (because it is generated by a finite number of curves) and hence closed. Since the closure of $\NE(\eta)$ is equal to the $\pi$-relative Mori cone $\NEb(\eta)$ (see Notation), we deduce that the convex cone spanned by the elliptic bridge curves is equal to $\NEb(\eta)$ and hence it is a face of $\NEb(\MM_{g,n}^{\ps})$.
\end{proof}

\begin{remark}\label{R:bridgecurves}
Assume that $g\geq 1$ (to avoid trivialities, since for $g=0$ there are no elliptic bridge curves).

The dimension of the elliptic bridge face, which is equal to the number of elliptic bridge curves, is equal to 
$$\dim(\text{Elliptic bridge face})=
\begin{cases}
1 & \text{ if } (g,n)=(2,0), \\
\frac{g-1}{2} & \text{ if } n=0 \text{ and } g\geq 3 \: \text{ is odd, }\\
\frac{g}{2}-1 & \text{ if } n=0 \text{ and } g\geq 4 \: \text{ is even, }\\
g 2^{n-1}-1 & \text{ if } g\geq 1 \: \text{ and } n\geq 1.
\end{cases}
$$
Comparing it with the Picard number of $\MM_{g,n}^{\ps}$, which can be obtained from Fact \ref{F:PicU}\eqref{F:PicU1}, Corollary \ref{C:Pic-ps} and Proposition \ref{P:Pic-Mgps}\eqref{P:Pic-Mgps2b}, we get that 
$$\codim(\text{Elliptic bridge face})=
\begin{cases}
0 & \text{ if } (g,n)=(2,0), \\
1 & \text{ if } n=0 \text{ and } g\geq 3 \: \text{ is odd, }\\
2 & \text{ if } n=0 \text{ and } g\geq 4 \: \text{ is even, }\\
2^{n-1}+1-\delta_{2,g}-(n+1)\delta_{1,g} & \text{ if } g\geq 1 \: \text{ and } n\geq 1,
\end{cases}
$$
where $\delta_{2,g}$ and $\delta_{1,g}$ are the Kronecker symbols.
\end{remark} 

The subfaces of the elliptic bridge face can be described as follows.

\begin{definition}\label{D:T-Face}[T-faces]
For any $T\subseteq T_{g,n}$, we denote by $F_T$ the cone in $N_1(\MM_{g,n}^{\ps})$ generated by the classes of elliptic bridge curves of type contained in $T$. We will call $F_T$ the \emph{T-face} of the Mori cone.
\end{definition}

The poset of $T$-faces is described by the following result, where we use the terminology of Definition \ref{def:admissible}.

\begin{lemma}\label{L:poset-T}
Assume that $\car(k)\neq 2,3$. 
\begin{enumerate}[(i)]
\item \label{L:poset-T1} For any $T\subseteq T_{g,n}$, the cone $F_T$ is a simplicial polyhedral  face of the Mori cone $\NEb(\MM_{g,n}^{\ps})$ whose dimension is equal to the number of minimal subsets of $T_{g,n}$ contained in $T$.
In particular, the extremal rays of the elliptic bridge face are given by $\{F_T : \: T \: \text{ is minimal}\}$.

\item \label{L:poset-T1b} If  $(g,n)\neq (1,2), (2,0)$ then a curve $B\subset \M_{g,n}^{\ps}$ is such that its class  in $N_1(\MM_{g,n}^{\ps})$ lies in $F_T$ if and only if  the only non-isotrivial components of the corresponding family of pseudostable curves $\mathcal C \to B$ are $A_1/A_1$-attached elliptic bridges of type contained in $T$.

\item \label{L:poset-T2} We have that 
$F_T\subseteq F_S$ if and only if  $T^{\adm}\subseteq S^{\adm}. $
In particular, we have that $F_T=F_S$ if and only if $T^{\adm}=S^{\adm}$. 
\end{enumerate}
\end{lemma}
\begin{proof}
Part \eqref{L:poset-T1}: the cone $F_T$  is a face of the elliptic bridge face of $\NEb(\MM_{g,n}^{\ps})$, which is a simplicial polyhedral  face of the Mori cone $\NEb(\MM_{g,n}^{\ps})$ whose extremal rays are generated by the elliptic bridge curves (by Proposition \ref{P:face}). Hence $F_T$ is a simplicial polyhedral  face of the Mori cone $\NEb(\MM_{g,n}^{\ps})$ whose extremal rays are generated by the elliptic bridge curves of type contained in $T$. We conclude since the elliptic bridge curves correspond to the minimal subsets of $T_{g,n}$. 
Part \eqref{L:poset-T1b} follows from  Proposition \ref{P:face}\eqref{P:face3} and the fact that $F_T$ is a face of the elliptic bridge face. 
Part \eqref{L:poset-T2}: by part \eqref{L:poset-T1},  we have that $F_T\subseteq F_S$ if and only if every minimal subset of $T_{g,n}$ contained in $T$ is also contained in $S$ and this is equivalent to the inclusion $T^{\adm}\subseteq S^{\adm}$.
\end{proof}

\end{section}

\begin{section}{The moduli space of $T$-semistable curves}\label{S:stackT}

The aim of this section is to study the geometric properties of the moduli space $\MMM_{g,n}^T$ of $T$-semistable curves and of the morphism $f_T:\MM_{g,n}^{\ps}\to \MMM_{g,n}^T$. Throughout this section, we assume that $\car(k)\gg (g,n)$ (see Definition \ref{A:char}), which is needed for the existence of the good moduli space  $\MMM_{g,n}^T$. The main result of this section says that, in characteristic zero,  the morphism $f_T$ is the contraction  of the $T$-face $F_T$ (see  Definition \ref{D:T-Face}) of the Mori cone $\NEb(\MM_g^{\ps})$.

\begin{theorem}\label{thm:contr}
Let $T\subseteq T_{g,n}$ with $(g,n)\neq (2,0)$. Assume that ${\rm char}(k)=0$.
The good moduli space $\MMM_{g,n}^T$ is  projective and the morphism $f_T: \MM_{g,n}^{\ps}\to \MMM_{g,n}^T$ is the contraction of the face $F_T$.
Moreover, $f_T$ is a $K_{\MM_{g,n}^{\ps}}$-negative contraction.
\end{theorem}

The theorem is trivial true in the following cases:
\begin{itemize}
\item If $T^{\adm}=\emptyset$ (which is always the case for $g=0$ or $(g,n)=(1,1)$) then $f_T$ is  the identity by Remark \ref{R:ourstacks}. On the other hand $F_T=(0)$, and hence $\gamma_T$ is also the identity.
\item If $(g,n)=(1,2)$ and $T^{\adm}\neq \emptyset$ (in which case it must be the case that $T^{\adm}=\{[0,\{1\}], [1,\{1\}]\}$) then $f_T:\MM_{1,2}^{\ps}\to \MMM_{1,2}^T=\Spec k$  by Remark \ref{R:special-gn}. On the other hand, $F_T=\NEb(\M_{1,2}^{\ps})$ (see the discussion following Proposition \ref{P:face}) so that the contraction $\gamma_T$ of $F_T$ is the map to $\Spec k$. 
\end{itemize}

Before proving  the above Theorem, we will need a description of the  fibres of $f_T$.  

\begin{proposition}\label{P:fibers1}
Let $T\subseteq T_{g,n}$ with $(g,n)\neq (2,0)$ and $\car(k)\gg (g,n)$.
\noindent 
\begin{enumerate}[(i)]
\item \label{P:fibers1a} The projective morphism $f_T$ is a  contraction, i.e.  $(f_T)_*(\cO_{\MM_{g,n}^{\ps}})=\cO_{\MMM_{g,n}^T}$.
\item \label{P:fibers1bis} Let $B$ an integral curve inside $\M_{g,n}^{\ps}$ with associated family of pseudostable curves $\cC\to B$ and let $C$ be the image of $B$ inside $\MM_{g,n}^{\ps}$. 
Then $C$ is contracted by $f_T$ if and only if the only non-isotrivial components of the family $\cC$ are $A_1$/$A_1$-attached elliptic bridges of type contained in $T$. 

\item \label{P:fibers1b} The exceptional locus of $f_T$ is the union of the following irreducible closed subsets 
$$\Ell([\tau,I], [\tau+1, I]):=\ov{\{(C,\{p_i\})\in \MM_{g,n}^{\ps}\:  \text{having an  elliptic bridge of type }\{[\tau,I], [\tau+1, I]\}\}} $$
for every  $\{ [\tau,I], [\tau+1, I]\}\subseteq  T-\{[1,\emptyset]\}$, and 
$$\Ell(\irr):=\ov{\{(C,\{p_i\})\in \MM_{g,n}^{\ps}\:   \text{having an elliptic bridge of type }\{\irr\}\}} \quad \text{ if } \irr \in T \: \text{and }\: g\geq 2.$$

Moreover, if $(g,n)\neq (1,2)$, then all the above closed subsets have codimension two except  $\Ell([0,\{i\}], [1, \{i\}])$ which coincides with the divisors $\Delta_{1,\{i\}}$ (for any $1\leq i \leq n$). In particular, $f_T$ is always birational and it is small if and only if $T$ does not contain any subset of the form $\{[0,\{i\}],[1,\{i\}]\}$ for some $1\leq i \leq n$.
\end{enumerate}
\end{proposition} 
Note that the closed subsets $\Ell([\tau,I], [\tau+1, I])$ (resp. $\Ell(\irr)$) are covered by the elliptic bridge curves $C([\tau,I], [\tau+1, I])$ (resp. $C(\irr)$). Hence part \eqref{P:fibers1b} is  a necessary condition to have that $f_T$ is the contraction of the face $F_T$. In the case $(g,n)=(1,2)$ and $T^{\adm}=\{[0,\{1\}], [1,\{1\}]\}$, the morphism $f_T$ is the map to a point and  its the exceptional locus  is equal to $\Ell([0,\{1\}], [1, \{1\}])=\MM_{1,2}^{\ps}$.

\begin{proof}
Part \eqref{P:fibers1a} follows from the Zariski main theorem using that $f_T$ is a proper morphism between irreducible normal algebraic spaces (see Theorem \ref{T:goodspaces}) which is moreover birational since it is an isomorphism when restricted to the dense open subset of smooth curves.

Let us now prove parts \eqref{P:fibers1bis} and \eqref{P:fibers1b}. By  Proposition \ref{prop:Tclosed}\eqref{prop:Tclosed0}, the morphism $f_T$ sends a pseudo-stable curve $(C, \{p_i\})$ into the $T$-closed curve $f_T((C, \{p_i\}))$ which is obtained from $(C, \{p_i\})$ by replacing each $A_1$/$A_1$-attached elliptic bridge of type contained in $T$ by an attached rosary  of length two. The type of any $A_1$/$A_1$-attached elliptic bridge of $(C, \{p_i\})$ can be equal to $\{\irr\}$ if $\irr \in T$ and  $g\geq 2$, or
$\{ [\tau,I], [\tau+1, I]\}$ if $\{ [\tau,I], [\tau+1, I]\}\subseteq  T-\{[1,\emptyset]\}$ (because $(C, \{p_i\})$ does not have elliptic tails). This  implies part \eqref{P:fibers1bis} and that the exceptional locus of $f_T$ is equal to 
$$\bbE_T:=\bigcup_{\{ [\tau,I], [\tau+1, I]\}\subseteq  T-\{[1,\emptyset]\}} \Ell([\tau,I], [\tau+1, I]) \bigcup_{\stackrel{\irr \in T}{g\geq 2}} \Ell(\irr).$$
We conclude observing that the closed subsets  $\Ell([\tau,I], [\tau+1, I])$ and $ \Ell(\irr)$ are irreducible of the stated codimension.
\end{proof}

\begin{proof}[Proof of Theorem \ref{thm:contr}]
As observed after the statement of the Theorem, we can assume that $(g,n)\neq (1,2)$ for otherwise the Theorem is trivially true.

Since $F_T$ is a $K_{\MM_{g,n}^{\ps}}$-negative face of $\NEb(\MM_g^{\ps})$ and $\MM_g^{\ps}$ has klt singularities by Proposition \ref{P:Pic-Mgps}\eqref{P:Pic-Mgps2}, the cone theorem \cite[Thm. 3.7(3)]{KM} implies that there is a $K_{\MM_{g,n}^{\ps}}$-negative contraction of $F_T$ 
$$\gamma_T:\MM_{g,n}^{\ps}\to (\MM_{g,n}^{\ps})_{F_T}.$$

Therefore, the Theorem will follow from the  Lemma \ref{L:rigidity} if we show that an integral  curve $C\subset \MM_{g,n}^{\ps}$ is contracted by $f_T$ if and only if its class $[C]$ belongs to $F_T$.

In order to prove this, fix an integral curve $C\subset \MM_{g,n}^{\ps}$ and observe that, since $\M_{g,n}^{\ps}$ has finite inertia by Proposition \ref{P:Mgps-DM}, the curve $C$ admits a finite cover that lifts to $\M_{g,n}^{\ps}$. Hence we can find an integral curve $B\subset \M_{g,n}^{\ps}$, with associated family of pseudostable curves $\cC\to B$,  whose image in  $\MM_{g,n}^{\ps}$ is the curve $C$.
Now, Proposition \ref{P:fibers1}\eqref{P:fibers1bis} says that $C$ is contracted by $f_T$ if and only if the only non-isotrivial components of the family $\cC$ are $A_1$/$A_1$-attached elliptic bridges of type contained in $T$, which is  equivalent to the fact that $[C]$ belongs to $F_T$ by Lemma \ref{L:poset-T}\eqref{L:poset-T1b}.
\end{proof}

As a Corollary of the above Theorem and some facts that are implicit in the proof of the cone theorem, we can describe the Neron-Severi group of $\MMM_{g,n}^T$ and its nef/ample cone. 
We will need the following definition, where we freely identify  the rational Picard groups of $\M_{g,n}^T$, $\M_{g,n}^{\ps}$ and $\MM_{g,n}^{\ps}$, using Corollary \ref{C:Pic-ps} and Proposition \ref{P:Pic-Mgps}\eqref{P:Pic-Mgps2b}.

\begin{definition}\label{D:Tcomp}
A $\bbQ$-line bundle $L$ on $\M_{g,n}^T$ (or equivalently on $\MM_{g,n}^{\ps}$ or on $\M_{g,n}^{\ps}$) is said to be \emph{$T$-compatible}  if $L$ intersects to zero all the elliptic bridge curves of type contained in $T$.

Explicitly, using Lemma \ref{L:int-1str}, a $\bbQ$-line bundle  
$$L=a\lambda+b_{\irr}\delta_{\irr}+\sum_{[i,I]\in T_{g,n}-\{[1,\emptyset], \irr\}} b_{i,I} \delta_{i,I} \in \Pic(\M_{g,n}^{T})_{\bbQ}$$
 is $T$-compatible if and only if
\begin{equation}\label{E:equaT}
\begin{sis}
a+10 b_{\irr}=0 & \: \text{ if } \irr \in T, \\
a+12 b_{\irr}-b_{\tau,I}-b_{\tau+1,I}=0 & \: \text{ for any } \{[\tau, I], [\tau+1, I]\}\subset T.
\end{sis}
\end{equation}
\end{definition}

\begin{cor}\label{C:contr}
Let $T\subseteq T_{g,n}$ with $(g,n)\neq (2,0)$. Assume that ${\rm char}(k)=0$. Then 
\begin{enumerate}[(i)]
\item \label{C:contr1} The real Neron-Severi vector space $N^1(\MMM_{g,n}^T)_{\bbR}$ can be identified, via pull-back along $f_T$, with the annihilator subspace $F_T^{\perp}\subset N^1(\MM_{g,n}^{\ps})_{\bbR}.$ 
This implies that a $\bbQ$-line bundle $L$ on $\M_{g,n}^T$ descends to a (necessarily unique) $\bbQ$-line bundle on $\MMM_{g,n}^T$ (which we will denote by $L^T$) if and only if $L$ is $T$-compatible.
\item \label{C:contr2} The nef (resp. ample) cone of $\MMM_{g,n}^T$ can be identified, via pull-back along $f_T$, with the dual face $F_T^{\vee}:=F_T^{\perp}\cap \Nef(\MM_{g,n}^{\ps})$ of $F_T$ (risp. the interior of $F_T^{\vee}$).
\end{enumerate}
In particular, $F_T$ and $F_T^{\vee}$ are perfect  dual faces, i.e. $\codim F_T=\dim F_T^{\vee}$, and hence they are exposed faces, i.e. they admit supporting hyperplanes. Moreover, every $\bbQ$-line bundle on $\MM_{g,n}^{\ps}$ whose class lies in  the interior of $F_T^{\vee}$ defines a supporting hyperplane for $F_T$ and it is semiample with associated contraction equal to $f_T$.
\end{cor}
In \cite[Prop. 3.13]{CTV}, we will prove that the second assertion of \eqref{C:contr1} holds true if $\car(k)\gg (g,n)$ arguing similarly to Proposition \ref{P:lbdesc}.
\begin{proof}
Since $f_T$ is the contraction of the  $K_{\MM_{g,n}^{\ps}}$-negative face  $F_T$ by Theorem \ref{thm:contr}, it follows from \cite[Thm. 3.7(4)]{KM} that $F_T^{\perp}$ is the pull-back via $f_{T}$ of  $N^1(\MMM_{g,n}^T)_{\bbR}$, which proves the first statement in \eqref{C:contr1}. The second one follows from the first one, the left part of the commutative diagram \eqref{E:diag-spaces} and Proposition \ref{P:Pic-Mgps}\eqref{P:Pic-Mgps2b}.

Next, since $F_T$ is a $K_{\MM_{g,n}^{\ps}}$-negative face of $\NEb(\MM_{g,n}^{\ps})$, it follows from Step 6 of the proof of \cite[Thm. 3.15]{KM} that $F_T$ is an exposed face. Hence any $\bbQ$-line bundle $L$ which is in the relative interior of $F_T^\vee$ is a supporting hyperplane for $F_T$  and conversely. 
Moreover, it follows  from  the basepoint-free theorem (see Step 7 of the proof of \cite[Thm. 3.15]{KM}) that any $\bbQ$-line bundle $L$ which is a supporting hyperplane for $F_T$ is semiample and the morphism associated to $|mL|$ (for $m\gg 0$) is $f_T$.  In particular, it follows that the relative interior of $F_T^{\vee}$ is the pull-back via $f_T$ of the ample cone of $\MMM_{g,n}^T$ and, by taking the closures, we get that $F_T^{\vee}$ is the pull-back via $f_T$ of the nef cone of $\MMM_{g,n}^T$, which proves \eqref{C:contr2}.

Finally, the last part of the Corollary \ref{C:contr} follows from what already proved and the equalities
$$\codim F_T=\dim N^1(\MMM_{g,n}^T)_{\bbR}=\dim F_T^{\vee},$$
where we have used \cite[Rmk. 7.40]{Deb} for the first equality and the fact that the nef cone is a full-dimensional cone in the real Neron-Severi vector space for the second equality.
\end{proof}

Note that the characteristic zero assumption is used in the proof of Theorem \ref{thm:contr} only to establish the projectivity of $\MMM_{g,n}^T$. There is a special case, however, where we can prove the projectivity in arbitrary characteristic (provided that it is large enough so that $\MMM_{g,n}^T$ exists).

\begin{example}\label{R:LMMP}
If $T=T_{g,n}$ (and $(g,n)\neq  (2,0)$) then the above Theorem \ref{thm:contr} is true for $\car(k)\gg (g,n)$ and it can be proved as it follows. 
From the proof of Proposition \ref{P:face}, it follows that the $\bbQ$-line bundle  on $\MM_{g,n}^{\ps}$
$$N_{g,n}:=K_{\M_{g,n}^{\ps}}+\frac{7}{10}\delta+\frac{3}{10}\psi=\frac{13}{10}(10\lambda-\delta+\psi)$$
is semiample  and its dual face in $\NE(\M_{g,n}^{\ps})$ is the elliptic bridge face (note that this is true also for $(g,n)=(1,2)$, in which case $N_{1,2}=0$ and the elliptic bridge face coincides with the entire effective cone of curves of $\MM_{1,2}^{\ps}$). 
Hence a sufficiently high multiple of $N_{g,n}$ induces a morphism
$$\psi: \MM_{g,n}^{\ps}\to \Proj \bigoplus_{m\geq 0} H^0\left(\MM_{g,n}^{\ps}, \lfloor m N_{g,n}\rfloor\right) $$
which is  the contraction of the elliptic bridge face and whose codomain coincides with  $\displaystyle \MM_{g,n}\left(\frac{7}{10} \right)$ by \cite[Prop. 7.2]{AFS3}.
Since the $f_{T_{g,n}}$-relative effective cone $\NE(f_{T_{g,n}})$ of curves  is equal to the elliptic bridge face (see Proposition \ref{P:fibers1}\eqref{P:fibers1bis}), Lemma  \ref{L:rigidity} implies that we have an isomorphism 
\begin{equation}\label{E:LMMP}
\MMM_{g,n}^{T_{g,n}}\cong \MM_{g,n}(7/10 ),
\end{equation}
under which $f_{T_{g,n}}$ gets identified with $\psi$.
Note that \eqref{E:LMMP} is a special case (if $\car(k)=0$) of \cite[Thm. 1.1]{AFS3}, and it was previously proved by Hassett-Hyeon \cite{HH2} for $n=0$.

From the above discussion and  Remark \ref{R:bridgecurves}, we can compute the Picard number of $\MMM_{g,n}^{T_{g,n}}\cong \MM_{g,n}(\frac{7}{10})$ and the relative Picard number of $f_{T_{g,n}}$ (assuming that $g\geq 1$, for otherwise we have that $\MM_{0,n}(\frac{7}{10})=\MM_{0,n}$):
\begin{enumerate}
\item \label{C:PicT1} The Picard number of $\MM_{g,n}(\frac{7}{10})$ is equal to 
$$\dim_{\bbQ} \Pic\left(\MM_{g,n}\left(\frac{7}{10}\right)\right)_{\bbQ}=
\begin{cases}
1 & \text{ if } n=0 \text{ and } g\geq 3 \: \text{ is odd, }\\
2 & \text{ if } n=0 \text{ and } g\geq 4 \: \text{ is even, }\\
2^{n-1}+1-\delta_{2,g}-(n+1)\delta_{1,g} & \text{ if } g\geq 1 \: \text{ and } n\geq 1.
\end{cases}
$$
\item \label{C:PicT2} The relative Picard number of $f_{T_{g,n}}$ is equal to 
$$\rho(f_{T_{g,n}})=
\begin{cases}
\frac{g-1}{2} & \text{ if } n=0 \text{ and } g\geq 3 \: \text{ is odd, }\\
\frac{g}{2}-1 & \text{ if } n=0 \text{ and } g\geq 4 \: \text{ is even, }\\
g 2^{n-1}-1 & \text{ if } g\geq 1 \: \text{ and } n\geq 1.
\end{cases}
$$

\end{enumerate}

\end{example}

In \cite{CTV}, we study several geometric properties of the space $\MMM_{g,n}^T$ and of the morphism $f_T$.  For completeness, we mention those results here. 
We will need the following 

\begin{definition}\label{D:Tdiv}
Given a subset $T\subseteq T_{g,n}$, we define the \emph{divisorial part} of $T$ as the (possibly  empty) subset $T^{\div}\subset T$ defined by  
$$T^{\div}:=
\begin{cases}
\emptyset & \text{ if } (g,n)=(1,1) \, \text{ or } (2,1), \\
 \{\{[0,\{i\}],[1,\{i\}]\}: \:  \{[0,\{i\}],[1,\{i\}]\}\subset T\} & \text{ otherwise. } 
\end{cases}
$$
\end{definition}
It is easily checked that $T^{\div}$ is admissible in the sense of Definition \ref{def:admissible}.

\begin{proposition}(\cite[Prop. 3.16, 3.17]{CTV})\label{P:geomT}
Assume that $(g,n)\neq (2,0)$, $\car(k)\gg (g,n)$, and let $T\subseteq T_{g,n}$. 
\begin{enumerate}
\item \label{P:geomT1} The following conditions are equivalent:
\begin{enumerate}[(i)]
\item $\MMM_{g,n}^T$ is $\bbQ$-factorial.
\item  $\MMM_{g,n}^T$ is $\bbQ$-Gorenstein.
\item  $T^{\adm}=T^{\div}$.
\end{enumerate}

\item \label{P:geomT2} The morphism $f_T:\MM_{g,n}^{\ps}\to \MMM_{g,n}^T$ can be factorised as follows 
\begin{equation}\label{E:fact-fT}
f_T\colon \MM_{g,n}^{\ps}\xrightarrow{f_{T^{\div}}}\MMM_{g,n}^{T^{\div}}\xrightarrow{\sigma_T} \MMM_{g,n}^T
\end{equation}
in such a way that 
\begin{enumerate}[(i)]
\item The morphism $f_{T^{\div}}$ is a composition of $\frac{1}{2}|T^{\div}|$ divisorial contractions,
each one of them having the relative Mori cone generated by a $K$-negative extremal ray. 
\item  The algebraic space $\MMM_{g,n}^{T^{\div}}$ is $\bbQ$-factorial and, if $\car(k)=0$, klt.

\item  The morphism $\sigma_T$ is a small contraction.
\item  \label{P:geomT2iv}
The relative Mori cone of $\sigma_T$ is  a $K_{\MMM_{g,n}^{T^{\div}}}$-negative face if and only if $T$ does not contain subsets of the form $\{[0,\{j\}],[1,\{j\}],[2,\{j\}]\}$ for some $j\in [n]$ or  $(g,n)=(3,1), (3,2), (2,2)$. 
\end{enumerate}
\end{enumerate}
\end{proposition}
Note that, if $\car(k)=0$, then all the spaces appearing in \eqref{E:fact-fT} are projective varieties, and hence $f_{T^{\div}}$ is the composition of divisorial contractions of $K$-negative rays while $\sigma_T$ is a small contraction of a $K$-negative face if and only if the condition on $T$ appearing in \eqref{P:geomT2iv} is satisfied.

\end{section}

\section{The moduli space of $T^+$-semistable curves}\label{S:T+}

The aim of this section (throughout which, we assume that $\car(k)\gg (g,n)$, see Definition \ref{A:char}) is to describe the map $f_T^+:\MMM_{g,n}^{T+}\to \MMM_{g,n}^{T}$  in terms of the Minimal Model Program (MMP).  In particular, we will describe $f_T^+$ as the flip of $f_T$ with respect to suitable $\bbQ$-line bundles.

\subsection{Preliminaries definitions and results about flips}
\begin{definition}\label{D:Dflip}
Let $f:X\to Y$ be a proper morphism between normal algebraic spaces of finite type over $k$ and let $D$ be an $f$-antiample $\bbQ$-Cartier $\bbQ$-divisor on $X$. A \emph{$D$-flip of $f$} is a  proper morphism $f_D^+:X_D^+\to Y$ of algebraic spaces fitting into the commutative diagram 
\begin{equation}\label{E:defflip}
\xymatrix{ & X \ar @{-->}[rr]^{\eta}  \ar[rd]_{f} && X_D^+ \ar[ld]^{f_D^+} \\  & & Y}
\end{equation}
where $\eta$ is a rational map, and such that 
\begin{enumerate}[(i)]
\item the algebraic space $X_D^+$ (which is automatically of finite type over $k$) is normal;
\item the morphism $f_D^+$ is a small contraction, i.e. it is a contraction whose exceptional locus  $\Exc(f_D^+)$ has codimension at least two;
\item  the $\bbQ$-divisor $D^{+}:=\eta_*(D)$ is $\bbQ$-Cartier and $f_D^+$-ample.
\end{enumerate}

A $D$-flip is called \emph{elementary} if $f$ has relative Picard number 1.
\end{definition}

The difference between Definition \ref{D:Dflip} and the classical definition of flip is that we do not require the map $f$ to be small.

\begin{remark} Assume that $f$ is birational. Then, since $f_D^+$ is small, we have that $\eta^{-1}$ does not contract any divisor, i.e.\ in the terminology of \cite[Page 424]{BCHM10} it is a birational contraction.
Moreover, the map $\eta$ is $D$-non-positive in the sense \cite[Def. 3.6.1]{BCHM10} and so $\eta$ is the ample model of $D$ over $Y$ (see \cite[Def. 3.6.5]{BCHM10}).
	 
In \cite[Definition 11]{AK} a diagram analogous to \ref{E:defflip} is called an MMP-step. 
\end{remark}

We discuss the existence and uniqueness of flips in the following result. The proof is standard, we include it for completeness.

\begin{lemma}\label{L:Dflip}
Let $f:X\to Y$ be a proper morphism of normal algebraic spaces of finite type over $k$ and let $D$ be an $f$-antiample $\bbQ$-Cartier $\bbQ$-divisor on $X$.
\begin{enumerate}[(i)]
\item \label{L:Dflip1} If the $D$-flip  of $f$ exists, then it is given by 
\begin{equation}\label{E:equflip}
f_D^+:X_D^+=\Proj \bigoplus_{m\geq 0} \cO_{Y}(\lfloor m f_*(D)\rfloor) \to Y.
\end{equation}
In particular, the $D$-flip of $f$ is unique. 

Moreover, the $D$-flip depends only on the $\bbQ$-line bundle $L=\cO_X(D)$ associated to $D$ and hence it will also be denoted by $f_L:X_L^+\to Y$ and called the \emph{$L$-flip} of $f$.
\item \label{L:Dflip2}
If ${\rm char}(k)=0$, $X$ is klt and $K_X$ is $f$-antiample, then the coherent sheaf $\bigoplus_{m\geq 0} \cO_{Y}(\lfloor m f_*(D)\rfloor)$ of $\cO_{Y}$-algebras is finitely generated, hence the $D$-flip of $f$ exists.
\end{enumerate}
\end{lemma}
\begin{proof}
Part \eqref{L:Dflip1}: suppose that the $D$-flip $f_D^+:X_D^+\to Y$ exists. 
Since  $D^+$ is $\bbQ$-Cartier and $f_D^+$-ample,
we have that 
$$X_D^+=\Proj_{Y}  \bigoplus_{m\geq 0}(f_D^+)_*(\lfloor mD^+\rfloor).$$
Since $X_D^+$ is normal and  the morphism $f_D^+$ is a small contraction, arguing  as in  the proof of \cite[Lemma 6.2]{KM} and using that $(f_D^+)_*(D^+)=(f_D^+)_*(\nu_*(D))=f_*(D)$ because of the commutativity of the diagram \eqref{E:defflip}, we have the equality of $\cO_{Y}$-algebras
$$ \bigoplus_{m\geq 0}(f_D^+)_*(\lfloor mD^+\rfloor)=\bigoplus_{m\geq 0} \cO_{Y}(\lfloor m (f_D^+)_*(D^+)\rfloor) =\bigoplus_{m\geq 0} \cO_{Y}(\lfloor m f_*(D)\rfloor) .$$
This concludes the proof of the first part \eqref{L:Dflip1}. The second part follows from the fact that the pushforward of divisors respects the linear equivalence of divisors.

Part \eqref{L:Dflip2}: by \cite[Corollary 4.5]{Fujino99} there exists an effective $\bbQ$-divisor $\Delta$ on $Y$ such that $(Y,\Delta)$ is klt. 
Hence we conclude applying \cite[Thm. 92]{Exercises}, which is a consequence of \cite{BCHM10} and says that the coherent sheaf $\bigoplus_{m\geq 0} \cO_{Y}(\lfloor m f_*(D)\rfloor)$ of $\cO_{Y}$-algebras is finitely generated.
\end{proof}

\subsection{Main results about $f_T^+$ and $\MM_{g,n}^{T+}$}

 The following theorem, which is the main result of this section, describes the morphism $f_T^+$ as the flip of $f_T$ with respect to suitable $\bbQ$-line bundles.

\begin{theorem}\label{T:flip+}
Assume that $(g,n)\neq (2,0), (1,2)$, $\car(k)\gg (g,n)$, and let $T\subseteq T_{g,n}$. 
Let  $L\in \Pic(\MM_{g,n}^{\ps})_{\bbQ}=\Pic(\M_{g,n}^{\ps})_{\bbQ}$. Then $f_T^+$ is the $L$-flip of $f_T$ if and only if 
$L$ is $f_T$-antiample and the restriction of $L$ to $\M_{g,n}^{T+}$ is $T^+$-compatible (see Definition \ref{D:T+comp}).

\end{theorem} 

The special cases $(g,n)=(1,2)$ and $(2,0)$ are discussed in Remark \ref{R:special-gn}.

The proof of the above theorem will be the outcome of several  propositions, that are interesting in their own. We first describe the rational Picard group of $\MMM_{g,n}^{T+}$.  Recall the description of the rational Picard group of $\M_{g,n}^{T+}$ given in Corollary \ref{C:Pic-ps}. 

\begin{definition}\label{D:T+comp}
A $\bbQ$-line bundle on $\M_{g,n}^{T+}$ 
\begin{equation}\label{E:lbT+}
L=a\lambda+b_{\irr}\delta_{\irr}+\sum_{[i,I]\in T_{g,n}-\{[1,\emptyset], \bigcup_j [1,\{j\}], \irr\}} b_{i,I} \delta_{i,I}
\end{equation}
is said to be \emph{$T^+$-compatible} if $b_{\tau,I}=b_{\tau+2,I}$ for any pair $\{[\tau,I],[\tau+2,I]\}\subset T_{g,n}$ such that 
\begin{equation}\label{E:cond-pair}
\{[\tau,I],[\tau+1,I],[\tau+2,I]\}\subset T \text{ and  }  [\tau,I],[\tau+2,I]\not\in\{ [1,\emptyset], \bigcup_j [1,\{j\}]\}.
\end{equation}
\end{definition}

 \begin{remark}
 If a $\bbQ$-line bundle on $\M_{g,n}^{T}$ is $T$-compatible (see Definition \ref{D:Tcomp}) then its restriction to $\M_{g,n}^{T+}$ is $T^+$-compatible. 
 This can be proven by direct inspection. Alternatively,  it also follows from the fact that  $T$-compatible $\bbQ$-line bundles are exactly $\bbQ$-line bundles on $\MMM_{g,n}^T$ by Corollary \ref{C:contr}\eqref{C:contr1} while $T^+$ compatible $\bbQ$-line bundles are exactly the $\bbQ$-line bundles on $\MMM_{g,n}^{T+}$ by Proposition \ref{P:lbdesc} below, and one can pull-back line bundles via the map $f_T^+\colon \MMM_{g,n}^T \to \MMM_{g,n}^{T+}$. 
 
 \end{remark}

\begin{proposition}\label{P:lbdesc}
Assume that $(g,n)\neq (2,0), (1,2)$ and $\car(k)\gg (g,n)$.
A $\bbQ$-line bundle $L$ on $\M_{g,n}^{T+}$ descends to a (necessarily unique) $\bbQ$-line bundle on $\MMM_{g,n}^{T+}$ (which we will denote by $L^{T+}$) if and only if $L$ is $T^+$-compatible.
\end{proposition}

\begin{proof}

Up to passing to a multiple, it is enough to prove the statement for a line bundle on $\M_{g,n}^{T+}$. Given such a line bundle $L$ on $\M_{g,n}^{T,+}$ and any one parameter subgroup $\rho:\Gm\to \Aut(C,\{p_i\})$ for some $k$-point $(C,\{p_i\})\in \M_{g,n}^{T+}(k)$, the group $\Gm$ will act  via $\rho$ onto the fibre $L_{(C, \{p_i\})}$ of the line bundle over $(C,\{p_i\})$ and we will denote by  $\langle L, \rho\rangle\in \bbZ$ the weight of this action. According to \cite[Theorem 10.3]{Alper} applied to the good moduli space $\phi^{T,+}: \M_{g,n}^{T+}\to \MMM_{g,n}^{T+}$,  the line bundle $L$ descends to a $\bbQ$-line bundle on $\MMM_{g,n}^{T+}$ if and only if $\langle L, \rho\rangle=0$ for any  one parameter subgroup $\rho:\Gm\to \Aut(C,\{p_i\})$ of any closed $k$-point $(C,\{p_i\})\in \M_{g,n}^{T+}(k)$. We will now show that this is the case if and only if $L$ is $T^{+}$-compatible.

To prove the if implication, assume that $L$ is $T^+$-compatible and fix a closed $k$-point $(C,\{p_i\})$ of $\M_{g,n}^{T+}(k)$. By Proposition \ref{P:T+closed}, either $(C, \{p_i\})$ is a closed rosary, and in this case the result follows from Lemma \ref{L:weight+}(\ref{L:weight+2}),  or it admits a $T^+$-canonical decomposition $C=K \cup (R_1,q_1^1,q_2^1) \cup \cdots \cup (R_r,q^r_1,q^r_2)$, where $R_i$ is a rosary of length $3$. In the second case,  the connected component of the identity of $\Aut(C,\{p_i\})$ is isomorphic to $\Pi_{i=1}^r\Aut(R_i,q_1^i,q_2^i)\cong  \Gm^{\times r}$, hence it is enough to show that $\langle L, \rho_i\rangle=0$ for $i=1,\dots ,r$, where $\rho_i$ is an isomorphism between $\Gm$ and $\Aut(R_i,q_1^i,q_2^i)$. The result now follows from Lemma \ref{L:weight+}(\ref{L:weight+1}).

To prove the converse direction, remark that for each triple as in Equation \eqref{E:cond-pair}, there exists a $T^+$-closed curve with an attached rosary of length $3$ and type $\{[\tau,I],[\tau+1,I],[\tau+2,I]\}$;  denote by $\rho$ the 1PS associated to this rosary. The necessary condition $\langle L, \rho\rangle=0$ implies, because of Lemma \ref{L:weight+}(\ref{L:weight+1}), that $b_{\tau,I}=b_{\tau+2,I}$.
\end{proof}

\begin{lemma}\label{L:weight+}
Assume that $\car(k)\neq 2$.
Consider a line bundle $L$ on $\M_{g,n}^{T+}$ written as in \eqref{E:lbT+}.

\begin{enumerate}[(i)]
\item \label{L:weight+1} Let  $(C,\{p_i\})$ be a $k$-point of $\M_{g,n}^{T+}(k)$ that has an  attached rosary $(R,q_1, q_2)$ of length $3$ and consider the one parameter subgroup $\rho_R:\Gm\stackrel{\cong}{\longrightarrow}\Aut((R,q_1,q_2))^o\subset \Aut((C, \{p_i\}))$ normalised so that 
$\wei_{\rho_R}(T_{q_1}(R))=1$. Then we have
 $$\langle L,  \rho_R\rangle=
 \begin{cases}
 0 & \text{ if } \type(R,q_1,q_2)=\{\irr\}, \\
-b_{\tau,I}+b_{\tau+2,I} & \text{ if } \type(R,q_1,q_2)=\{[\tau,I],[\tau+1,I], [\tau+2,I]\}.
 \end{cases}
 $$
 \item \label{L:weight+2} Let $R\in \M_{r+1,0}^{T+}(k)$ be a closed rosary of even length $r$ (which can occur only if $\irr\in T$) and consider the one parameter subgroup $\rho_R:\Gm\stackrel{\cong}{\longrightarrow}\Aut(R)^o.$
 Then we have that 
 $$\langle L,  \rho_R\rangle=0.$$
 \end{enumerate}
\end{lemma}
\begin{proof}
Let us first prove part \eqref{L:weight+1}.  Since the weight is linear in $L$, the result will follow from the following identities:
\begin{equation}\label{E:weipiece+}
\begin{sis}
& \langle \lambda, \rho_R \rangle=0,\\
& \langle \delta_{\irr}, \rho_R\rangle=0, \\
& \langle \delta_{i,I},\rho_R\rangle =
\begin{cases}
  -1 & \text{ if } \type(R,q_1,q_2)=\{[i,I], [i+1,I], [i+2,I]\}, \\
 1 & \text{ if } \type(R,q_1,q_2)=\{[i-2,I], [i-1,I], [i,I]\},\\
 0& \text{otherwise.}
\end{cases}\\
\end{sis}
\end{equation}
The above identities can be proved by adapting the computations in \cite{AFS0}, as we now explain.

To compute the weights of the $\psi$ classes, recall that the fibre of $\psi_i$ over a pointed curve $(C,\{p_i\}) $ is canonically isomorphic to the cotangent vector space $T_{p_i}(C)^{\vee}$. Hence, $\langle \psi_i, \rho_R\rangle$ is the weight of the action of $\Gm$, via the one parameter subgroup $\rho_R$,  on the $1$-dimensional $k$-vector space $T_{p_i}(C)^{\vee}$.  This is not trivial if and only if $p_i$ is either $q_1$ or $q_2$, and it is computed in Remark \ref{R:ros-eq}.

To compute the other weights, we first make the following key remark. The $\Gm$-action on $(R, q_1, q_2)$, which is explicitly described in  Remark \ref{R:ros-eq},  is such that the weights of $\Gm$ on the coordinates $(x_1,y_1)$ that define the first tacnode $t_1:=\{y_1^2-x_1^4=0\}$ are opposite to the weights of $\Gm$ on the coordinates $(x_2,y_2)$
 that define the second tacnode $t_2:=\{y_2^2-x_2^4=0\}$. This will  imply that the contributions that come from the two tacnodes cancel out.

In order to compute the other contributions, consider the formally smooth morphism 
$$\Phi:\Def(C,\{p_i\})\longrightarrow \Def(\wh{O}_{C, t_1})\times \Def(\wh{O}_{C, t_2})\times \prod_{q_i \: \text{node}}\Def(\wh{O}_{C,q_i}),$$
into the product of the (formal) semiuniversal deformation spaces of the two tacnodes $a_1$ and $a_2$ of $R$, and of  nodes belonging to $\{q_1, q_2\}$. The group $\Aut(R, q_1, q_2)^o\cong \Gm$ acts on the above deformation spaces in such a way that the morphism $\Phi$ is equivariant. 

Let us know write down explicitly the deformation spaces of the above singularities together with the action of $\Gm$, using the equation given in Remark \ref{R:ros-eq}. The  semiuniversal deformation space of $q_i$ (for $i=1,2$), whenever it is a node, is equal to $\Spf k[b_i]$ and the semiuniversal deformation family is  $n_i z_i=b_i$ where $z_i$ is a local coordinate on the branch of the node $q_i$ not belonging to $R$. The action of $\Gm$ is given by $t\cdot (b_i)=(tb_i)$. The locus of singular deformations of the node $q_i$ is cut out by the equation $\{b_i=0\}$, which has $\Gm$-weight one. 

On the other hand, the semiuniversal deformation space of the tacnode $t_i$ is equal to  $\Def(\wh{O}_{C, p})\cong \Spf k[a_2,a_1,a_0]$ and the semiuniversal deformation family is given by $y^2=x^4+a_2x^2+a_1x+a_0$. This forces the action of $\Gm$ to be given by $t\cdot (a_2,a_1,a_0)=(t^{-2}a_2, t^{-3}a_1, t^{-4}a_0)$. The locus of singular deformations of $p$ is cut out  in  $\Def(\wh{O}_{C, p})$ by the equation $\{\Delta=0\}$, where $\Delta:=\Delta(a_2,a_1,a_0)$ is the discriminant of the polynomial $x^4+a_2x^2+a_1x+a_0$. Since the discriminant is a homogeneous polynomial of degree 12 in the roots of the above polynomial and $\Gm$ acts on the roots with weight $-1$ (the same weight of $x$), it follows that $\Gm$ acts on the discriminant associated to $t_1$ with weights $-12$, and $+12$ on the discriminant associated to $t_2$.

If both point $q_i$ are nodes, it follows from the above discussion that the only boundary divisor of $\M_{g,n}^T$ that can have a non-zero weight against $\rho_R$ is the one whose equation on $\Def(C,\{p_i\})$ is given by $ \Phi^*(b_1b_2)=0$. This divisor is $2\delta_{\irr}$ if $\type(R,q_1,q_2)=\irr$, and  $\delta_{i, I}+ \delta_{g-2-i,I^c}$ if $\type(R,q_1,q_2)=\{[i,I],[i+1,I], [i+2,I]\}$. The result now  follows from \cite[Lemma 3.11]{AFS0} and Remark \ref{R:ros-eq}. If one of  the $q_i$ is a node and the other a marked point, the result follows by combining the above discussion with argument about $\psi$-classes. When $(g,n)=(2,2)$, it could be that both $q_i$'s are marked points, in this case the argument about $\psi$-classes is enough.

To compute the weight of $\lambda$, by combining \cite[Cor. 3.3]{AFS0} and the computations in \cite[Sec. 3.1.3]{AFS0} for $A_3$, we deduce that $\langle \lambda, \rho_R \rangle=0$, as we get $+1$ from one tacnode, and $-1$ for the other tacnode.

Part \eqref{L:weight+2} can be proven in a similar way, the key remark is that since the length of the rosary is even, all contributions cancel out.
\end{proof}

As a corollary, we can now determine when $\MMM_{g,n}^{T+}$ is $\bbQ$-factorial or $\bbQ$-Gorenstein.

\begin{cor}\label{C:Gor-fact}
Assume that $(g,n)\neq (2,0), (1,2)$, $\car(k)\gg (g,n)$, and let $T\subseteq T_{g,n}$. Then:
\begin{enumerate}[(i)]
\item \label{C:Gor-fact0} If $(g,n)\neq (2,1)$ or $(3,0)$ then the pull-back of the (Weil) divisor $K_{\MMM_{g,n}^{T+}}$  via the morphism $\phi^{T+}:\M_{g,n}^{T+}\to \MMM_{g,n}^{T+}$ is equal to 
\begin{equation}\label{E:canT+}
(\phi^{T+})^*(K_{\MMM_{g,n}^{T+}})=K_{\M_{g,n}^{T+}}=13\lambda-2\delta+\psi.
\end{equation}
\item \label{C:Gor-fact1} $\MMM_{g,n}^{T+}$ is $\bbQ$-factorial if and only if $T$ does not contain subsets of the form $\{[\tau,I],[\tau+1,I],[\tau+2,I]\}$ with $ [\tau,I],[\tau+2,I]\not\in\{ [1,\emptyset], \bigcup_j [1,\{j\}]\}$ and 
$ [\tau,I]\neq [\tau+2,I]$.
\item \label{C:Gor-fact2} $\MMM_{g,n}^{T+}$ is $\bbQ$-Gorenstein if and only if $T$ does not contain subsets of the form $\{[0,\{j\}],[1,\{j\}],[2,\{j\}]\}$ for some $j\in [n]$, or  $(g,n)=(3,1), (3,2), (2,2)$  
\end{enumerate}
\end{cor}
Note the following special cases:
\begin{itemize}
\item  if $T^{\adm}$ is minimal (in the sense of Definition \ref{def:admissible}) or $T^{\adm}=T^{\div}$ (see Definition \ref{D:Tdiv})  then $\MMM_{g,n}^{T+}$ is $\bbQ$-factorial;
\item If $g=1$ then $\MMM_{g,n}^{T+}$ is $\bbQ$-factorial for any $T\subseteq T_{1,n}$;
\item if $n=0$ then $\MMM_{g,n}^{T+}$ is $\bbQ$-Gorenstein for any $T\subseteq T_{g,0}$.
\end{itemize}
\begin{proof}

Part \eqref{C:Gor-fact0}: under the assumptions on the pair $(g,n)$,  the morphism $\phi^{T+}:\M_{g,n}^{T+}\to \MMM_{g,n}^{T+}$ is an isomorphism in codimension one  when restricted to the open substack ${\mathcal M}_{g, n}$ of smooth curves (see \cite[Chap. XII, Prop. 2.15]{GAC2}). Moreover, the generic point in each boundary divisor of $\M_{g,n}^{T+}$ does not have any non-trivial automorphisms and it is $T^+$-closed (see Definition \ref{def:T+closed}), and hence it is a closed point of the stack $\M_{g,n}^{T+}$. This implies that the morphism $\phi^{T+}$ is an isomorphism in codimension one, which implies that $(\phi^{T+})^*(K_{\MMM_{g,n}^{T+}})=K_{\M_{g,n}^{T+}}$. We now conclude using Mumford's formula (see Fact \ref{F:PicU}\eqref{F:PicU2}).

Part \eqref{C:Gor-fact1}: by the above discussion, the morphism $\phi^{T+}: \M_{g,n}^{T+}\to \MM_{g,n}^{T+}$ is an isomorphism in codimension one. Hence the pull-back map via the morphism $\phi^{T+}$ induces an isomorphism on the divisor class groups
$$(\phi^{T+})^*:\Cl(\MMM_{g,n}^{T+})_{\bbQ}\stackrel{\cong}{\longrightarrow} \Cl(\M_{g,n}^{T+})_{\bbQ}=\Pic(\M_{g,n}^T)_{\bbQ},$$
where in the last equality we used that $\M_{g,n}^T$ is a smooth stack. Hence, Proposition \ref{P:lbdesc} implies that $\MMM_{g,n}^{T+}$ is $\bbQ$-factorial, i.e. $\Pic(\MMM_{g,n}^{T+})_{\bbQ}=\Cl(\MMM_{g,n}^{T+})_{\bbQ}$, if and only if any $\bbQ$-line bundle on $\M_{g,n}^{T+}$ is $T^+$-compatible. An inspection of Definition \ref{D:T+comp} gives the result.

Part \eqref{C:Gor-fact2}: first of all, in the special cases $(g,n)= (2,1)$ or $(3,0)$, it is easy to check, using part \eqref{C:Gor-fact1}, that $\M_{g,n}^T$ is $\bbQ$-factorial for any $T$. Hence we can assume that $(g,n)\neq (2,1)$ or $(3,0)$, which implies that formula \eqref{E:canT+} for $(\phi^{T+})^*(K_{\MMM_{g,n}^{T+}})$ holds true. By Proposition \ref{P:lbdesc}, $\MMM_{g,n}^{T+}$ is $\bbQ$-Gorenstein if and only if 
$$13\lambda-2\delta+\psi=13\lambda-2\delta_{\irr}-2\sum_{[i,I]\not\in \{[1,\emptyset], \bigcup_j [1,\{j\}], \bigcup_j [0,\{j\}]\}}  \delta_{i,I}- \sum_{j=1}^n \delta_{0,\{j\}}$$ 
is $T^+$-compatible. An inspection of Definition \ref{D:T+comp} gives the result.
\end{proof}

\begin{remark}\label{R:QfactHK}
	It follows from Corollary \ref{C:Gor-fact} that  the algebraic space   $\MMM_{g,n}^{T_{g,n}+}$  is:
	\begin{itemize}
	\item $\bbQ$-factorial if and only if $g\leq 1$, or $(g,n)=(2,1), (3,0), (3,1), (3,2), (4,0), (5,0), (6,0)$.
	\item $\bbQ$-Gorenstein if and only if $g\leq 1$ or $n=0$ or $(g,n)=(2,1), (2,2), (3,1), (3,2)$.
	\end{itemize}
	In particular, we recover the result of Alper-Hyeon \cite[Sec. 6]{AH}: $\MMM_{g}^{T_{g}+}$ (which coincides with $\MM_{g}(\frac{7}{10}-\epsilon)$ if $\car(k)=0$, see Remark \ref{R:LMMP+}) is $\bbQ$-factorial if and only if $g\leq 6$.
	
Note that when 	$\MMM_{g,n}^{T_{g,n}+}$  is not $\bbQ$-factorial then it cannot be reached via a sequence of elementary steps (i.e.\ relative Picard number 1 steps) of an MMP of $\MM_{g,n}$. This shows that there is a difference between flipping the elliptic bridge face in one single step and trying to flip each extremal ray one by one. 

\end{remark}

Another  corollary of the above Proposition \ref{P:lbdesc} is the  computation of the Picard number of $\MMM_{g,n}^{T_{g,n}+}$ (which coincides with $\MM_{g,n}(\frac{7}{10}-\epsilon)$ if $\car(k)=0$, see Remark \ref{R:LMMP+}) and of the relative Picard number of the morphism $f_{T_{g,n}}^+$ (using Remark \ref{R:LMMP}). We assume that $g\geq 1$, for otherwise we have that $\MMM_{0,n}^{T_{0,n}+}=\MM_{0,n}$.

\begin{cor}\label{C:PicT+}
Assume that $g\geq 1$, $\car(k)\gg (g,n)$, and that $(g,n)\neq (2,0), (1,2)$. 
\begin{enumerate}[(i)]
\item \label{C:PicT+1} The Picard number of $\MMM_{g,n}^{T_{g,n}+}$ is equal to 
$$\dim_{\bbQ} \Pic\left(\MMM_{g,n}^{T_{g,n}+}\right)_{\bbQ}=
\begin{cases}
3-\delta_{3,g} & \text{ if } n=0 \text{ and } g\geq 3 \: \text{ is odd, }\\
4-\delta_{4,g} & \text{ if } n=0 \text{ and } g\geq 4 \: \text{ is even, }\\
2^{n}+2-(n+2)\delta_{2,g}-(2n+2)\delta_{1,g} & \text{ if } g\geq 1 \: \text{ and } n\geq 1.
\end{cases}
$$
\item  \label{C:PicT+2} The relative Picard number of $f_{T_{g,n}}^+$ is equal to 
$$\rho(f_{T_{g,n}}^+)=
\begin{cases}
2-\delta_{3,g} & \text{ if } n=0 \text{ and } g\geq 3 \: \text{ is odd, }\\
2-\delta_{4,g} & \text{ if } n=0 \text{ and } g\geq 4 \: \text{ is even, }\\
2^{n-1}+1-(n+1)\delta_{2,g}-(n+1)\delta_{1,g} & \text{ if } g\geq 1 \: \text{ and } n\geq 1.
\end{cases}
$$
\end{enumerate}
\end{cor}

We now show that  $f_T^+$ is projective  by producing  an $f_T^+$- ample line bundle on $\MMM_{g,n}^{T+}$.

\begin{proposition}\label{P:relample+}
Assume that $(g,n)\neq (2,0), (1,2)$ and $\car(k)\gg (g,n)$. The line bundle $-\wh{\delta}=-(\delta-\psi)$ on $\M_{g,n}^{T,+}$ descends to an $f_T^+$-ample  $\bbQ$-line bundle $(-\wh{\delta})^{T+}$ on  $\MMM_{g,n}^{T,+}$.

In particular, the morphism $f_T^+$ is projective.
\end{proposition}
\begin{proof}
The fact that $-\wh{\delta}\in \Pic(\M_{g,n}^{T+})$ descends to a  $\bbQ$-line bundle $(-\wh{\delta})^{T+}$ on $\MMM_{g,n}^{T+}$ follows from Proposition \ref{P:lbdesc}. The fact that $(-\wh{\delta})^{T+}$ is $f_T^+$-ample follows from the same argument of \cite[Prop. 7.4]{AFS3} using that the open embeddings 
$$\M_{g,n}^{\ps}\hookrightarrow \M_{g,n}^T\hookleftarrow \M_{g,n}^{T+}$$
arise from local VGIT with respect to the line bundle $\wh \delta$ on $\M_{g,n}^T$ by Proposition \ref{P:VGIT}.
\end{proof}

\begin{cor}\label{C:projT+}
Assume that $(g,n)\neq (2,0), (1,2)$ and  ${\rm char}(k)=0$. Then $\MMM_{g,n}^{T+}$ is projective.
\end{cor}
\begin{proof}
$\MMM_{g,n}^T$ is projective if ${\rm char}(k)=0$ by Theorem \ref{thm:contr}; the corollary now  follows from the projectivity of $f_T^+$ proven in Proposition \ref{P:relample+}.
\end{proof}

\begin{remark}\label{R:LMMP+}
If $T=T_{g,n}$ (and $(g,n)\neq (2,0), (1,2)$), then the projectivity of  $\MMM_{g,n}^{T_{g,n}+}$ follows from Remark \ref{R:LMMP} and Proposition \ref{P:relample+}.
Furthermore, if $\car(k)=0$ then it follows from \cite[Thm. 1.1]{AFS3} that $\MMM_{g,n}^{T_{g,n}+}$  is identified with a log canonical model of $\MM_{g,n}$:
\begin{equation}\label{E:LMMP+}
\MMM_{g,n}^{T_{g,n}+}\cong \MM_{g,n}(7/10-\epsilon ):=\Proj \bigoplus_{m\geq 0} H^0(\M_{g,n}, \lfloor m(K_{\M_{g,n}}+\psi+\left(\frac{7}{10}-\epsilon\right)(\delta-\psi))\rfloor),
\end{equation}
extending the previous result of Hassett-Hyeon \cite{HH2} for $n=0$.
\end{remark}

Next, we study the fibres and the exceptional loci of the  morphism $f_T^+$.

\begin{proposition}\label{P:fibers+}
Assume that $(g,n)\neq (2,0), (1,2)$, and $\car(k)\gg (g,n)$.
\noindent 
\begin{enumerate}[(i)]
\item \label{P:fibers+a} The  morphism $f_T^+$ is a contraction, i.e.   $(f_T^+)_*(\cO_{\MMM_{g,n}^{T+}})=\cO_{\MMM_{g,n}^T}$.
\item \label{P:fibers+b} The exceptional locus of $f_T^+$ is the union of the following irreducible closed subsets
$$\Tac([\tau,I], [\tau+1, I]):=\ov{\{(C,\{p_i\})\in \MMM_{g,n}^{T+}\: : \: (C,\{p_i\}) \: \text{has a tacnode of type } \{[\tau,I], [\tau+1, I]\}\}} $$
for every  $\{ [\tau,I], [\tau+1, I]\}\subseteq  T-\{[1,\emptyset]\}$ which is not of the form $\{[0,\{i\}],[1,\{i\}]\}$ for some $1\leq i \leq n$, and 
$$\Tac(\irr):=\ov{\{(C,\{p_i\})\in \MMM_{g,n}^{T+}\: : \: (C,\{p_i\}) \: \text{has a tacnode of type }\{\irr\}\}} \quad \text{ if } \irr \in T \text{and }g\geq 2.$$
All the above closed subsets have codimension three, so that the morphism $f_T^+$ is small.
\end{enumerate}
\end{proposition}

\begin{proof}
Part \eqref{P:fibers+a} follows from the Zariski main theorem using that $f_T^+$ is a proper morphism between irreducible normal algebraic spaces (see Theorem \ref{T:goodspaces}) which is moreover birational since it is an isomorphism when restricted to the dense open subset of smooth curves.

Part \eqref{P:fibers+b}: first of all, the closed subsets in the statement are irreducible and they have codimension three since the semiuniversal deformation space of a tacnode has dimension three (since $\car(k)\neq 2$).
By  Proposition \ref{prop:Tclosed},  the morphism $f_T^+$ sends a $T^+$-closed curve $(C, \{p_i\})$ into the $T$-closed curve $f_T^+((C, \{p_i\}))$ which is the stabilisation of the $n$-pointed curve which is obtained from $(C, \{p_i\})$ by replacing each tacnode (necessarily of type contained in $T-\{[1,\emptyset]\}$ since $(C, \{p_i\})$ cannot have $A_3$-attached elliptic tails) by an attached rosary  of length two. Now observe that a tacnode has local moduli isomorphic to $\Gm$ because it is constructed from the normalisation by gluing together the two tangent spaces at the two smooth branches, see \cite[Sec. 4.1]{HH1} for details. Since $\omega_C(\sum p_i)$ is ample, these local moduli do not give rise to global moduli if and only if one of the two branches of the tacnode belongs to a rational curve with only one other marked point (which always happen if  the type of the tacnode is equal to  $\{[0,\{i\}],[1,\{i\}]\}$ for some $1\leq i \leq n$), in which case the automorphism group of the $2$-pointed rational curve cancels out the local moduli.
The curve $f_T^+((C, \{p_i\}))$ does not depend on the global moduli given by the tacnodes of $(C,\{p_i\})$.
By putting everything together, we deduce that the exceptional locus of $f_T^+$ is equal to the union of the closed subsets described in the statement. 
\end{proof}

As a corollary of the above proposition, we can determine when $f_T^+$ is an isomorphism.

\begin{cor}\label{C:isof+}
Assume that $(g,n)\neq (2,0), (1,2)$, and $\car(k)\gg (g,n)$. Then $f_T^+:\MMM_{g,n}^{T+}\to \MMM_{g,n}^T$ is an isomorphism if and only if $T^{\adm}=T^{\div}$. 
\end{cor} 
\begin{proof}
Proposition \ref{P:fibers+}\eqref{P:fibers+a} implies that the exceptional locus of $f_T^+$ is empty, i.e. $f_T^+$ is an isomorphism, if and only $T^{\adm}=T^{\div}$.  
\end{proof}

The final ingredient we need is a description of the relative Mori cone of the morphism $f_T^+$. With this in mind, we introduce  the following curves, which were already considered in 
 \cite[Propositions 4.1 and 4.2]{HH2}.

\begin{definition}\label{D:1strata+}[Tacnodal curves]
Let $(g,n)\neq (2,0), (1,2)$ be an hyperbolic pair. Consider the following irreducible curves (well-defined up to numerical equivalence) in $\M_{g,n}^{T+}$, which we call \emph{tacnodal curves}:
\begin{enumerate}
\item If $\irr\in T$ and $g\geq 2$ then let $D(\irr)^o\cong \Gm$ to be the curve in $\M_{g,n}^{T+}$ which parametrises $T^+$-semistable curves that are obtained from a fixed  smooth irreducible curve $E$ of genus $g-2$ with $n+2$ marked points
by gluing  the last two marked points, which we call $a$ and $b$, to form a tacnode of type $\irr$ using the identification of $T_aE$ and $T_bE$ provided by the elements of $\Gm$.
We denote by $D(\irr)$ the closure of $D(\irr)^o$ in $\M_{g,n}^{T+}$. The curve $D(\irr)$ is isomorphic to $\mathbb{P}^1$; the two points on the closure parametrise the two curves formed by gluing $a$ and $b$ with a $\mathbb{P}^1$ which is attached nodally at $a$ and tacnodally at $b$ (or the other way around).

\item For any pair $\{[\tau, I], [\tau+1, I]\}=\{[\tau, I], [g-1-\tau, I^c]\}\subset T-\{[1,\emptyset], \bigcup_j [1,\{j\}], \irr\}$, we let $D([\tau,I],[\tau+1,I])^o\cong \Gm$ to be the curve in $\M_{g,n}^{T+}$ which parametrises $T^+$-semistable curves 
that are obtained from two fixed irreducible curves $A$ and $B$, the first of genus  $\tau$ with $I\cup\{a\}$ marked points and the second one of genus $g-1-\tau$ with $I^c\cup \{b\}$ marked points, by gluing the points $a$ and $b$ to form a tacnode of type $\{[\tau, I], [\tau+1, I]\}$, using the identification of $T_aA$ and $T_bB$ provided by the elements of $\Gm$.
We denote by $D([\tau,I],[\tau+1,I])$ the closure of $D([\tau,I],[\tau+1,I])^o$ in $\M_{g,n}^{T+}$. The curve $D([\tau,I],[\tau+1,I])$ is isomorphic to $\mathbb{P}^1$; the two points on the closure parametrise the two curves formed by gluing $a$ and $b$ with a $\mathbb{P}^1$ which is attached nodally at $a$ and tacnodally at $b$ (or the other way around).

\end{enumerate}
The \emph{type}  of a tacnodal curve is defined as follows: $D(\irr)$ has type $\{\irr\}\subset T_{g,n}$ while $D([\tau, I], [\tau+1, I])$ has type equal to $\{[\tau, I], [\tau+1, I]\}\subset T_{g,n}$.
It is straightforward to see that the tacnodal curves parametrises $T^+$-closed points of $\M_{g,n}^{T+}$ (see Definition \ref{def:T+closed}); hence they descend to integral curves (which we will continue to call tacnodal curves and we will denote them with the same notation)   in the good moduli space  $\MMM_{g,n}^{T+}$ by Proposition \ref{P:T+closed}\eqref{P:T+closed1}.

\end{definition}

\begin{remark}\label{R:tac-special}
Notice that we have not defined the tacnodal curves $D([0, \{i\}], [1, \{i\}])$  and $D([1, \{i\}], [2, \{i\}])$ for $1\leq i \leq n$. This is due to the following reasons:
\begin{itemize}
\item If we define   $D([0, \{i\}], [1, \{i\}])^o$ as in the above definition, then $D([0, \{i\}], [1, \{i\}])^o$ is a point and not a curve inside $\M_{g,n}^{T+}$, since the continuous automorphism group of the curve $A$ of genus and with $2$ marked points kills the gluing data that is needed to construct the tacnode. 
\item The curve $D([1, \{i\}], [2, \{i\}])$, defined as the closure of the curve $D([1, \{i\}], [2, \{i\}])^o$ defined as above, is contracted when mapped into $\MMM_{g,n}^{T+}$ via the morphism $\phi^{T+}$ since its generic point is not $T^+$-closed (because it contains an $A_1/A_3$-attached elliptic bridge of type $\{[1, \{i\}], [2, \{i\}]\}\subseteq T$, see Proposition \ref{P:T+closed}\eqref{P:T+closed0}).
\end{itemize}

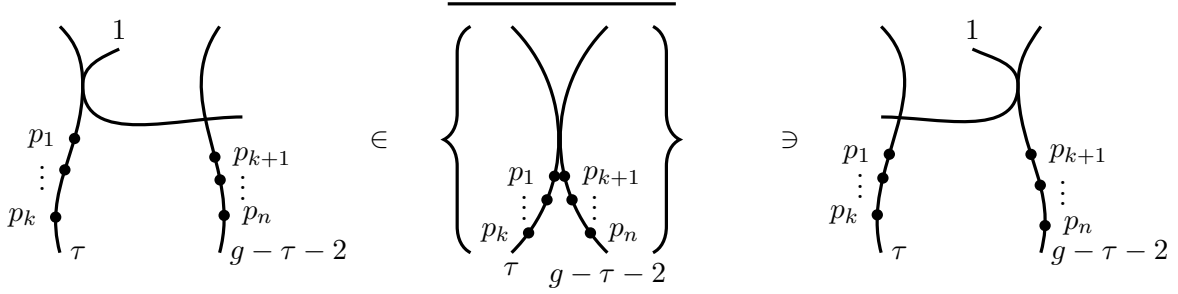
\begin{figure}[!h]
	\begin{center}
		\begin{tikzpicture}[scale=0.6]
		
		\coordinate (x) at (-2-8,2);
		\coordinate (y) at (-1.5-8, 0.7);
		\coordinate[label=right:{$\tau$}] (z) at (-2-8, -3);
		\draw [very thick] (x) to[in=90, out=-45] (y) to[in=105, out=-90]  
		node[dot, pos=0.3, label=left:$p_1$]{}
		node[dot, pos=0.5, label=left:$\vdots$]{}
		node[dot, pos=0.8, label=left:$p_k$]{}
		(z);

		\coordinate[label=above:1] (a) at (-0.7-8, 1.5);
		\coordinate (c) at (2-8, 0);
		\draw [very thick] (a)  to[out=205, in=90]   (y)  to[out=-90, in=180] (c);

		\coordinate (w) at (1.5-8,2);
		\coordinate[label=right:{$g-\tau-2$}] (z) at (1.5-8, -3);
		\draw [very thick, in=75, out=-125] (w) to 
		node[dot, pos=0.6, label=right:$p_{k+1}$]{}
		node[dot, pos=0.7, label=right:$\vdots$]{}
		node[dot, pos=0.85, label=right:$p_n$]{}
		(z);

		\draw[very thick] (-1.5,2.5) to (3.5,2.5);
		
		\draw[very thick, decoration={brace, amplitude=10pt}, decorate] (-1,-3) -- (-1, 2);
		
		\node at (-3,-0.5) {$\in$};		
		
		\coordinate (x) at (-0.1,2);
		\coordinate[label=below:{$\tau$}] (y) at (-0.1, -3);
		\draw [very thick, in=45, out=-45] (x) to 
		node[dot, pos=0.65, label=left:$p_1$]{}
		node[dot, pos=0.75, label=left:$\vdots$]{}
		node[dot, pos=0.9, label=left:$p_k$]{}
		(y);

		\coordinate (w) at (2,2);
		\coordinate[label=below:{$g-\tau-2$}] (z) at (2, -3);
		\draw [very thick, in=135, out=-135] (w) to 
		node[dot, pos=0.65, label=right:$p_{k+1}$]{}
		node[dot, pos=0.75, label=right:$\vdots$]{}
		node[dot, pos=0.9, label=right:$p_n$]{}
		(z);
		
		\draw[very thick, decoration={brace, mirror, amplitude=10pt}, decorate] (3,-3) -- (3, 2);

		\node at (6,-0.5) {$\ni$};

		\coordinate (x) at (-0.5+8.5,2);
		\coordinate[label=right:{$\tau$}] (y) at (-0.5+8.5, -3);
		\draw [very thick, in=105, out=-45] (x) to 
		node[dot, pos=0.6, label=left:$p_1$]{}
		node[dot, pos=0.7, label=left:$\vdots$]{}
		node[dot, pos=0.85, label=left:$p_k$]{}
		(y);

		\coordinate (a) at (-1+9, 0);
		\coordinate (b) at (2+9, 0.7);
		\coordinate[label=above:1] (c) at (1+9, 1.5);
		\draw [very thick] (a)   to[out=0, in=-90]  (b)  to[out=90, in=-25] (c);
		
		\coordinate (w) at (11.5,2);
		\coordinate[label=right:{$g-\tau-2$}] (z) at (11.5, -3);
		\draw [very thick] (w) to[in=90, out=-125] (b)  to[out=-90, in=75] 
		node[dot, pos=0.4, label=right:$p_{k+1}$]{}
		node[dot, pos=0.6, label=right:$\vdots$]{}
		node[dot, pos=0.85, label=right:$p_n$]{}
		(z);

		\end{tikzpicture}		
		
	\end{center}
	\caption{The tacnodal curve $D([\tau, I], [\tau+1, I])$ with the two limit points, where $I=\{1,\ldots, k\}$.} 
	\label{F:tacnodalcurve}
\end{figure}
\end{remark}

\begin{proposition}\label{P:tacinter}
Assume that $(g,n)\neq (2,0), (1,2)$ and that  $\car(k)\gg (g,n)$.
\begin{enumerate}[(i)]
\item \label{P:tacinter1}
The relative Mori cone of the morphism $f_T^+$ is the subcone of $\NEb(\MMM_{g,n}^{T+})$ spanned by the tacnodal curves of type contained in $T$.
\item \label{P:tacinter2}
Given a $\bbQ$-line bundle  
$$
L=a\lambda+b_{\irr}\delta_{\irr}+\sum_{[i,I]\in T_{g,n}^*-\{[1,\emptyset], \bigcup_j [1,\{j\}]\}} b_{i,I} \delta_{i,I}
$$
on $\M_{g,n}^{T+}$, we have the following intersection formulas
$$\begin{sis}
& D([\tau, I], [\tau+1, I])\cdot L=-a-12 b_{\irr}+b_{\tau,I}+b_{\tau+1,I},\\
& D(\irr) \cdot L= -a-10b_{\irr}   .
\end{sis}$$

\end{enumerate}
\end{proposition}
\begin{proof}
Part \eqref{P:tacinter1}: let $D$ be an integral curve inside $\MMM_{g,n}^{T+}$ that is contracted by the morphism $f_T^+$. By Proposition \ref{prop:Tclosed}\eqref{prop:Tclosed0}, the geometric generic point of  $D$ parametrises a $T^+$-closed curve $C$ (by Proposition \ref{P:T+closed}\eqref{P:T+closed1}) with a tacnode $t$ of type contained in $T$ and having some non trivial global gluing data, which happens if and only if $\type(t)$ is not equal to $\{[0,\{i\}], [1,\{i\}]\}$ for some $1\leq i \leq n$. Moreover, since $C$ is  $T^+$-closed curve, $\type(t)$ cannot be equal to $\{[1,\emptyset], [2,\emptyset]\}$ (for otherwise $C$ would contain an $A_3$-attached elliptic tail) or to $\{[1,\{i\}], [2,\{i\}]\}$ for some $1\leq i \leq n$ (for otherwise $C$ would contain an $A_1$/$A_3$-attached elliptic bridge of type contained in $T$).  From this discussion, it follows that $D$ is numerically equivalent to a tacnodal curve of type contained in $T$ and part \eqref{P:tacinter1} follows. 

Part \eqref{P:tacinter2}: let $D\cong \P^1\subset \M_{g,n}^{T+}$ be a tacnodal curve and let $\pi:\cX\to D$ be the associated (flat and projective) family of $n$-pointed $T^+$-semistable curves of genus $g$. The family $\cX\to D$ has a tacnodal section $\tau$ (which is also the only singularity of the each fibre over $\Gm\subset \P^1$) and two nodes over $0$ and $\infty$ that are of type $[\tau,I]$ and $[\tau+1,I]$ if $D=D([\tau,I],[\tau+1,I])$, or both of type $\{\irr\}$ if $D=D(\irr)$. This implies that the only boundary divisor that contains $D$ is $\delta_{\irr}$ and that for any $[i, J]\in T_{g,n}- \{\irr\}$, we have that 
\begin{equation}\label{E:inter-iJ}
\begin{sis}
& \delta_{i,J}\cdot D(\irr)=0, \\
& \delta_{i,J}\cdot D([\tau,I],[\tau+1,I])=
\begin{sis}
1 & \text{ if } [i,J]= [\tau,I] \: \text{ or } [\tau+1,I], \\
0 & \text{ otherwise.} 
\end{sis}
\end{sis} 
\end{equation} 
Consider now the normalisation $\wt{\pi}:\cY\to D$ of the family $\cX\to D$ along the tacnodal section $\tau$. The (flat and projective) family $\cY\to D$ has $n+2$ section, the first $n$ of which are the pull-back of the $n$ sections of the family $\cX\to D$, and the last two sections, call them $\sigma_a$ and $\sigma_b$, are the inverse image of the tacnodal section $\tau$ along the normalisation morphism $\cY\to \cX$. 
We can apply \cite[Prop.6.1]{AFS3} in order to get that:
\begin{equation}\label{E:forAFS}
\begin{sis}
& \lambda\cdot D=\deg_D(\lambda_{\cY/D})-\frac{\deg_D(\psi_a+\psi_b)}{2}, \\
& \delta\cdot D=\deg_D(\delta_{\cY/D})-6\deg_D(\psi_a+\psi_b),
\end{sis}
\end{equation}
where $\delta_{\cY/D}$ is the total boundary of the family $\wt{\pi}:\cY\to D$, $\lambda_{\cY/D}:=\det \wt{\pi}_*(\omega_{\cY/D})$ and $\psi_a=\sigma_a^*(\omega_{\cY/D})$ and $\psi_b=\sigma_b^*(\omega_{\cY/D})$. By the definition of the tacnodal curve $D$, it follows that the family $\cY\to D\cong \P^1$ together with the two sections $\sigma_a$ and $\sigma_b$ is obtained from a constant family $ F\times \P^1\to \P^1$ (where, using the notations of Definition \ref{D:1strata+},  $F=E$ if $D=D(\irr)$ or $F=A\coprod B$ if $D=D([\tau, I], [\tau+1, I])$) together with two constant sections $\{a\}\times \P^1$ and $\{b\}\times \P^1$ by blowing up the points $\{a\}\times \{0\}$ and $\{b\}\times \{\infty\}$ and taking the strict transform of the two constant sections. Therefore, the family $\wt{\pi}:\cY\to D$ has  two singular fibres, namely $\wt{\pi}^{-1}(0)$ and $\wt{\pi}^{-1}(\infty)$ which are  formed by $F$ and the exceptional divisors $E_0$ and $E_{\infty}$, respectively, meeting in one node; hence 
we have that 
\begin{equation}\label{E:deltaY}
\deg_D(\delta_{\cY/D})=2.
\end{equation}
Moreover, since there is no variation of moduli in the fibres of the family $\wt{\pi}:\cY\to D$, we have that 
\begin{equation}\label{E:lambdaY}
\deg_D(\lambda_{\cY/D})=0.
\end{equation}
Finally, since $\sigma_a^*(\omega_{\cY/D})=\sigma_a^*(\cO_{\cY}(-\Im(\sigma_a)))$, we have that $\deg_D(\psi_a)=-(\Im \sigma_a)^2$. Since the pull-back of the constant section
$\{a\}\times \P^1$ to the blow-up family $\wt{\pi}:\cY\to D$ is equal to $E_0+\Im \sigma_a$, we get that 
\begin{equation}\label{E:psia}
0=(E_0+\Im \sigma_a)^2=E_0^2+2E_0\cdot \Im \sigma_a+(\Im \sigma_a)^2=-1+2+(\Im \sigma_a)^2\Rightarrow \deg_D(\psi_a)=-(\Im \sigma_a)^2=1.
\end{equation}
And similarly we have that 
\begin{equation}\label{E:psib}
\deg_D(\psi_b)=1.
\end{equation}
Substituting \eqref{E:deltaY}, \eqref{E:lambdaY}, \eqref{E:psia} and \eqref{E:psib} into \eqref{E:forAFS}, we get that 
\begin{equation}\label{E:forAFSbis}
 \lambda\cdot D=-1 \: \text{ and } \:  \delta\cdot D=-10.
 \end{equation}
By combining \eqref{E:inter-iJ} and \eqref{E:forAFSbis}, we conclude the proof of part \eqref{P:tacinter2}. 
\end{proof}

We are now ready, by combining the above propositions, to give a proof of  Theorem \ref{T:flip+}.

\begin{proof}[Proof of Theorem \ref{T:flip+}]

Note that the algebraic space $\MMM_{g,n}^{T+}$ is normal by Theorem \ref{T:goodspaces} and the morphism $f_T^+$ is a small contraction  by Proposition \ref{P:fibers+}. 
Hence the first two conditions of Definition \ref{D:Dflip} are always satisfied.  Moreover, in order for $f_T^+$ to be the $L$-flip of $f_T$, we need that $L$ is $f_T$-antiample (see Definition \ref{D:Dflip}).

It remains  to check the last condition of Definition \ref{D:Dflip}  with respect to the rational morphism 
$$\eta:=(f_T^{+})^{-1}\circ f_T:\MMM_{g,n}^T\dashrightarrow \MMM_{g,n}^{T+}$$ 
and any  $\bbQ$-Cartier $\bbQ$-divisor $D$ on $\MMM_{g,n}^T$ whose associated $\bbQ$-line bundle is $L$.
If the restriction of $L$  to $\M_{g,n}^{T+}$ (which we denote again by $L$) is $T^+$-compatible, it will descend
 to a $\bbQ$-line bundle $L^{T+}$ on $\MMM_{g,n}^{T+}$ by Proposition \ref{P:lbdesc}. By the commutativity of the diagram \eqref{E:diag-spaces}, we have that the linear equivalence class of the $\bbQ$-divisor $\eta_*(D)$ is $L^{T+}$, which implies  that $\eta_*(D)$ is $\bbQ$-Cartier. Conversely, if $\eta_*(D)$ is $\bbQ$-Cartier then 
 its linear equivalence class is a $\bbQ$-line bundle on $\MMM_{g,n}^{T+}$ whose pull-back to $\M_{g,n}^{T+}$ is the restriction of $L$ to  $\M_{g,n}^{T+}$, and this implies that $L$ is $T^+$-compatible again by Proposition \ref{P:lbdesc}.

Hence it remains to show that if $L$ is $f_T$-antiample then $L^{T+}$ is   $f_{T}^+$-ample. 
Since $f_T^+$ is projective by Proposition \ref{P:relample+} and the relative Mori cone of $f_T^+$ is generated by  the tacnodal curves of type contained in $T$ by Proposition \ref{P:tacinter}\eqref{P:tacinter1}, it is enough  to show, by the relative Kleiman ampleness criterion (\cite[Thm. 1.44]{KM}),  that $L$ intersects negatively these curves. By combining Proposition  \ref{P:tacinter}\eqref{P:tacinter2} with Lemma \ref{L:int-1str} and using that the intersection of $L$ with all the elliptic bridge curves of type contained in $T$ is negative because $L$ is $f_T$-antiample, we get that 
$$
\begin{sis}
& D(\irr)\cdot L=-C(\irr)\cdot L>0 \: \text{ if } \irr \in T, \\
& D([\tau,I],[\tau+1,I])\cdot L=-C([\tau,I],[\tau+1,I])\cdot L>0,  
\end{sis}
$$
 for any   $\{[\tau, I], [\tau+1, I]\}\subset T-\{[1,\emptyset], \bigcup_j [1,\{j\}]\}$, and this concludes the proof.
\end{proof}

We now describe two important special cases of the main Theorem \ref{T:flip+}. 

\begin{cor}\label{C:K-flip}
Assume that $(g,n)\neq (2,0), (1,2)$, and $\car(k)\gg (g,n)$. 
\begin{enumerate}[(i)]
\item \label{C:K-flip1} The  morphism $f_T^+: \MMM_{g,n}^{T,+}  \to \MMM_{g,n}^T$ is the $(K_{\M_{g,n}^{\ps}}+\psi)$-flip of $f_T$. 
\item \label{C:K-flip2} The  morphism $f_T^+: \MMM_{g,n}^{T,+}  \to \MMM_{g,n}^T$ is the $K_{\MM_{g,n}^{\ps}}$-flip of $f_T$ if and only if $\MMM_{g,n}^{T,+}$ is $\bbQ$-Gorenstein, i.e. if and only if $T$ does not contain subsets of the form $\{[0,\{j\}],[1,\{j\}],[2,\{j\}]\}$ for some $j\in [n]$ or  $(g,n)=(3,1), (3,2), (2,2)$.
\end{enumerate}
\end{cor}
\begin{proof}
Since the relative Mori cone of $f_T$ is generated by the elliptic bridge curves of type contained in $T$ by Proposition \ref{P:fibers1}\eqref{P:fibers1bis} and the elliptic bridge curves are both $K_{\MM_{g,n}^{\ps}}$ and $(K_{\M_{g,n}^{\ps}}+\psi)$-negative by Proposition \ref{P:face}\eqref{P:face1}, the relative Kleiman's ampleness criterion (which can be applied since $f_T^+$ is projective by Proposition \ref{P:relample+}) implies that  
$K_{\MM_{g,n}^{\ps}}$ and $(K_{\M_{g,n}^{\ps}}+\psi)$ are $f_T$-antiample.
By Mumford's formula (see Fact \ref{F:PicU}\eqref{F:PicU2}), we have that $K_{\M_{g,n}^{\ps}}+\psi=13\lambda -2\wh{\delta}$ and the restriction of  $13\lambda -2\wh{\delta}$ to $\M_{g,n}^{T+}$ is $T^+$-compatible, see Definition \ref{D:T+comp}. Hence we conclude that $f_T^+$ is the $(K_{\M_{g,n}^{\ps}}+\psi)$-flip of $f_T$ by Theorem \ref{T:flip+}. 

In order to prove part \eqref{C:K-flip2}, observe first that  
\begin{equation}\label{E:pushK}
((f_T^+)^{-1}\circ f_T)_*(K_{\MMM_{g,n}^{\ps}})=K_{\MMM_{g,n}^{T+}}.
\end{equation}
 Therefore, if $f_T^+$  is the $K_{\MMM_{g,n}^{\ps}}$-flip of $f_T$, then $K_{\MMM_{g,n}^{T+}}$ is $\bbQ$-Cartier, i.e. $\MMM_{g,n}^{T+}$ is $\bbQ$-Gorenstein, which happens if and only if $T$ does not contain subsets of the form $\{[0,\{j\}],[1,\{j\}],[2,\{j\}]\}$ for some $j\in [n]$ or  $(g,n)=(3,1), (3,2), (2,2)$ by Corollary \ref{C:Gor-fact}\eqref{C:Gor-fact2}.
Conversely, if $K_{\MMM_{g,n}^{T+}}$ is $\bbQ$-Cartier then, by the diagram \eqref{E:diag-spaces}, we deduce that the restriction of the $\bbQ$-line bundle $K_{\MMM_{g,n}^{\ps}}$ (seen as a $\bbQ$-line bundle on $\M_{g,n}^T$ by Corollary \ref{C:Pic-ps} and Proposition \ref{P:Pic-Mgps}\eqref{P:Pic-Mgps2b}) to $\M_{g,n}^{T+}$ descends to the $\bbQ$-line bundle $K_{\MMM_{g,n}^{T+}}$, and hence it is $T^+$-compatible. Hence, we conclude that $f_T^+$ is the $K_{\MM_{g,n}^{\ps}}$-flip of $f_T$ by Theorem \ref{T:flip+}. 
\end{proof}

Theorem \ref{T:flip+} implies that, when $\MMM_{g,n}^{T+}$ is $\bbQ$-factorial (cf. Corollary \ref{C:Gor-fact}\eqref{C:Gor-fact1}), then the morphism $f_T^+$ is the $L$-flip of $f_T$ with respect to any $\bbQ$-line bundle $L$ on $\MM_{g,n}^{\ps}$ which is $f_T$-antiample. Under these assumptions and assuming furthermore that $f_T$ is small  (cf. Proposition \ref{P:geomT}\eqref{P:geomT2}), we will now prove that $f_T^+$ is the composition of elementary $L$-flips.

\begin{cor}\label{C:compflips}
	Assume $(g,n) \neq (2,0), (1,2)$ and ${\rm char}(k)=0$. Let $T \subset T_{g,n}$ such that $f_T: \MM_{g,n}^{\ps} \to \MMM_{g,n}^T$ is small and $\MMM_{g,n}^{T,+}$ is $\bbQ$-factorial (cf. Proposition \ref{P:geomT}\eqref{P:geomT2} and Corollary \ref{C:Gor-fact}(\ref{C:Gor-fact1})).  	
	Let $L$ be a $\bbQ$-line bundle on $\MM_{g,n}^{\ps}$ which is $f_T$-antiample. 
	
	Then the rational map $(f_T^+)^{-1}\circ f_T:\MM_{g,n}^{\ps} \dashrightarrow \MMM_{g,n}^{T,+}$ can be decomposed (up to isomorphism) as a sequence of elementary $L$-flips. 
\end{cor}

\begin{proof}
	The morphism $f_T:\MM_{g,n}^{\ps} \to \MMM_{g,n}^T$ is a relative Mori dream space because it is $K_{\MM_{g,n}^{\ps}}$-negative (by Theorem \ref{thm:contr}) and $\MM_{g,n}^{\ps}$ is klt and $\bbQ$-factorial (by Proposition \ref{P:Pic-Mgps}\eqref{P:Pic-Mgps}) with a discrete Picard group (by Corollary \ref{C:Pic-ps} and Proposition \ref{P:Pic-Mgps}\eqref{P:Pic-Mgps2}). Hence, we can run an MMP for $L$ over $\MMM_{g,n}^{T}$ and obtain a relative minimal model
	\begin{equation}\label{E:defflip2}
	\xymatrix{ & \MM_{g,n}^{\ps} \ar @{-->}[rr]^{\eta}  \ar[rd]_{f_T} && X \ar[ld]^{g} \\  & & \MMM_{g,n}^{T}}
	\end{equation}
	Since $f_T$ is small, $g$ is also small and $\eta$ is a composition of flips.  Moreover, since $\MMM_{g,n}^{T,+}$ is the ample model of $L$ over  $\MMM_{g,n}^{T}$ there is a birational morphism $X \to \MMM_{g,n}^{T,+}$ over $\MMM_{g,n}^T$, which is again small. Since both spaces are $\bbQ$-factorial we conclude that the morphism $X\to  \MMM_{g,n}^{T,+}$  
	is an isomorphism, as wanted.  
\end{proof}

\bibliographystyle{amsalpha}
\bibliography{Library}

\providecommand{\bysame}{\leavevmode\hbox to3em{\hrulefill}\thinspace}
\providecommand{\MR}{\relax\ifhmode\unskip\space\fi MR }
\providecommand{\MRhref}[2]{%
  \href{http://www.ams.org/mathscinet-getitem?mr=#1}{#2}
}
\providecommand{\href}[2]{#2}
\begin{thebibliography}{AFSvdW17}

\bibitem[ACG11]{GAC2}
Enrico Arbarello, Maurizio Cornalba, and Pillip~A. Griffiths, \emph{Geometry of
  algebraic curves. {V}olume {II}}, Grundlehren der Mathematischen
  Wissenschaften [Fundamental Principles of Mathematical Sciences], vol. 268,
  Springer, Heidelberg, 2011, With a contribution by Joseph Daniel Harris.
  \MR{2807457}

\bibitem[AFS16]{AFS0}
Jarod Alper, Maksym Fedorchuk, and David~Ishii Smyth, \emph{Singularities with
  {$\Bbb{G}_m$}-action and the log minimal model program for
  {$\overline{\mathbb{M}}_g$}}, J. Reine Angew. Math. \textbf{721} (2016),
  1--41. \MR{3574876}

\bibitem[AFS17a]{AFS3}
\bysame, \emph{Second {F}lip in the {H}assett--{K}eel {P}rogram:
  {P}rojectivity}, Int. Math. Res. Not. IMRN (2017), no.~24, 7375----7419.
  \MR{3802125}

\bibitem[AFS17b]{AFS2}
\bysame, \emph{Second flip in the {H}assett-{K}eel program: existence of good
  moduli spaces}, Compos. Math. \textbf{153} (2017), no.~8, 1584--1609.
  \MR{3649808}

\bibitem[AFSvdW17]{AFSV1}
Jarod Alper, Maksym Fedorchuk, David~Ishii Smyth, and Frederick van~der Wyck,
  \emph{Second flip in the {H}assett-{K}eel program: a local description},
  Compos. Math. \textbf{153} (2017), no.~8, 1547--1583. \MR{3705268}

\bibitem[AH12]{AH}
Jarod Alper and Donghoon Hyeon, \emph{G{IT} constructions of log canonical
  models of {$\overline{M}_g$}}, Compact moduli spaces and vector bundles,
  Contemp. Math., vol. 564, Amer. Math. Soc., Providence, RI, 2012,
  pp.~87--106. \MR{2895185}

\bibitem[AHLH18]{AHLH}
Jarod Alper, Daniel Halpern-Leistner, and Jochen Heinloth, \emph{Existence of
  moduli spaces for algebraic stacks}, Preprint arXiv:1812.01128 (2018).

\bibitem[AHR20]{AHR}
Jarod Alper, Jack Hall, and David Rydh, \emph{A {L}una \'{e}tale slice theorem
  for algebraic stacks}, Ann. of Math. (2) \textbf{191} (2020), no.~3,
  675--738. \MR{4088350}

\bibitem[AK19]{AK}
Florin Ambro and J\'{a}nos Koll\'{a}r, \emph{Minimal models of
  semi-log-canonical pairs}, Moduli of {K}-stable varieties, Springer INdAM
  Ser., vol.~31, Springer, Cham, 2019, pp.~1--13. \MR{3967370}

\bibitem[Alp10]{Alp10}
Jarod Alper, \emph{On the local quotient structure of {A}rtin stacks}, J. Pure
  Appl. Algebra \textbf{214} (2010), no.~9, 1576--1591. \MR{2593684}

\bibitem[Alp13]{Alper}
\bysame, \emph{Good moduli spaces for {A}rtin stacks}, Ann. Inst. Fourier
  (Grenoble) \textbf{63} (2013), no.~6, 2349--2402. \MR{3237451}

\bibitem[Alp14]{Alper2}
\bysame, \emph{Adequate moduli spaces and geometrically reductive group
  schemes}, Algebr. Geom. \textbf{1} (2014), no.~4, 489--531. \MR{3272912}

\bibitem[AOV08]{AOV}
Dan Abramovich, Martin Olsson, and Angelo Vistoli, \emph{Tame stacks in
  positive characteristic}, Ann. Inst. Fourier (Grenoble) \textbf{58} (2008),
  no.~4, 1057--1091. \MR{2427954}

\bibitem[ASVdW]{ASV}
Jarod Alper, , David~Ishii Smyth, and Fredrerick Van~der Wyck, \emph{Weakly
  proper moduli stacks of curves.}, Preprint arXiv:1012.0538.

\bibitem[AV02]{AV}
Dan Abramovich and Angelo Vistoli, \emph{Compactifying the space of stable
  maps}, J. Amer. Math. Soc. \textbf{15} (2002), no.~1, 27--75. \MR{1862797}

\bibitem[BC13]{BC}
Aaron Bertram and Izzet Coskun, \emph{The birational geometry of the {H}ilbert
  scheme of points on surfaces}, Birational geometry, rational curves, and
  arithmetic, Simons Symp., Springer, Cham, 2013, pp.~15--55. \MR{3114922}

\bibitem[BCHM10]{BCHM10}
C.~Birkar, P.~Cascini, C.~Hacon, and J.~M\textsuperscript{c}Kernan,
  \emph{Existence of minimal models for varieties of log general type}, J.
  Amer. Math. Soc. \textbf{23} (2010), no.~2, 405--468.

\bibitem[BM14]{BM}
Arend Bayer and Emanuele Macr{\`i}, \emph{Mmp for moduli of sheaves on k3s via
  wall-crossing: nef and movable cones, lagrangian fibrations}, Inventiones
  mathematicae \textbf{198} (2014), no.~3, 505--590.

\bibitem[CCF19]{Giulio}
Cinzia Casagrande, Giulio Codogni, and Andrea Fanelli, \emph{The blow-up of
  $\mathbb{P}^4$ at 8 points and its {F}ano model, via vector bundles on a
  degree 1 del {P}ezzo surface}, Revista Matem\'{a}tica Complutense \textbf{32}
  (2019), no.~2, 475--529.

\bibitem[CH15]{CH}
Izzet Coskun and Jack Huizenga, \emph{The birational geometry of the moduli
  spaces of sheaves on {$\Bbb P^2$}}, Proceedings of the {G}\"okova
  {G}eometry-{T}opology {C}onference 2014, G\"okova Geometry/Topology
  Conference (GGT), G\"okova, 2015, pp.~114--155. \MR{3381441}

\bibitem[CH18]{CH2}
\bysame, \emph{The nef cone of the moduli space of sheaves and strong
  {B}ogomolov inequalities}, Israel J. Math. \textbf{226} (2018), no.~1,
  205--236. \MR{3819692}

\bibitem[CHL06]{CHL}
Hung-Jen Chiang-Hsieh and Joseph Lipman, \emph{A numerical criterion for
  simultaneous normalization}, Duke Math. J. \textbf{133} (2006), no.~2,
  347--390. \MR{2225697}

\bibitem[CTV19]{CTV}
Giulio Codogni, Luca Tasin, and Filippo Viviani, \emph{On some modular
  contractions of the moduli space of stable pointed curves}, To appear on
  Algebra and Number Theory. Preprint arXiv:1904.13212 (2019).

\bibitem[Deb01]{Deb}
Olivier Debarre, \emph{Higher-dimensional algebraic geometry}, Universitext,
  Springer-Verlag, New York, 2001. \MR{1841091}

\bibitem[DPT80]{Tei}
Michel Demazure, Henry~Charles Pinkham, and Bernard Teissier (eds.),
  \emph{S\'eminaire sur les {S}ingularit\'es des {S}urfaces}, Lecture Notes in
  Mathematics, vol. 777, Springer, Berlin, 1980, Held at the Centre de
  Math\'ematiques de l'\'Ecole Polytechnique, Palaiseau, 1976--1977.
  \MR{579026}

\bibitem[FS13]{FS}
Maksym Fedorchuk and David~Ishii Smyth, \emph{Alternate compactifications of
  moduli spaces of curves}, Handbook of moduli. {V}ol. {I}, Adv. Lect. Math.
  (ALM), vol.~24, Int. Press, Somerville, MA, 2013, pp.~331--413. \MR{3184168}

\bibitem[Fuj99]{Fujino99}
Osamu Fujino, \emph{Applications of {K}awamata's positivity theorem}, Proc.
  Japan Acad. Ser. A Math. Sci. \textbf{75} (1999), no.~6, 75--79. \MR{1712648}

\bibitem[FV20]{FV}
Roberto Fringuelli and Filippo Viviani, \emph{On the {P}icard group scheme of
  the moduli stack of stable pointed curves}, Preprint arXiv:2005.06920 (2020).

\bibitem[GKM02]{GKM}
A.~Gibney, S.~Keel, and I.~Morrison, \emph{Towards the ample cone of
  {$\overline M_{g,n}$}}, J. Amer. Math. Soc. \textbf{15} (2002), no.~2,
  273--294.

\bibitem[Gro67]{EGAIV4}
A.~Grothendieck, \emph{\'el\'ements de g\'eom\'etrie alg\'ebrique. {IV}.
  \'etude locale des sch\'emas et des morphismes de sch\'emas {IV}}, Inst.
  Hautes \'Etudes Sci. Publ. Math. (1967), no.~32, 361. \MR{0238860}

\bibitem[Har10]{Har}
Robin Hartshorne, \emph{Deformation theory}, Graduate Texts in Mathematics,
  vol. 257, Springer, New York, 2010. \MR{2583634}

\bibitem[Has05]{Has}
Brendan Hassett, \emph{Classical and minimal models of the moduli space of
  curves of genus two}, Geometric methods in algebra and number theory, Progr.
  Math., vol. 235, Birkh\"auser Boston, Boston, MA, 2005, pp.~169--192.
  \MR{2166084}

\bibitem[HH09]{HH1}
Brendan Hassett and Donghoon Hyeon, \emph{Log canonical models for the moduli
  space of curves: the first divisorial contraction}, Trans. Amer. Math. Soc.
  \textbf{361} (2009), no.~8, 4471--4489. \MR{2500894}

\bibitem[HH13]{HH2}
\bysame, \emph{Log minimal model program for the moduli space of stable curves:
  the first flip}, Ann. of Math. (2) \textbf{177} (2013), no.~3, 911--968.
  \MR{3034291}

\bibitem[HL07]{HL}
Donghoon Hyeon and Yongnam Lee, \emph{Stability of tri-canonical curves of
  genus two}, Math. Ann. \textbf{337} (2007), no.~2, 479--488. \MR{2262795}

\bibitem[HM10]{HM}
Donghoon Hyeon and Ian Morrison, \emph{Stability of tails and 4-canonical
  models}, Math. Res. Lett. \textbf{17} (2010), no.~4, 721--729. \MR{2661175}

\bibitem[Hui17]{Hui}
Jack Huizenga, \emph{Birational geometry of moduli spaces of sheaves and
  {B}ridgeland stability}, Surveys on recent developments in algebraic
  geometry, Proc. Sympos. Pure Math., vol.~95, Amer. Math. Soc., Providence,
  RI, 2017, pp.~101--148. \MR{3727498}

\bibitem[Kem78]{Kempf}
George~R. Kempf, \emph{Instability in invariant theory}, Ann. of Math. (2)
  \textbf{108} (1978), no.~2, 299--316. \MR{506989}

\bibitem[KM97]{KeM}
Se\'an Keel and Shigefumi Mori, \emph{Quotients by groupoids}, Ann. of Math.
  (2) \textbf{145} (1997), no.~1, 193--213. \MR{1432041}

\bibitem[KM98]{KM}
J\'anos Koll\'ar and Shigefumi Mori, \emph{Birational geometry of algebraic
  varieties}, Cambridge Tracts in Mathematics, vol. 134, Cambridge University
  Press, Cambridge, 1998, With the collaboration of C. H. Clemens and A. Corti,
  Translated from the 1998 Japanese original. \MR{1658959}

\bibitem[Kol90]{KolProj}
J\'anos Koll\'ar, \emph{Projectivity of complete moduli}, J. Differential Geom.
  \textbf{32} (1990), no.~1, 235--268. \MR{1064874}

\bibitem[Kol10]{Exercises}
\bysame, \emph{Exercises in the birational geometry of algebraic varieties},
  Analytic and algebraic geometry, IAS/Park City Math. Ser., vol.~17, Amer.
  Math. Soc., Providence, RI, 2010, pp.~495--524. \MR{2743822}

\bibitem[LZ19]{LZ}
Chunyi Li and Xiaolei Zhao, \emph{Birational models of moduli spaces of
  coherent sheaves on the projective plane}, Geom. Topol. \textbf{23} (2019),
  no.~1, 347--426. \MR{3921322}

\bibitem[Mat02]{Matsukibook}
Kenji Matsuki, \emph{Introduction to the {M}ori program}, Universitext,
  Springer-Verlag, New York, 2002. \MR{1875410}

\bibitem[Mor82]{Mori82}
S.~Mori, \emph{Threefolds whose canonical bundles are not numerically
  effective}, Ann. of Math. (2) \textbf{116} (1982), no.~1, 133--176.

\bibitem[Mor01]{Mor}
Atsushi Moriwaki, \emph{The {$\Bbb Q$}-{P}icard group of the moduli space of
  curves in positive characteristic}, Internat. J. Math. \textbf{12} (2001),
  no.~5, 519--534. \MR{1843864}

\bibitem[MS17]{Mac}
Emanuele Macr\`\i and Benjamin Schmidt, \emph{Lectures on {B}ridgeland
  stability}, Moduli of curves, Lect. Notes Unione Mat. Ital., vol.~21,
  Springer, Cham, 2017, pp.~139--211. \MR{3729077}

\bibitem[Nue16]{Nue}
Howard Nuer, \emph{Projectivity and birational geometry of {B}ridgeland moduli
  spaces on an {E}nriques surface}, Proc. Lond. Math. Soc. (3) \textbf{113}
  (2016), no.~3, 345--386. \MR{3551850}

\bibitem[Sch91]{Sch}
David Schubert, \emph{A new compactification of the moduli space of curves},
  Compositio Math. \textbf{78} (1991), no.~3, 297--313. \MR{1106299}

\bibitem[Ser06a]{sernesi}
Edoardo Sernesi, \emph{Deformations of algebraic schemes}, Grundlehren der
  Mathematischen Wissenschaften [Fundamental Principles of Mathematical
  Sciences], vol. 334, Springer-Verlag, Berlin, 2006. \MR{2247603}

\bibitem[Ser06b]{Ser}
\bysame, \emph{Deformations of algebraic schemes}, Grundlehren der
  Mathematischen Wissenschaften [Fundamental Principles of Mathematical
  Sciences], vol. 334, Springer-Verlag, Berlin, 2006. \MR{2247603}

\bibitem[Smy11]{Smyth_non_reduced}
David~Ishii Smyth, \emph{Modular compactifications of the space of pointed
  elliptic curves {I}}, Compos. Math. \textbf{147} (2011), no.~3, 877--913.
  \MR{2801404}

\bibitem[{Sta}18]{stacks-project}
The {Stacks Project Authors}, \emph{\textit{Stacks Project}},
  \url{http://stacks.math.columbia.edu}, 2018.

\bibitem[vOV07]{vOV}
Michael~A. van Opstall and R{\u a}zvan Veliche, \emph{Maximally symmetric
  stable curves}, Michigan Math. J. \textbf{55} (2007), no.~3, 513--534.
  \MR{2372614}

\bibitem[Yos16]{Yos}
K\=ota Yoshioka, \emph{Bridgeland's stability and the positive cone of the
  moduli spaces of stable objects on an abelian surface}, Development of moduli
  theory---{K}yoto 2013, Adv. Stud. Pure Math., vol.~69, Math. Soc. Japan,
  [Tokyo], 2016, pp.~473--537. \MR{3616985}

\end{thebibliography}

\end{document}